\documentclass[12pt]{amsart}
\usepackage{fullpage,amssymb,amsfonts,amsmath}
\usepackage{enumerate}
\usepackage[shortlabels]{enumitem}
\usepackage{comment}
\usepackage{hyperref} 
\usepackage[capitalise]{cleveref}
\usepackage{url}
\usepackage{todonotes}
\usepackage{mathrsfs}
\usepackage{tikz}
\usetikzlibrary{shapes}
\usetikzlibrary{plotmarks}
\usetikzlibrary{decorations.markings}

\newtheorem{theorem}{Theorem}
\newtheorem{hypothesis}{Hypothesis} 

\newtheorem{thmnum}{Theorem}[section]
\newtheorem{propnum}[thmnum]{Proposition}
\newtheorem{lemnum}[thmnum]{Lemma}
\newtheorem{cornum}[thmnum]{Corollary}

\newtheorem{thmA}{Theorem}

\newtheorem{exA}[thmA]{Example}

\newtheorem{thmB}{Theorem}

\newtheorem{corB}[thmB]{Corollary}

\newtheorem{exB}[thmB]{Example}

\newtheorem{thmK}{Theorem}

\newtheorem{propK}[thmK]{Proposition}

\theoremstyle{definition}

\newtheorem{remark}{Remark}
\counterwithin{remark}{subsection}

\numberwithin{equation}{section}

\renewcommand{\epsilon}{\varepsilon}

\newcommand{\R}{\mathbb{R}}

\newcommand{\Z}{\mathbb{Z}}

\newcommand{\Res}{\mathop{\mathrm{Res}}}

\newcommand{\EQNlabel}[2]{
\begin{equation} 
#1
\begin{split}
	#2
\end{split}
\end{equation}
}

\newcommand{\EQN}[1]{
\begin{align*}
	#1
\end{align*}
}
 
 \renewcommand{\phi}{\varphi}
 
 \definecolor{green}{rgb}{.4,.7,.4}
\definecolor{grey}{RGB}{211,211,211}
 
\newcommand{\ep}{\varepsilon}

\newcommand{\con}{\equiv}

\newcommand{\modd}[1]{\; ( \mathrm{mod} \; #1)}

\newcommand{\maps}{\rightarrow}

\newcommand{\Fscr}{\mathscr{F}}

\newcommand{\supp}{{\rm supp \;}}

\newcommand{\al}{\alpha}
\newcommand{\be}{\beta}

\newcommand{\ga}{\gamma}
\newcommand{\del}{\delta}

\newcommand{\om}{\omega}
\newcommand{\Om}{\Omega}

\newcommand{\sig}{\sigma}
\newcommand{\lam}{\lambda}

\newcommand{\Ga}{\Gamma}

\newcommand{\Mcal}{\mathcal{M}}

\newcommand{\C}{\mathbb{C}}

\newcommand{\N}{\mathbb{N}}

\newcommand{\Q}{\mathbb{Q}}

\newcommand{\beq}{\begin{equation}}
\newcommand{\eeq}{\end{equation}}

\title{A guide to Tauberian theorems \\for arithmetic applications}

\author{Lillian  B. Pierce}
\author{Caroline L. Turnage-Butterbaugh}
\author{Asif Zaman}
 \date{\today}

\begin{document}

\begin{abstract}
A Tauberian theorem deduces an asymptotic for the partial sums of a sequence of non-negative real numbers from analytic properties of an associated Dirichlet series. Tauberian theorems appear in a tremendous variety of applications, ranging from well-known classical applications in analytic number theory, to new applications in arithmetic statistics, group theory, and  the intersection of number theory and algebraic geometry. The goal of this article is to provide a useful reference for practitioners who wish to apply a Tauberian theorem. We explain the hypotheses and proofs of two types of Tauberian theorems: one with and one without an explicit remainder term. We furthermore provide counterexamples that illuminate that neither theorem can reach an essentially stronger conclusion unless its hypothesis is strengthened.

\end{abstract}

\maketitle

\author

\begin{center}
   \emph{Dedicated to John B. Friedlander and Henryk Iwaniec \\in honor of their long-standing commitment to exposition.}
\end{center}

\tableofcontents
\section{Introduction}

Let $1 \leq \lambda_1 < \lambda_2 < \lambda_3 < \cdots$ be an   increasing sequence of  real numbers   tending to infinity,  and let $\{a_n\}_{n \geq 1}$ be any  sequence of non-negative real  numbers. A \emph{Tauberian theorem} provides an asymptotic estimate for the partial sum
\EQNlabel{\label{eqn:original-sum}}{
\sum_{\lambda_n \leq x} a_n
}
as $x \maps \infty$ by using analytic information for its associated general Dirichlet series, given as a function of a complex variable $s$  by
\EQNlabel{\label{def:DirichletSeries}}{
A(s) := \sum_{n=1}^{\infty} \frac{a_n}{\lambda_n^s}, 
	}
	well-defined for $s$ in an appropriate right half-plane, and possibly meromorphically continued in a larger region. Most commonly, the sequence $\{\lam_n\}_{n \geq 1}$ is simply $\{ n\}_{n \geq 1}$.  
	
Tauberian theorems have tremendously wide applications in arithmetic settings. Within number theory, classical examples  include the application of a Tauberian theorem to prove the Prime Number Theorem (see \cite[\S 7.3]{BatDia04}), or to 
count the number of integers $n \leq x$ represented as a sum of two squares (see \cite[Thm. 10.5]{BatDia04}). 
Recently, Tauberian theorems have also become an essential tool for deep problems of central interest in algebraic geometry, group theory, and arithmetic statistics. Let us sketch how Tauberian theorems play a role in each of these settings, and then we will turn to describing the main types of Tauberian theorems we explain in this article.

\subsection{Manin's conjecture} First, Manin's conjecture   is a broad area of  current interest at the intersection of number theory and algebraic geometry. It formulates a precise conjecture for the asymptotic behavior of the number of rational points of bounded height (in a Zariski open set, with respect to an anticanonical height function) on an algebraic variety; the asymptotic is considered as the height goes to infinity.
The predicted asymptotic depends on geometric invariants of the variety; the conjecture has been formulated for smooth Fano varieties \cite{BatMan90} \cite{FMT89},  classes of singular Fano varieties \cite{BatTsc98a} or almost-Fano varieties \cite{Pey95}. Verifying Manin's conjecture in specific cases is a topic of great current interest, and there has been notable progress, for example for  toric varieties \cite{BatTsc98b}, certain del Pezzo surfaces  \cite{Bro09,Bro10}, and certain settings with an ergodic flavor  \cite{STT07,GMO08,GorOh11}. But there are many further open questions.

In this setting, the number of rational  points of height at most $x$ is encoded by a   partial sum  $\sum_{n \leq x} a_n$ analogous to (\ref{eqn:original-sum}), and proving the predicted asymptotic for such partial sums depends on verifying certain analytic properties of a height zeta function that plays the role of the Dirichlet series $A(s)$ in (\ref{def:DirichletSeries}). For example, in order to obtain the asymptotic predicted by Manin's conjecture (via a Tauberian theorem), a key question is whether the height zeta function can be meromorphically continued in a certain region. The clear presence of a ``Tauberian theorem'' in the setting of Manin's conjecture is already seen in the early formulation  \cite{FMT89}, and a ``standard Tauberian argument'' is frequently required in papers on  Manin's conjecture (see  \cite{CLT01,CLT02,CLT10,CLT12} as just a few examples). On the other hand, in some cases of recent progress toward Manin's conjecture, the strength of the work lies in going beyond an off-the-shelf application of a Tauberian theorem (see for example \cite{BBS14,BetDes19}, both of which give clear summaries of some state-of-the-art considerations).

\subsection{Subgroup growth}\label{sec_zeta_group}

A second area in which Tauberian theorems have become crucial tools lies in group theory, and in particular the study of the subgroup growth problem, and related questions. Let $G$ be a finitely generated group, such as the free abelian group $\Z^m$ of rank $m$, or the discrete Heisenberg group (a finitely generated nilpotent group, given by $3\times3$ upper triangular matrices with integral entries, and 1's on the diagonal). Then $G$ has only finitely many subgroups of each index. Let $a_n(G)$ denote the number of subgroups of  $G$ of index precisely $n$, and let $S_x(G) = \sum_{n \leq x}a_n(G)$ denote the number of subgroups of  $G$ of index at most $x$. The subgroup growth problem asks: how fast does $S_x(G)$ grow as $x \maps \infty$? A group $G$ is said to have polynomial subgroup growth if  $S_x(G)$ has at most polynomial growth as $x \maps \infty$.

To study this question, Grunewald, Segal, and Smith \cite{GSS88} introduced  a zeta function associated to the group $G$, defined by
\beq\label{zeta_group_dfn}
\zeta_G(s) = \sum_{n=1}^\infty a_n(G) n^{-s} = \sum_{H \leq G} [G:H]^{-s}.
\eeq
This can be thought of as a non-commutative generalization of the Dedekind zeta function of a number field. There is a basic relationship between holomorphicity of $\zeta_G(s)$ and the subgroup growth problem:   $\zeta_G(s)$ is holomorphic in some right-half plane if and only if $S_x(G)$ has at most polynomial growth (see Remark  \ref{remark_zeta_group_convergence}).

To study the question of polynomial subgroup growth among finitely generated groups, it is no loss of generality to restrict attention to  so-called residually finite groups (see for example  \cite[\S 1.1]{Vol11}). In a celebrated result, 
Lubotzky, Mann, and Segal \cite{LMS93}  proved a characterization: among finitely generated residually finite groups, those with polynomial subgroup growth  are precisely  those which  have a subgroup of finite index that is soluble of finite rank. But this still leaves open the quantitative question of  how fast the polynomial growth is, for a given group $G$. Define the abscissa of convergence $\al_G$ of $\zeta_G(s)$ to be 
\beq\label{zeta_group_abscissa_dfn}
\al_G:= \inf \{ \al : \text{$\exists\, c>0 $ such that $S_x(G) \leq c(1+x^\al)$ for all $x \geq 1$}\}.
\eeq
(Thus $\al_G>0$ if  $G$ is an infinite   group with polynomial subgroup growth, while $\al_G = -\infty$ if $G$ is finite.)
It is a central problem  to prove that for suitable infinite finitely generated groups $G$, the rate of polynomial subgroup growth obeys the asymptotic
\beq\label{group_asymp} 
S_x(G) \sim c_G x^{\al_G} (\log x)^{b_G}, \qquad \text{as $x \maps \infty$,}
\eeq
for $\al_G$ the abscissa of convergence and for some real numbers $c_G>0$ and   $b_G \geq 0$. If $G$ is nilpotent, this has been proved in a landmark result of du Sautoy and Grunewald \cite[Thm. 1.1]{SauGru00}: they showed that $\zeta_G(s)$ is holomorphic for $\Re(s)\geq \al_G$ except for a pole at $s=\al_G$ of order $m$ for some (unspecified) integer $m \geq 1$, and from this fact, a Tauberian theorem produced (\ref{group_asymp}), with $b_G=m-1$.  

This work  opened the door to studying the subgroup growth question in more general contexts: for example,    the number of subrings of bounded index in a ring \cite[\S 1.2]{Vol11}, or    the number of bounded-dimensional irreducible complex representations of a rigid group  \cite[\S 3.2]{Vol11}, and further questions for zeta functions of groups and rings  \cite[\S 4]{Vol11}. It can be very difficult to explicitly determine the abscissa of convergence of such a zeta function, and to determine whether it has a meromorphic continuation strictly to the left of the abscissa of convergence (see e.g. \cite{SauWoo08}), but  the final ingredient is an appropriate Tauberian theorem in order to ultimately obtain a result such as (\ref{group_asymp}).

\subsection{Counting number fields} As a third area rich with potential applications for Tauberian theorems, we highlight current themes in arithmetic statistics, and in particular 
  a central and difficult type of question related to counting number fields in a family.  
 For example, a very general conjecture, attributed to Linnik, predicts that if $\mathscr{F}^d(X)$ represents the set of degree $d$ extensions $K/\Q$ with bounded absolute discriminant $|\mathrm{Disc}(K/\Q)| \leq X,$ then 
 \[ |\mathscr{F}^d(X)| \sim c_d X  \qquad \text{as $X \maps \infty$}\]
 for a particular constant $c_d>0$. 
 This is true for $d=2$ by a classical computation (related to counting square-free integers up to $X$  \cite[Appendix]{EPW17}; for $d=3$ by Davenport and Heilbronn \cite{DavHei71}; for $d=4$ by Cohen, Diaz y Diaz and Olivier \cite{CDO02} and Bhargava \cite{Bha05}; for $d=5$ by Bhargava \cite{Bha10a}. 
 It is not known for $d \geq 6$. Furthermore, there are many more  conjectures of a closely-related flavor. For example, let $\mathscr{F}^d(G;X)$ denote the set of degree $d$ extensions $K/\Q$ with bounded absolute discriminant $|\mathrm{Disc}(K/\Q)| \leq X$ and specified Galois group $G$, for a given transitive subgroup $G \subseteq S_d$. That is to say, with all $K$ in a  fixed algebraic closure $\overline{\Q}$,  if $\tilde{K}$ denotes a Galois closure of $K$ over $\Q$,
  the Galois group of $\tilde{K}/\Q $
  (regarded as a permutation group on the $n$ embeddings of $K$ in $\overline{\Q}$)   is isomorphic to $G$  via an isomorphism of permutation groups.
 In this setting, the well-known Malle heuristics \cite{Mal02} predict the order of growth   of $|\mathscr{F}^d(G;X)|$ as $X \maps \infty$, or at least upper and lower bounds (up to factors of $\log X$, as noticed in \cite{Mal04,Klu05}). One can furthermore order a family of fields by other invariants than the discriminant, or  impose finitely many or infinitely many local conditions on the fields in the family, and then ask for the order of growth of the family as $X \maps \infty$. This is a vast area of current research; see for example an overview in \cite{Woo16}.

Suppose now that  $\mathscr{F}(X)$ represents a particular family of degree $d$ number fields $K$ with $|\mathrm{Disc}(K/\Q)| \leq X$. Consider, for example, a standard Dirichlet series defined by  $A(s)=\sum_n a_n n^{-s}$,  in which each coefficient  $a_n \geq 0$  represents the number of fields in the family with precisely $|\mathrm{Disc}(K/\Q)|=n$. Then, the question of computing an asymptotic for $|\mathscr{F}(X)|$ as $X \maps \infty$ corresponds   to the question of computing an asymptotic for a partial sum $\sum_{n \leq X} a_n$, as in (\ref{eqn:original-sum}).  The philosophy of a    ``Tauberian theorem'' suggests  this should be, in turn, a question of verifying certain analytic properties of the Dirichlet series $A(s)$. Verifying such properties is a difficult type of open problem in general, but also a promising avenue: the Malle-Bhargava principle (currently a set of heuristics or questions)  suggests a technique to construct such Dirichlet series, and (potentially) interpret them as a kind of Euler product, which may facilitate verifying their analytic properties. (See Malle \cite{Mal04} and Bhargava \cite{Bha07} as well as \cite[\S 10]{Woo16}.)

Which analytic properties are required? This is the core motivation for the present work. Many aspects of the Malle-Bhargava principle remain unknown; it can be difficult to justify the construction of the associated series $A(s)$, and then difficult to verify its analytic properties. It is desirable to have a Tauberian theorem at hand which assumes only a ``minimal'' hypothesis about $A(s)$. 

At the same time, there are many questions that require quantifying a partial sum $\sum_{\lam_n \leq x} a_n$ more precisely than a ``minimal'' Tauberian theorem can accomplish. In particular, proving an asymptotic for the partial sum (\ref{eqn:original-sum}),  with an explicit remainder term, requires a Tauberian theorem with additional hypotheses---that is to say, more quantitative, analytic information about the Dirichlet series $A(s)$. Proving a power-saving remainder term in an asymptotic is often interesting in its own right---or might reveal a second-order main term (see for example the celebrated works \cite{BBP10,TanTho13}). But a power-saving remainder term can also have ramifications beyond the original problem. 

Here is one such example:
 it is conjectured that for each degree $d$ and each positive integer $D$, for any $\ep>0$ there are at most $\ll_{d,\ep} D^\ep$  number fields of degree $d$ with precisely $|\mathrm{Disc} (K/\Q)| =D$.  
This has been named the Discriminant Multiplicity Conjecture in \cite{PTBW20}, and it is closely bound in a web of deep problems involving class groups and the Cohen-Lenstra-Martinet heuristics, and families of number fields and the Malle heuristics; see the survey \cite{Pie23}.
The Discriminant Multiplicity Conjecture relates to Tauberian theorems: suppose the partial sum $\sum_{n \leq X} a_n $ encodes the count for $|\mathscr{F}^d(X)|$ (so that $a_n$ counts number fields of degree $d$ with discriminant precisely $n$). Then the Discriminant Multiplicity Conjecture predicts that 
 $a_D \ll_{d,\ep} D^\ep$ for every $D \geq 1$ (for all $\ep>0$). 
 
  Information about one term  in a partial sum $\sum_{n \leq X} a_n$  can be extracted from an asymptotic with an explicit remainder term (regardless of whether the coefficients $a_n$ are non-negative). Precisely, suppose   it is known that for  all $X \geq 1$, 
\[ \sum_{n \leq X} a_n  = X^\al + O(X^{\be}),\]
for some $\al \geq 1$ and some $\al-1 \leq \be< \al$. Then by applying this with $X=D$ and $X=D-1$, we can deduce
\beq\label{applications_pointwise}
a_D = O(D^\be).
\eeq
For example,  the best result known toward the Discriminant Multiplicity Conjecture for $d=5$ deduces that there are at most $O(D^{199/200+\ep})$ quintic number fields with $|\mathrm{Disc}(K/\Q)|=D$  from an asymptotic count $|\mathscr{F}^5(X)| = c_5X + O(X^{199/200+\ep})$; see \cite[Thm. 5.1]{EPW17}.
This illustrates the utility of a Tauberian theorem with power-saving remainder term.

\subsection{Types of Tauberian theorem} While the broad outline of a ``Tauberian theorem'' may be familiar,  when faced with a specific application, questions often arise about exactly which hypotheses about the Dirichlet series $A(s)=\sum_{n} a_n \lam_n^{-s}$ are required to arrive at a specific conclusion for the partial sum $\sum_{\lam_n \leq x} a_n$. In fact, Tauberian theorems have been studied extensively for a century, as the titles of the encyclopedic surveys by Korevaar indicate \cite{Kor02,Kor04}. Yet the  vastness of the literature, as well as its heavily analytic perspective,  can make it difficult to recognize the result one seeks for an arithmetic, algebraic, or geometric application.  
 This article focuses on giving a streamlined, accessible reference for two theorems of a Tauberian style that have proven useful in many recent arithmetic applications, and are likely to be used frequently as the fields mentioned above continue to develop.  
 
 We outline two different hypotheses on Dirichlet series, and give an accessible proof for each corresponding Tauberian theorem: version A, with a minimal hypothesis and no explicit remainder term, and version B, with a quantitative hypothesis and an explicit remainder  term.   We also record why caution is needed: in each case, we demonstrate why the hypothesis leads to a conclusion that cannot (essentially) be improved without a stronger hypothesis.
In particular,  we clarify the relationship between  an explicit remainder term in a Tauberian theorem, and a quantitative growth hypothesis for $A(s)$. On this topic, an error was present in a commonly-cited source (Remark \ref{remark_CLT}); clarifying this situation is one of our main motivations.
 Much of what we describe is not new;  we   indicate any source we follow or deviate from, and briefly describe a few more general results in \S \ref{sec_lit}, with accompanying citations.  
 
 With these Tauberian theorems in hand, practitioners with an arithmetic, algebraic, or geometric application in mind will face a further question: how to verify   analytic properties of the relevant Dirichlet series, to confirm that Hypothesis \ref{hyp:Weak} or Hypothesis \ref{hyp:Strong} holds. Useful strategies for accomplishing this, often by comparison to other well-known Dirichlet series, have recently been cataloged by Alberts \cite{Alb24xb}.

\section{Statement of the Tauberian theorems}

 \subsection{Hypothesis \ref{hyp:Kronecker}: to obtain an upper bound}
We begin with an elementary result on general Dirichlet series with complex coefficients, known as Kronecker's lemma. While we do not consider this a Tauberian theorem per se, because it does not produce an asymptotic result, we present it in a parallel format to our main Tauberian theorems, with a hypothesis, theorem, advantages, disadvantages, and limiting-case counterexample. 

  For clarity, when we refer to a general Dirichlet series without assuming its coefficients are real and non-negative, we will  denote it by $D(s)=\sum_n d_n \lam_n^{-s}$, while we reserve the notation $A(s)=\sum_{n} a_n \lam_n^{-s}$ for the case $a_n \geq 0$.  Thus let $\{d_n\}_n$ be a sequence of complex numbers, and let $\{\lam_n\}_n$ be an increasing sequence of real numbers tending to infinity, with $\lam_1 \geq 1$.  

\setcounter{hypothesis}{10}
\setcounter{theorem}{10}

\begin{hypothesis}[for $s$]\label{hyp:Kronecker}
    The general Dirichlet series $D(s) = \sum_n d_n \lam_n^{-s}$ with complex coefficients $\{d_n\}_n$ converges for a given $s \in \C$ with $\Re(s)>0$. 
\end{hypothesis}

 \begin{theorem}[Kronecker's lemma]\label{thm:Kronecker}
Assume Hypothesis \ref{hyp:Kronecker} for a given $s \in \C$ with $\Re(s)>0$, for  a general Dirichlet series  $D(s) = \sum_{n \geq 1}d_n \lam_n^{-s}$ with complex coefficients $\{d_n\}_n$.  Then  
\[\sum_{\lam_n \leq x}d_n = o(x^{\Re(s)}) \qquad \text{as $x \maps \infty.$} 
\]
\end{theorem}
 The proof, which we provide in \S \ref{sec_Kronecker}, is elementary, based merely on the assumed convergence and partial summation (adapted to general Dirichlet series). 

 \subsubsection{Advantages}
 An advantage of this result is that the hypothesis requires only convergence, and allows complex coefficients.
 To obtain an accurate upper bound on the partial sum of coefficients of a general Dirichlet series, it is advantageous to apply Kronecker's lemma for $s$ with $\Re(s)$ as small as possible.
 In particular, if $\al>0$ is given and $D(s)$ is known to have a pole of order $m \geq 1$ at the real point $s=\al$ but to converge for all real $s >\al$, then Kronecker's lemma implies that for all $\ep>0$, 
 \[\sum_{\lam_n \leq x}d_n = o(x^{\al +\ep}) \qquad \text{as $x \maps \infty.$} \]

 \subsubsection{Disadvantages, desiderata, and reasons for caution}
In the setting of Dirichlet series, the real heart of any sophisticated ``Tauberian theorem'' is that it provides not just a reasonable upper bound, but an asymptotic.  A counterexample shows that Hypothesis \ref{hyp:Kronecker} is not strong enough to guarantee such a conclusion. 
Suppose as above that $D(s)$ is known to have a pole of order $m \geq 1$ at the real point $s=\al$ but to converge for all  real  $s >\al$ (and therefore for all $s$ in the open half-plane $\Re(s)>\al$, see \S \ref{sec_convergence}). 
A counterexample shows that a Tauberian theorem that aims to deduce an asymptotic for the partial sum of the coefficients of $D(s)$  must assume in addition some hypothesis that $D(s)$ is ``well-behaved''  in a larger region. Indeed, the following example  shows that at a minimum, the hypothesis must guarantee that $D(s)$ has no pole  on the vertical line $\Re(s)=\al$ other than at $s=\al$. 

\begin{propK}\label{prop_limit_of_Kronecker}
There exists a standard Dirichlet series $A(s) = \sum_n a_n n^{-s}$ with non-negative real coefficients $a_n \in [0,2]$ that is holomorphic for $\Re(s)>1$, and has a meromorphic continuation to $\C$ with only a simple pole at each of $s=1, s=1+i$ and $s=1-i$, yet the limit
\[\lim_{x \maps \infty} \frac{\sum_{n \leq x} a_n}{x} 
\]
fails to exist. This example is given by the coefficients $a_n = 1+\cos(\log n)$ and asymptotically $\sum_{n \leq x} a_n  = x + \tfrac{1}{\sqrt{2}}x \cos (\log x-\pi/4) + O (\log x)$ as $x \maps \infty$.
\end{propK}
This counterexample, which we explain in \S \ref{sec_PNT_counterexample},  has long been known to illustrate an obstacle to an easy proof of the Prime Number Theorem, as it shows that even when the coefficients are non-negative, Theorem \ref{thm:Kronecker} does not suffice to produce an asymptotic.
  Thus we turn to the derivation of Tauberian theorems that   assume an appropriate hypothesis and produce an asymptotic.

\subsection{Hypothesis \ref{hyp:Weak}: to obtain an asymptotic}

\setcounter{hypothesis}{0}
\setcounter{theorem}{0}

Our first main hypothesis for a Tauberian theorem is as follows, for a given real $\al>0$ and integer $m \geq 1$. 

\begin{hypothesis}[for $\al,m$] \label{hyp:Weak}
 The general Dirichlet series $A(s)=\sum_n a_n \lam_n^{-s}$ with non-negative coefficients $\{a_n\}_n$ converges for $\Re(s)>\al$, and in this region 
 $A(s) = g(s)/(s-\alpha)^m + h(s)$ where $g$ is a polynomial of degree at most $m-1$ with $g(\al) \neq 0$ and $h$ is holomorphic in $\Re(s) > \alpha$ and continuous in $\Re(s) \geq \alpha$.
\end{hypothesis}

In particular, this property holds if $A(s)$ can be analytically continued to the region $\Re(s) \geq \alpha$ aside from a pole of order $m \geq 1$ at the real point $s=\alpha$. In that case, $A(s)$ can be analytically continued to satisfy $A(s) = g(s)/(s-\al)^m$ with $g$ holomorphic in $\Re(s) \geq \al$ and $0 \neq g(\al)  = \lim_{s \maps \al} (s-\al)^mA(s)$.

It is useful to make an initial observation. 
By a result of Landau (for general Dirichlet series with non-negative coefficients), Hypothesis \ref{hyp:Weak} implies  that   $A(s)$ converges (and in fact converges absolutely) for each $s$ in the region $\Re(s)>\al$; see \S \ref{sec_Landau}. 
Thus,   under Hypothesis \ref{hyp:Weak}, Kronecker's lemma (Theorem \ref{thm:Kronecker}) shows that for all $\ep>0$, 
 \beq\label{Kronecker_consequence_Hyp_A} \sum_{\lam_n \leq x} a_n = o(x^{\al+\ep}) \qquad \text{as $x \maps \infty.$}
 \eeq

The first  Tauberian theorem we consider deduces an asymptotic from Hypothesis \ref{hyp:Weak}:
 
\begin{theorem} \label{thm:Weak} Assume \cref{hyp:Weak} for a real $\al>0$ and an integer $m \geq 1$.
Then,   
\[ \sum_{\lam_n \leq x} a_n  \sim c x^\al (\log x)^{m-1} \qquad \text{as $x \to \infty$,}\]
in which 
\[ c=\lim_{s \maps \al} \frac{(s-\al)^m A(s)}{\al(m-1)!}.\]
\end{theorem} 
(By convention, $(m-1)!=1$ if $m=1$.)
We prove this in \S \ref{sec_thm_weak} and \S \ref{sec_Delange}.    If $A(s)$ is holomorphic for $\Re(s) \geq \al$ aside from a pole at $s=\al$, the expression $cx^\al (\log x)^{m-1}$ can be interpreted as the leading order term in an expansion of the residue $\Res_{s=\al}( A(s)x^s/s)$; see Remark \ref{remark_residue}.

 As an immediate corollary, under the hypothesis of the theorem, we have
\beq\label{HypA_BigO}
\sum_{\lam_n \leq x} a_n  = O(x^\al(\log x)^{m-1}).
\eeq
This   improves on Kronecker's lemma (Theorem \ref{thm:Kronecker}); see also \S \ref{sec_remark_upper_bound} if only an upper bound of this strength is needed.

We mention three example applications of Theorem \ref{thm:Weak}, with detailed justification and further examples provided in \S \ref{sec_A_corollaries}.  

\begin{exA}\label{cor_HypA_PNT}
Let $\zeta(s)$ denote the Riemann zeta function, and let 
\[A(s) = -\frac{\zeta'(s)}{\zeta(s)} = \sum_{n=1}^\infty \frac{\Lambda(n)}{n^s},\]
in which the von Mangoldt function $\Lambda(n)$ evaluates to $\log p$ when $n=p^k$ for some $k\in \mathbb{Z}^+$ and $0$ otherwise. Then $\sum_{n \leq x} \Lambda(n) \sim x$ as $x \maps \infty$, so that $\pi(x) = \# \{ p \leq x: \text{$p$ prime}\} \sim x/\log x$ as $x \maps \infty$. 
\end{exA}

The second example is a theorem of du Sautoy and Grunewald \cite[Thm.1.1]{SauGru00}.
\begin{exA}\label{cor_HypA_dSG}
  Let $G$ be an infinite, finitely generated, nilpotent  group. There exists a rational number $\al_G>0$, a real number $c_G>0$, and an integer $b_G \geq 0$ such that 
  \[ \#\{\text{subgroups of $G$ of index at most $x$}\} \sim c_G x^{\al_G}(\log x)^{b_G} \qquad \text{as $x \maps \infty$.}\]
\end{exA}

The third example application is a special case of a result of Kl\"uners \cite[Lemma 2.2]{Klu22}.
\begin{exA}\label{cor_HypA_Klu}
Let $k$ be a number field, and for each prime ideal $\mathfrak{p}$ in the ring of integers $\mathcal{O}_k$, let $\mathrm{Nm}(\mathfrak{p}) =|\mathcal{O}_k/\mathfrak{p}|$ denote the norm (an integer, as the cardinality of an associated finite field). Let $\ell$ be a rational prime, and $L$ the set of prime ideals $\mathfrak{p}$ of $k$ such that $\mathrm{Nm}(\mathfrak{p}) \con 0$ or $1 \modd{\ell}$. Fix  integers  $d,e \geq 1$ and  let $m =e\cdot [k(\zeta_\ell):k],$ in which $\zeta_\ell$ is a primitive $\ell$-th root of unity. Define  a Dirichlet series with non-negative coefficients $a_n$ by
\[A(s) = \sum_{n=1}^\infty \frac{a_n}{n^s} = \prod_{\mathfrak{p} \in L}\left(1 + \frac{m}{\mathrm{Nm}(\mathfrak{p})^{ds}}\right).\]
Then there is a constant $c=c(k,\ell,d,e) >0$ such that as $x \maps \infty$,
\[ \sum_{n \leq x} a_n  \sim  c  x^{1/d} (\log x)^{e-1}.\]
\end{exA}
Finally, in Theorem \ref{thm_hyp_weak_counterex} below, we provide an example of a Dirichlet series for which Hypothesis \ref{hyp:Weak} is satisfied, but a stronger hypothesis fails.

Hypothesis \ref{hyp:Weak} does not assume any property of the series to the left of the abscissa of convergence.   
In \cref{cor_HypA_PNT}, the function $A(s)$ is known to be holomorphic for $\Re(s) \geq 1$ aside from a pole of order $1$ at $s=1$, but there is no  $\del>0$  for which $(s-1)A(s)$ is known (unconditionally) to be holomorphic in the larger half-plane $\Re(s) > 1-\del$. (This is because   $\zeta'(s)/\zeta(s)$ has a pole of order 1 at any nontrivial zero of the Riemann zeta function; of course if the Riemann Hypothesis is assumed, then $(s-1)A(s)$ is holomorphic for $\Re(s)>1/2.$) Similar remarks apply to  \cref{cor_HypA_dSG}, due to the use of Artin $L$-functions to obtain the meromorphic continuation of the relevant zeta function $\zeta_G(s)$  (see \S \ref{sec_A_corollaries_group}). 
 If a Dirichlet series has multiplicative coefficients, such as Example \ref{cor_HypA_Klu}, alternative strategies are also possible (see \S \ref{sec_Wirsing}).
For  a Dirichlet series that is known to have holomorphic extensions to a larger half-plane, there is potential to extract explicit remainder terms, and not just the leading term in an asymptotic,  a topic we will pursue momentarily with the stronger Hypothesis \ref{hyp:Strong}.

\subsubsection{Relation to previous literature}

In previous literature, statements closely related to Theorem \ref{thm:Weak} are typically stated under a hypothesis,  stronger than Hypothesis \ref{hyp:Weak}, that $A(s)$ has non-negative coefficients and can be analytically continued to $\Re(s) \geq \al$ aside from a pole of order $m$ at $s=\al$. When the pole is simple ($m =1$), such a theorem was developed by Ikehara \cite{Ike31}, building on techniques of Wiener in analysis, and also presented in Wiener's monumental survey \cite[\S 10]{Wie32}; this has led to the terminology ``Wiener--Ikehara method.'' When $A(s)$ possesses a higher-order pole $(m \geq 2)$, or even more general types of singularities (for example, non-integral $m$), celebrated work of Delange  proved asymptotics in   great generality; see the original proof in \cite{Del54} and a delightful survey with arithmetic applications in the lecture series \cite{Del55}. Delange referred to such results as generalizations of Ikehara's theorem (in both the title of \cite{Del54} as well as \cite[p. 60]{Del55}), leading to the terminology ``Ikehara-Delange method,'' even though Delange's work is vastly more general than the original Ikehara theorem. For a list of classical papers in this area, see \cite[p. 309-310]{Nar00}.

Under this stronger hypothesis,
statements closely related to Theorem \ref{thm:Weak} are available in textbooks (but only for standard Dirichlet series, while we  consider general Dirichlet series). For example, \cite{BatDia04} presents the Wiener--Ikehara method. In that source, Section 7.2 presents a complete proof for $m=1$, while Section 7.4 considers higher order poles (and indeed other types of singularities), but the method of proof is only hinted at for   poles of order $m \geq 2$. Alternatively, Chapter III \S 3 of Narkiewicz \cite{Nar83} develops the Ikehara--Delange method to obtain a result that is applicable to even further types of singularities (see also \S \ref{sec_lit}); that treatment essentially exposes the celebrated method of Delange \cite{Del54,Del55}. A strength of the theorem recorded in \cite[Thm. 3.8]{Nar83} is its generality, but in the spirit of exposing the conceptual role   of  Hypothesis \ref{hyp:Weak}, we present a complete proof of Theorem \ref{thm:Weak} that retains only the aspects of \cite[Thm. 3.8]{Nar83} that are necessary for our setting. In addition,  the presentation of the proof in \cite[Thm. 3.8]{Nar83}  appears to have  a gap, which we correct (Remark \ref{remark_Narkiewicz}).  

Finally, Theorem \ref{thm:Weak} strengthens the results above, by working with the weaker Hypothesis \ref{hyp:Weak} in order to recover a result attributed to unpublished notes of H. Stark (see Corollary \ref{cor_thm_Delange_Stark} and Remark \ref{remark_Stark}).

 \subsubsection{Advantages}
  Hypothesis \ref{hyp:Weak} does not assume knowledge that $A(s)$ can be meromorphically continued in any region to the left of $\Re(s) \geq \al$, or  any control of the size of $A(s)$ in vertical strips. This can be an advantage if such knowledge is not available. 

\subsubsection{Disadvantages, desiderata, and reasons for caution}\label{sec_HypA_caution}

Note that Hypothesis \ref{hyp:Weak} requires that $A(s)$ has no   pole on the line $\Re(s)=\al$ other than a pole at $s=\al$. There are Dirichlet series with non-negative coefficients and infinitely many poles on a vertical line; see for example $A(s) = \prod_p (1-2p^{-s})^{-1}$, which has a (right-most) pole of order $3$ at $s=1$, as well as infinitely many (right-most) poles on the line $\Re(s)=1$ (all of which are due to the factor for $p=2$). See \cite{Gro56}, or more recently an explanation in \cite[Ex. 5.5]{Alb24xb}. Thus it is not a candidate for Theorem \ref{thm:Weak}. Recall also the example in Proposition \ref{prop_limit_of_Kronecker}.

To say that asymptotically $f(x) \sim g(x)$ as $x \maps \infty$ is to say that $f(x)=g(x) + o(g(x))$ as $x \maps \infty$; we will call such a remainder  term \emph{asymptotically smaller}. It would be stronger to obtain a remainder  term that is $O(g(x)/\log x)$ as $x \maps \infty$, which we will call  \emph{log-saving}. Finally, it would be stronger still to obtain a remainder  term that is $O(g(x)/x^{\del_0})$ as $x \maps \infty$ for a fixed $\del_0>0$, which we will call \emph{power-saving}.

One rule-of-thumb formulation of ``a Tauberian theorem'' is to state that under ``appropriate hypotheses'' on a Dirichlet series $A(s)=\sum_{\lam_n \geq 1} a_n \lam_n^{-s}$ with $a_n \geq 0$ and a pole of order $m$ at the real point $s=\al>0$, asymptotically 
\beq\label{full_poly_asymp} 
\sum_{\lam_n \leq x} a_n = x^\al P_{m-1}(\log x) +o(x^\al), \qquad \text{as $x \maps \infty,$}
\eeq 
   for a   polynomial $P_{m-1}$ of degree $m-1$ with real coefficients. 
 This is strictly stronger than the conclusion of Theorem \ref{thm:Weak}, in which the main term corresponds only to the contribution of the leading order term in $P_{m-1}(\log x)$, and the remainder  term is only asymptotically smaller. Extracting an asymptotic with any lower-order terms in $P_{m-1}$ is not guaranteed, if only   Hypothesis \ref{hyp:Weak} is assumed. To make this clear, we provide an example, for which we provide full details in \S \ref{sec_hyp_weak_counterex} (and further study in Theorem \ref{thm_hyp_weak_counterex_vertical_growth}).

\begin{thmA}
\label{thm_hyp_weak_counterex}
  There is a Dirichlet series $A(s) = \sum_{n=1}^\infty a_n n^{-s}$ with $a_n \geq 0$ that is absolutely convergent for $\Re(s) >1$ and extends holomorphically to $\Re(s) \geq 1$ aside from precisely one pole of order 2 at $s=1$, and such that 
\beq\label{no_log_savings}
\sum_{n \leq x} a_n  = x \log x + x \left( \tfrac{1}{2} \sin (\log ^2 (x)) - 1\right) + O(x/\log x), \qquad \text{as $x \maps \infty$.}
\eeq
In particular, the statement $\sum_{n \leq x} a_n = x P_1(\log x) + o(x)$ for a linear polynomial $P_1(u) = a_1u + a_0$ is false.
This example is defined by the coefficients $a_n = (\log n)(1+\cos (\log^2 n))$.
\end{thmA}

This example is recorded in Tenenbaum \cite[Ch. II \S 7.4 Eqn. (22)]{Ten15} which credits Karamata (c. 1952); see also   \cite[Prob. 7.5]{BatDia04} and \cite{JS21}. 
In particular, this shows that in general Hypothesis \ref{hyp:Weak} is  not sufficient to prove (\ref{full_poly_asymp}), which in this case would have taken the form $x(\log x) + c_0x + o(x)$ for some constant $c_0$.   Further examples of   functions (not necessarily Dirichlet series) that expose the limitations of Hypothesis \ref{hyp:Weak} are also provided in \S \ref{sec_hyp_weak_counterex_further}.

\subsection{Hypothesis \ref{hyp:Strong}: to obtain an asymptotic with power-saving remainder}
In some applications, it is essential to achieve a log-saving or power-saving remainder  term. We now state a stronger hypothesis, in order to obtain such improved asymptotics.
Fix real numbers $\al>0,$ $\delta \in (0,\alpha), \kappa \geq 0,$ and $m \in \Z^+$.

\begin{hypothesis}[for $\al,\del,\kappa,m$] \label{hyp:Strong} 
	 The general Dirichlet series $A(s)=\sum_n a_n \lam_n^{-s}$ with non-negative coefficients $\{a_n\}_n$ can be analytically continued to the region $\Re(s) \geq \alpha-\delta$ aside from a pole of order $m$ at the real point $s=\alpha$.  Moreover, for some $M_1,M_2 \geq 0$, 
     \beq\label{PH_prep_hyp_B} |(s-\al)^mA(s)| \leq M_1 \exp(|s|^{M_2}) \quad \text{for all $s$ in the strip $\al -\del \leq \Re(s) \leq \al$,}
     \eeq
     and there exists   $C  \geq 1$ such that for $\Re(s) = \alpha-\delta$,
			\beq\label{hypStrong_growth_condition}
			|A(s)| \leq C (1+|\Im(s)|)^{\kappa} ( \log(3+|\Im(s)|) )^{m-1}.
			\eeq

\end{hypothesis}
\noindent

\begin{theorem} \label{thm:Strong}
	Assume \cref{hyp:Strong} for $\al,\del,\kappa>0,m$. If $x \geq 2$ then
	\EQN{
	\sum_{\lambda_n \leq x} a_n =  \Res_{s=\alpha} \Big[ A(s) \frac{x^s}{s} \Big] + O_{\al,\del,\kappa,m,C,A(\alpha+\delta)}\Big(   x^{\alpha - \frac{\delta}{\kappa+1}} (\log x)^{m-1} \Big).
	}
\end{theorem}

We prove this in \S \ref{sec_B_prelim} and \S \ref{sec_B_proof}.  While  \cref{thm:Strong} may not traditionally be referred to as a Tauberian theorem, we refer to it this way due to its relation to \cref{thm:Weak}. Ideas behind its proof date back to Landau \cite{Landau1912,Landau1917,Landau1918}, for example, and are described in \S\ref{sec_B_prelim}.

 It is worth remarking on dependence on the parameters: for all work under Hypothesis \ref{hyp:Strong} we allow implicit constants to depend only on  $\al,\del,\kappa,m, C,$ and $A(\alpha+\delta)$, but not otherwise on $A(s)$ or $(a_n)_n$ or $(\lambda_n)_n$. (We remark that the conclusion of the theorem does not depend on the constants $M_1,M_2$ in (\ref{PH_prep_hyp_B}) since this condition is only used in an application of the Phragm\'{e}n-Lindel\"of theorem (Lemma \ref{lemma_PH}), whose conclusion is independent of $M_1,M_2$.)
Hypothesis \ref{hyp:Strong} could have been posed with (\ref{hypStrong_growth_condition}) replaced by the condition
\beq\label{hypStrong_growth_condition_alternate}
|A(s)| \leq C' (1+|\Im(s)|)^{\kappa'}  ,
\eeq
which does not include logarithmic factors.
Certainly, if (\ref{hypStrong_growth_condition}) holds for $\kappa$ then (\ref{hypStrong_growth_condition_alternate}) holds for any $\kappa'>\kappa$. But we have used the formulation (\ref{hypStrong_growth_condition}) in order to facilitate the specific remainder term we obtain (which is useful in formulating Theorems \ref{thm:Example} and \ref{thm:Example_bounded}, below).
Observe the role of the growth parameter $\kappa$ in the remainder term of Theorem \ref{thm:Strong}: as $\kappa \maps 0$, up to logarithmic factors, the remainder term approaches $O(x^{\al - \del})$, which is the remainder term that can be achieved for a smoothed sum (see (\ref{smooth_argument_conclusion_rescaled})). (In Remark \ref{remark_kappa_zero} we verify Theorem \ref{thm:Strong} also for $\kappa=0$, but with a remainder term of $O(x^{\al-\del}(\log x)^{m})$.)

The  residue in the theorem statement may be interpreted as $x^\al P_{m-1}(\log x)$ for $P_{m-1}$ a polynomial of degree $m-1$ with real coefficients, as we compute next. 

 \begin{remark}[Residue calculation]\label{remark_residue}
Suppose $A(s) = g(s)(s-\al)^{-m}$ for a function $g$ that is holomorphic for $\Re(s) \geq \al>0$ with $g(\al) \neq 0$. Then
  \begin{align*}
 \Res_{s =\alpha} \Big[ \frac{A(s) x^{s }}{s} \Big]
 &= \frac{1}{(m-1)!}\lim_{s \maps \al}\left(\frac{d}{ds}\right)^{m-1}  \left[ \frac{g(s)}{s} x^s\right]\\
 & = \frac{1}{(m-1)!}\lim_{s \maps \al}  \sum_{\ell=0}^{m-1} { m-1 \choose \ell} \left(\frac{d}{ds}\right)^{m-\ell-1}\left[\frac{g(s)}{s}\right] \left(\frac{d}{ds}\right)^{\ell}[x^s]  \\
  & = d_{m-1}x^\al(\log x)^{m-1}  + x^\al   \sum_{\ell=0}^{m-2} d_\ell (\log x)^\ell  =: x^{\alpha} P_{m-1}( \log x)
    \end{align*}
    in which  the polynomial $P_{m-1}(y)=d_{m-1}y^{m-1} + \cdots + d_1y + d_0$ is defined by coefficients
     \[ d_{m-1} = \frac{g(\al)}{\al(m-1)!}  = \lim_{s \maps \al} \frac{(s-\al)^m A(s)}{\al(m-1)!}\]
     and 
     \[d_\ell = \frac{1}{\ell! (m-\ell-1)!} \lim_{s \maps \al} \left(\frac{d}{ds}\right)^{m-\ell-1}\left[\frac{g(s)}{s}\right], \qquad 0 \leq \ell \leq m-2.\]
    In the context of Theorem \ref{thm:Weak}, only the first term in this expansion appears in the asymptotic of  the theorem statement. In the context of Theorem \ref{thm:Strong}, the asymptotic takes the form
$ x^\al P_{m-1}(\log x) +O(x^{\al - \del'})$ for some $\del'>0$.

 \end{remark}  

We provide two example applications, which we deduce from Theorem \ref{thm:Strong} in \S \ref{sec_B_corollaries}.  The first is  purely instructive and builds on \cite[Ex. 2.1]{Alb24xb}.
 \begin{exB}\label{cor:Strong_1}
 Define a Dirichlet series with non-negative coefficients by 
  \[ A(s) = \prod_p (1+2p^{-s})  = \sum_{n \geq 1} \frac{a_n}{n^s}.\]
  Then for any $\ep>0$,
\[\sum_{n \leq x} a_n =  x P_1(\log x)+O_{\epsilon}(x^{2/3+\ep})\]
for a  real polynomial $P_1$ of degree $1$ with leading coefficient 
$\prod_p (1-3p^{-2}+2p^{-3})$.
 \end{exB}
This example is directly related to the Dirichlet divisor problem, so the error term can be improved along the same lines   as the hyperbola method giving $O_{\epsilon}(x^{1/2+\epsilon})$; see e.g. \cite[Eqn. (1.75)]{IwaKow04} or \cite[Ch. 8 \S 3]{SSFour}. Our only aim here is to illustrate how to apply \cref{thm:Strong} in a simple situation.

 The second example stems from   \cite[Prop. 2.1]{PTBW20}.
 Let $G$ be a cyclic group of order $d \geq 2$. Let $F_d(\Q,G)$ denote the set of cyclic extensions of $\Q$ of degree $d$ (inside a fixed algebraic closure $\overline{\Q}$) with Galois group $\simeq G$.   Let $F^*_d(\Q,G)$ denote the subset of $F_d(\Q,G)$ comprised of fields $K$ with the property that every rational prime that ramifies tamely in $K$ is totally ramified in $K$.  
 The number of fields in $F_d(\Q,G)$ with absolute discriminant at most $x$ (in absolute value) is asymptotically  $c_d x^{\frac{1}{d-d/\ell}}$ for a constant $c_d>0$, in which $\ell$ is the smallest prime divisor of $d$; see references and related comments in \cite[\S 2.1]{PTBW20}. If $d$ is prime, then $F^*_d(\Q,G)=F_d(\Q,G)$. If  $d$ is composite, a Tauberian theorem shows that $F^*_d(\Q,G)$ has density zero in $F_d(\Q,G)$ (when ordered by discriminant). 
 
  \begin{exB}\label{cor:Strong_2}
   Let $G$ be a cyclic group of order $d \geq 2$.
 Define a Dirichlet series $A(s)$ with non-negative coefficients by setting $A(s) = \sum_{n \geq 1} a_n n^{-s}$ in which $a_1=1$ and for each integer $n \geq 2$, $a_n$ denotes $|\mathrm{Aut}(G)|$ times the number of fields $K \in F^*_d(\Q,G)$ with $|\mathrm{Disc}(K/\Q)| = n.$ 
 Then 
 \[ \sum_{n \leq x} a_n  = c_d' x^{\frac{1}{d-1}} + O(x^{\frac{1}{d-1} - \del' + \ep})\]
 for a real constant $c_d'>0$,   $\del' = \frac{2}{(d-1)(\phi(d)+4)}$, and any $\ep>0$.
 Consequently $F^*_d(\Q,G)$ contains asymptotically $c_d'' x^{1/(d-1)}+o( x^{1/(d-1)-\del'+\ep})$ fields of discriminant $\leq x$ as $x \maps \infty$.
 \end{exB}

\subsubsection{Relation to previous literature}
We provide a complete proof of Theorem \ref{thm:Strong} in \S \ref{sec_B_prelim} and \S \ref{sec_B_proof}. Brumley, Lesesvre and Mili\'{c}evi\'{c} provide a clear proof of a result of similar strength, but only in the case $m=1$, and only in the case of a standard Dirichlet series; see the recent preprint \cite[Appendix B]{BLM21x}. Alberts also provides a complete proof of a similar, useful result in \cite[Thm. 6.1]{Alb24xa}, based on a   theorem of Roux \cite{Rou11}.
We do not follow either method   (although in spirit they are similar). Instead, we provide an explicit argument involving \emph{smoothing and  unsmoothing}, which practitioners may find useful when they need to adapt the ideas for bespoke Tauberian theorems in other contexts.

 One commonly-cited source for a result based on Hypothesis \ref{hyp:Strong} is
Theorem A.1 in \cite[Appendix A]{CLT01}; but there is a gap in the proof provided there, and that theorem as stated is false. (The stated explicit remainder term is too strong to be guaranteed by the assumed hypothesis; see Remark \ref{remark_CLT}.)    Alberts has also recently pointed out in \cite[\S 6.4]{Alb24xa} misconceptions about power-saving remainder terms in other recent literature that has applied Tauberian theorems.  To clarify the important correspondence between the quantitative growth bounds in Hypothesis \ref{hyp:Strong} and power-saving remainder terms, we provide limiting counterexamples in Theorem \ref{thm:Example} and Theorem \ref{thm:Example_bounded} below.  

\subsubsection{Advantages}

One advantage of an explicit remainder term is that it allows the estimation of an individual coefficient $a_n,$ by expressing it as a difference of two partial sums and applying Theorem \ref{thm:Strong} to each. We record this corollary, which we prove in \S \ref{sec_B_proof},  as:
\begin{corB} \label{cor:Pointwise}
  Assume \cref{hyp:Strong} for $\al,\del,\kappa,m$.	 For each integer $n \geq 2$, 
  \[ 0 \leq a_n \ll \lambda_n^{\alpha- \frac{\delta}{\kappa+1}} (\log  \lambda_n)^{m-1}.\] 
\end{corB}
This corollary emphasizes that in applications where bounding $a_n$ pointwise is the goal, it can be desirable to obtain as small a power-saving remainder term as possible. Comparing Theorem \ref{thm:Strong} with Theorem \ref{thm:Example} (below)   suggests the importance of verifying Hypothesis \ref{hyp:Strong} for $\delta$ as large as possible and  $\kappa$ as small as possible.

\subsubsection{Disadvantages, desiderata, and reasons for caution}
While Hypothesis \ref{hyp:Strong} allows a pole of any order, there are applications that require consideration of other singularity types; we provide   brief references   for other singularity types in \S \ref{sec_lit}.
 
The remainder term stated in Theorem \ref{thm:Strong} saves   $\del/(\kappa+1)$ in the exponent of $x^{\al}$. Can this be improved under Hypothesis \ref{hyp:Strong}?
Our next result clarifies that the best possible savings is at most $\del/(\kappa+1/2)$, which suggests \cref{thm:Strong} is nearly optimal for large $\kappa$ values.  
\begin{thmB} \label{thm:Example}
Given $\alpha > 0, \delta \in (0,\alpha)$, $\kappa > \frac{1}{2}$, and $m \in \Z^+,$ there exists an increasing sequence   $\{\lambda_n\}_{n \geq 1}$ of real numbers tending to infinity with $\lam_1 \geq 1$, and a  sequence of non-negative real numbers $\{a_n\}_{n \geq 1}$ defining a general  Dirichlet series $A(s) = \sum_{n=1}^{\infty} a_n \lambda_n^{-s}$ which satisfies all of the following:
	\begin{enumerate}[(i)]
		\item The Dirichlet series $A(s)$ satisfies \cref{hyp:Strong} with   $\alpha, \delta, \kappa,$ and $m$. 

        \medskip 

		\item As $n \to \infty$,  $ a_n = \Omega\Big( \lambda_n^{\alpha-\frac{\delta}{\kappa+1/2}} (\log \lambda_n)^{m-1}  \Big)$.

        \medskip
		
		\item 	As $x \to \infty$, $\displaystyle \sum_{\lambda_n \leq x} a_n = \Res_{s=\alpha} \Big[ A(s) \frac{x^s}{s} \Big]  + \Omega\Big( x^{\alpha - \frac{\delta}{\kappa+1/2}} (\log x)^{m-1} \Big). $

	\end{enumerate}

\end{thmB}

  Here, as is standard in analytic number theory, the notation $f(x)=\Omega(g(x))$ indicates that $\limsup_{x \maps \infty}|f(x)/g(x)|>0.$
As we remark in the construction provided in  \S \ref{sec_example_unbounded} to prove this theorem,  if  $(\kappa+1/2)/\del \in \N$, then $\lam_n$ are natural numbers so that $A(s)$ is a standard Dirichlet series. 

Regarding (ii), this $\Omega$ result can be compared to the previous result $ a_n \ll  \lambda_n^{\alpha-\frac{\delta}{\kappa+1}} (\log \lambda_n)^{m-1}$ of  Corollary \ref{cor:Pointwise}.  For a fixed  value of $\kappa$, comparing Corollary \ref{cor:Pointwise} and Theorem \ref{thm:Example} leaves an open question of whether  there is potentially room to improve Corollary \ref{cor:Pointwise} by a factor of $\lambda_n^{\frac{\delta}{(\kappa+1)(2\kappa+1)}}$ (and likewise room to improve \cref{thm:Strong} by  a factor of $x^{\frac{\delta}{(\kappa+1)(2\kappa+1)}}$). For large values of $\kappa$, this room for improvement is relatively small, so for such applications, the savings provided by \cref{thm:Strong} and \cref{cor:Pointwise} is close to optimal without assuming stronger analytic properties of $A(s)$.

The coefficients $a_n$ in the particular example provided by Theorem \ref{thm:Example} are unbounded. In some situations, one might be studying a Dirichlet series for which one knows a priori that the coefficients are bounded, so it is reasonable to ask: can the remainder term in \cref{thm:Strong} be improved under \cref{hyp:Strong} for Dirichlet series with bounded coefficients? Our next theorem, proved in \S \ref{sec_example_bounded}, gives a negative answer by developing a more elaborate example in which the individual coefficients are bounded.
\begin{thmB} \label{thm:Example_bounded}
Under the hypotheses of Theorem \ref{thm:Example}, the same conclusion holds with (ii) replaced by the statement 
	\begin{enumerate}[(i*)]
	\setcounter{enumi}{1}
	\item \text{For all $n \geq 1$, $0 \leq a_n \leq 1$.}
	\end{enumerate}
\end{thmB}
 While the proof of \cref{thm:Example} constructs a standard Dirichlet series  whenever $(\kappa+1/2)/\delta$ is a positive integer,   the proof of \cref{thm:Example_bounded} does not appear to produce a standard Dirichlet series.

  Finally, we mention an example of a Dirichlet series that does not satisfy Hypothesis \ref{hyp:Strong}; we provide the proof in \S \ref{sec_hyp_weak_counterex}.
\begin{thmB}\label{thm_hyp_weak_counterex_vertical_growth}
The Dirichlet series $A(s)$ constructed in Theorem \ref{thm_hyp_weak_counterex} extends meromorphically to $\Re(s) >0$, and in this region the only singularity is  precisely the pole of order 2 at $s=1$.
 Moreover, for every fixed $\del \in (0,1)$, as $|t| \maps \infty$,
\[ A(1-\del+it) =  \tfrac{\sqrt{\pi}}{4} t (\ep_+(\del,t)e^{+ \tfrac{1}{2}\del t}+  \ep_-(\del,t)e^{- \tfrac{1}{2}\del t}) (1+O_\del(1/|t|)) +O(1+|t|),\]
for a continuous function $\ep_\pm(\del,t)$ such that $|\ep_\pm(\del,t)|=1$ for all $\del,t$.
Thus $A(s)$ does not satisfy a polynomial growth condition on $\Re(s)=1-\del$ (of the type required by Hypothesis \ref{hyp:Strong}), for any $\del \in (0,1).$
\end{thmB}
This illustrates that a Dirichlet series $A(s) =\sum a_n n^{-s}$ can admit a meromorphic continuation far to the left of precisely one real pole of order $m$, and yet still not yield an asymptotic with power-saving remainder term for the partial sums $\sum_{n \leq x}a_n$. It remains an interesting possibility to study other conditions on growth that can replace the pointwise polynomial growth condition in Hypothesis \ref{hyp:Strong}; one recent exploration of average control appears in \cite{Alb25x}.

\section{Preliminaries}

We gather in this section  several introductory facts,  and the notation we will use in the proofs of the main Tauberian theorems.  In particular, in the proof of Theorem \ref{thm:Weak} we will use the Fourier transform and the Laplace transform; in the proof of Theorem \ref{thm:Strong} we will use the Mellin transform, and we gather useful properties of these transforms here.

\subsection{Notation}\label{sec_notation}
 
We use the  convention that for $A \in \C$ and $B \in \R_{\geq 0}$, the expression $A\ll B$ indicates that there exists a constant $C$ such that $|A| \leq CB$. The expression $A \ll_\kappa B$ indicates that the constant $C$ may depend on $\kappa$. 

 \subsection{Convergence in half-planes} \label{sec_convergence}
Given a general Dirichlet series $D(s) = \sum_{n} d_n \lam_n^{-s}$ with complex coefficients $\{d_n\}_n$, define the
 abscissa of   convergence  (in two equivalent forms by \cite[Lemma 6.11]{BatDia04}) by
 \begin{align*} 
 \sig_c &= \inf\{ \sig: \text{$D(s)$ converges for all $s$ with  $\Re(s)>\sig$}\} \\
 &= \inf\{ \sig: \text{$D(s)$ converges for some $s$ with $\Re(s) = \sig$}\} .
 \end{align*}
 The series $D(s)$ is locally uniformly convergent for $ \Re(s)>\sig_c$, so that it is in fact analytic in the right half-plane $\Re(s)> \sig_c$; see e.g. \cite[p. 13]{MonVau07}. 
 Define the
 abscissa of absolute convergence   
 \[\sig_a = \inf\left\{ \sig: \text{$\sum_{n=1}^\infty |d_n \lam_n^{-s}| < \infty$ for all $s$ with $\Re(s)>\sig$}\right\}.\]
 For any general Dirichlet series $D(s)$, these two abscissa satisfy a relation $\sig_c \leq \sig_a \leq \sig_c+L$, in which  we may take $L := \limsup_{n \maps \infty} \log n/ \log \lam_n$. Thus, in particular,    $L=1$  when $D(s)= \sum d_n n^{-s}$ is a standard Dirichlet series (that is $\lam_n=n$ for all $n$), while $L$ may be infinite for a general Dirichlet series $D(s)= \sum d_n \lam_n^{-s}$ if $\lam_n$ grows slowly. To verify the relation $\sig_c \leq \sig_a \leq \sig_c +1$ for the case of a standard Dirichlet series,   note that by studying $D(s- \sig_c)$ in place of $D(s)$,  it suffices to consider the case in which $\sig_c=0$. Then if $\Re(s) = \sig>1$, 
 \[ \sum_{n=1}^\infty \left|\frac{d_n}{n^s}\right| =\sum_{n=1}^\infty \frac{|d_n|}{n^\sig} \leq \left(\sup_{n \geq 1} \frac{|d_n|}{n^{\ep_0}} \right) \sum_{n = 1}^\infty \frac{1}{n^{\sig-\ep_0}} < \infty,
 \]
 where we take $\ep_0>0$ sufficiently small that $\sig-\ep_0>1$, and apply the fact that the $\sup$ is finite, since $\sum_{n}d_n n^{-s}$ converges for all $\Re(s)>0$ (so $\lim_{n \maps \infty} |d_n|n^{-\ep_0}=0$). This verifies that $L=1$ suffices for all standard Dirichlet series; see also \cite[Thm. 1.4]{MonVau07}.
 This argument adapts to the case where $D(s) = \sum_n d_n \lam_n^{-s}$ is a general Dirichlet series upon taking $L := \limsup_{n \maps \infty} \log n/ \log \lam_n.$ See for example the classic reference \cite[Ch. II \S 7 Thm. 9]{HarRie15}. 
 On the other hand, if $D(s)$ has non-negative coefficients, which is our focus, then $\sig_c=\sig_a$, since for any $s=\sig+it$, 
 \[   \sum_{n \geq 1}\left| \frac{d_n}{\lam_n^{\sig+it}}\right| = \sum_{n \geq 1} \frac{|d_n|}{\lam_n^\sig} = \sum_n \frac{d_n}{\lam_n^\sig} = D(\sig), \]
 which converges for all $\sig>\sig_c.$
 
 \subsection{Landau's theorem for non-negative coefficients}\label{sec_Landau}
 A classical theorem of Landau   shows that if the coefficients $d_n$ of a general Dirichlet series $D(s)  = \sum_{n} d_n \lam_n^{-s}$ are real and non-negative (so that the partial sum $\sum_{n \leq x} d_n$ is a monotone non-decreasing function of $x$) then $D(s)$ has a singularity at the point $s=\sig_c$ on the real line. That is to say, if the coefficients of $D(s)$ are non-negative, then the abscissa of convergence $\sig_c$ is precisely the largest real number at which $D(s)$ has a singularity.  Indeed, it suffices for the coefficients of $D(s)$ to be ``eventually'' non-negative (non-negative for all $n$ sufficiently large), since any   initial sum over finitely many indices is analytic. (See e.g. \cite[\S 9.2]{Tit32}, \cite[Thm. 6.31]{BatDia04}, \cite[Thm. 1.7]{MonVau07} or for general Dirichlet series \cite[Ch. II \S 8 Thm. 10]{HarRie15}.)
In our setting, under Hypothesis \ref{hyp:Weak}, since the series $A(s)$   has non-negative coefficients and   only one pole (namely $s=\al$)  in the region $\Re(s) \geq \al$, Landau's theorem shows that $A(s)$ has abscissa of convergence (and absolute convergence) $\sig_c=\sig_a=\al$.

\begin{remark}[Polynomial subgroup growth and convergence of $\zeta_G(s)$]\label{remark_zeta_group_convergence}   Recall the zeta function $\zeta_G(s)$ of a finitely generated group $G$, as defined in \S \ref{sec_zeta_group}.  The series $\zeta_G(s)$ has non-negative coefficients $a_n(G)$, so by \S \ref{sec_convergence} the series $\zeta_G(s)$ is convergent in a right half-plane $\Re(s)>\al$ for a given $\al>0$, if and only if it is absolutely convergent in that half-plane, if and only if it is holomorphic in that half-plane. There is a quantitative relationship between $\al$ and the rate of subgroup growth. Indeed, in one direction, suppose that $G$ has polynomial subgroup growth, so that for some $\al>0$, $S_x(G) = \sum_{n \leq x} a_n(G) \ll x^\al$ for   $x \maps \infty$. Then for a given $s \in \C$ with $\Re(s)=\sig>0$, partial summation shows that
\[ |\zeta_G(s)| \leq \sum_{n=1}^\infty \frac{a_n(G)}{|n^s|} =\sum_{n=1}^\infty \frac{a_n(G)}{n^{\sig}} = \lim_{x \maps \infty} (x^{-\sig} S_x(G)+ \sig \int_{1}^x S_u(G)  u^{-\sig-1}du).\]
This shows that $\zeta_G(s)$ converges absolutely for all $s$ with $\Re(s)>\al,$ and is holomorphic   in  this right half-plane. In the other direction, if $\zeta_G(s)$ converges (absolutely) for a given complex $s$ with $\Re(s)>0$, then  Kronecker's lemma (Theorem \ref{thm:Kronecker}) shows that $S_x(G)=\sum_{n\leq x} a_n = o(x^{\Re(s)})$ as $x \maps \infty$, and  $G$ has polynomial subgroup growth. In particular, if $\zeta_G(s)$ converges for all $\Re(s)>\al$ then for every $\ep>0$, $S_x(G) =o(x^{\al+\ep})$. 
\end{remark}

\subsection{Fourier transform}\label{sec_Fourier_transform}
Given integrable functions $f$ and $F$ on $\R$ we define the Fourier transform by 
\[ \Fscr(f)(\xi) = \int_{-\infty}^\infty f(x) e^{-i\xi x}dx\] 
and the Fourier inverse by 
\[ \Fscr^{-1}(F)(x) = \frac{1}{2\pi} \int_{-\infty}^{\infty} F(\xi) e^{i \xi x}d\xi.\] 
 These conventions are scaled so that  $\Fscr^{-1}(\Fscr(f)) = f$ and $\Fscr(\Fscr^{-1}(F))=F$. Additionally, for these identities to hold, $f$ and $\Fscr(f)$ (respectively, $F$ and $\Fscr^{-1}(F)$) must be appropriately well-behaved; for example, the first identity holds if $f$ is continuous and both $f$ and $\Fscr(f)$ are integrable (and analogously for the second identity). For any real $a$,  \[ \Fscr^{-1}(F(\cdot -a))(x) = \Fscr^{-1}(F)(x) e^{i a x}.\] Moreover, for any real $\lam>0$, if $F_\lam(\xi)$ is defined by $F_\lam(\xi) = \lam F(\lam \xi)$ then rescaling shows that $\Fscr^{-1}(F_\lam)(x) = (\Fscr^{-1}(F)) (x/\lam)$. Recall the class of Schwartz functions: a Schwartz function $f$ has the property that $f$ is infinitely differentiable and $f$ and all of its derivatives are rapidly decreasing, i.e. they decay faster than any power of $|x|$ as $|x| \maps \infty$. If $f$ is a Schwartz function, then $\Fscr(f)$ is also a Schwartz function \cite[Ch. 5 Thm. 1.3]{SSFour}.  

In this section we use properties of the Fourier transform to construct a type of weight function (Lemma \ref{lemma_build_kernel}) that will be used in the proof of Theorem \ref{thm:Weak}. We begin with:
\begin{lemnum}\label{lemma_Ft_even_fn}
Let $f$ be an even, integrable, real-valued function on $\R$. Then $\Fscr(f) $ is also a real-valued, even function. 
\end{lemnum}
 \begin{proof}
Since $f$ is even,
\[ \Fscr(f)(\xi) = \int_{0}^\infty f(x) e^{-i \xi x}dx +\int_{0}^\infty f(x) e^{i \xi x}dx  = 2 \int_0^\infty f(x) \cos (x\xi)dx,\] 
which is a real-valued even function of $\xi$.  
 \end{proof}

 \subsubsection{Additive convolution}\label{sec_additive_convolution}
In the context of applying the Fourier transform, we will denote the (additive) convolution of two integrable functions by  
\[ (f \star g) (x) = \int_{-\infty}^\infty f(x-y)g(y)dy.\] 
This notation for additive convolution is only used in this subsection  of the paper; while it is usually denoted  $f*g$, we use the $\star$ symbol to distinguish it from   multiplicative convolution, introduced in \S \ref{sec_notation_Mellin} in the context of the Mellin transform.

\begin{lemnum}\label{lemma_additive_convolution}
 
The following facts hold:
\begin{enumerate}[label=(\roman*)]
\item   If $g$ is integrable and $f \in C^k$, and the $j$-th derivative $  f^{(j)}$ is bounded for $j \leq k$, then $f\star g \in C^k$ and $(f\star g)^{(j)} = ( f^{(j)} \star  g)$ for $j \leq k$. 
\item If $f, g$ are Schwartz functions then $f\star g$ is a Schwartz function.
\item If $f$ and $g$ are integrable, $\Fscr(f \star g) = \Fscr(f) \cdot \Fscr(g)$.
\item If $f$ and $g$ are both compactly supported, then so is $f\star g$, and  $\supp(f\star g) \subseteq \overline{S}$ where $S=\{x+y: x \in \supp(f), y \in \supp(g)\}$. 
\item If $f$ and $g$ are piecewise continuous and compactly supported, then $f\star g$ is continuous.
\item Assuming the relevant integral converges, if $f$ and $g$ are  even, then $f\star g$ is even. 

\end{enumerate}
In particular, if $f$ is a real-valued, piecewise continuous,   even function that is compactly supported in $[-c_0,c_0]$, then $f\star f$ is a real-valued, continuous,  even function that is compactly  supported in $[-2c_0,2c_0]$.
\end{lemnum}
\begin{proof}
    Properties (i), (ii), (iii), (iv) can be found in standard real analysis books such as Folland \cite{Fol99}, in which they appear as Prop. 8.10, Prop 8.11, Thm. 8.22, and Prop. 8.6, respectively. Briefly, the key insight to (i) is that $f\star g$ is at least as smooth  as the smoother of $f,g$, since formally 
    \[ \partial^\al (f\star g)(x) = \partial^\al \int f(x-y)g(y) dy = \int \partial^\al f(x-y)g(y) dy = ((\partial^\al f)\star g) (x),\] 
    or equivalently evaluates to $(f \star (\partial^\al g))(x)$. The conditions assumed in (i) allow the differentiation under the integral sign to be justified. Property (ii) follows from (i) since $f, g$ are $C^\infty$ and decay rapidly (as do all their derivatives), so that $f \star g$ inherits these properties. Property (iii) holds by Fubini's theorem, since 
    \begin{align*}        
     \Fscr(f \star g)(\xi)& = \iint f(x-y) g(y) e^{-i \xi \cdot x} dy dx = \iint f(x-y)e^{-i \xi \cdot (x-y)} g(y) e^{-i \xi \cdot y} dx dy  \\
     &= \Fscr(f)(\xi) \int g(y) e^{-i\xi \cdot y} dy =
     \Fscr(f)(\xi)\Fscr(g)(\xi).
     \end{align*}
     Property (iv) can be seen from the fact that if $x \not\in \overline{S}$ then for any $y \in \supp(g)$, it must be that $x-y \not\in \supp(f)$, so that $f(x-y)g(y)=0$ for all $y$. Thus $(f\star g)(x)=0$. Property (v) follows from noting that 
     \[ (f \star g)(x+h) - (f\star g)(x) = \int[ f(x+h-y) - f(x-y)]g(y) dy,\]
     which can be written as a finite sum of integrals over compact intervals upon which $f$ is uniformly continuous. By uniform continuity, for any $\ep>0$, there is some $h_0$ such that for all $|h| \leq h_0$ the integral above is bounded by $\ep \int g(y)dy \ll_g \ep$, which suffices. For (vi), we observe that 
     \begin{align*}
     (f\star g)(-x) & = \int f(-x-y) g(y) dy  = \int f(x+y)g(y) dy = \int f(x-y) g(-y) dy \\
     & = \int f(x-y) g(y) dy = (f \star g)(x).
      \end{align*}
     
\end{proof}

A classical kernel arises from the choice $f(x) = \tfrac{1}{2}\mathbf{1}_{[-1,1]}(x)$. With this definition, we can compute   a function $k$ and its Fourier transform $K$:
\begin{align} 
k(x) &:=(f\star f)(x) = \tfrac{1}{4} \mathrm{meas} ( [x-1,x+1] \cap [-1,1]) = \tfrac{1}{2}(1 - \tfrac{|x|}{2}) \mathbf{1}_{[-2,2]}(x), \nonumber\\
K(\xi)&:=\Fscr(k)(\xi) = \Fscr(f\star f)(\xi) = (\Fscr(f)(\xi))^2 = (\tfrac{\sin \xi}{\xi})^2.\label{Fejer_dfn}
\end{align} 
The function $K$    is known as  the Fej\'er kernel; we record the fact that
\beq\label{Fejer_integral}
\int_{-\infty}^\infty K(\xi)d\xi  = 2\pi  \Fscr^{-1}(K)(0) = 2\pi k(0) = \pi.
\eeq

More generally,  we record a construction of a large class of kernel functions, any of which can be used in the proof of Theorem \ref{thm:Weak} (see Lemma \ref{lemma_allowable_approx_id}).
 \begin{lemnum}\label{lemma_build_kernel}
   Let $\eta$ be a compactly supported, even, non-negative Schwartz function on $\R$. Then  
  \[ k(x) = \frac{1}{2\pi} \frac{1 }{(\eta \star  \eta)(0)} (\eta \star  \eta)(x) \] 
is a compactly supported, even, non-negative Schwartz function. In particular, if $\eta$ is supported in $[-\del_0,\del_0]$ then $k$ is supported in $[-2\del_0,2\del_0]$. The Fourier transform $K:= \Fscr(k) $  is an even, non-negative Schwartz function with 
\[\int_{-\infty}^\infty K(\xi)d\xi = 2\pi \Fscr^{-1}(K)(0) = 2\pi k(0)=1.\]
 \end{lemnum} 
 \begin{proof}
     That $k$ is a  compactly supported, even, non-negative Schwartz function, follows from (iv), (vi) and (ii) of Lemma \ref{lemma_additive_convolution}; it is supported in $[-2\del_0,2\del_0]$ if $\eta$ is supported in $[-\del_0,\del_0]$ by (iv). The Fourier transform is an even real-valued Schwartz function by Lemma \ref{lemma_Ft_even_fn}. Moreover, by (iii) we see that $\Fscr(K) = (2\pi (\eta \star \eta)(0))^{-1}\Fscr(\eta)^2$, where $\Fscr(\eta)$ is real-valued by Lemma \ref{lemma_Ft_even_fn}, so that $\Fscr(K)$ takes only non-negative values. We have chosen the normalization correctly so that the average of $K$ is $1$.
 \end{proof}
 \subsection{Laplace transform}\label{sec_notation_Laplace}
Given a complex-valued measurable function $b(t)$   on $[0,\infty)$, the Laplace transform is defined for $s \in \C$ (where convergent) by 
\[B(s) = \int_0^\infty b(t) \exp(-st)dt.\]
Thus for example if $s=iy$ for $y \in \R$, then $B(s) = \Fscr(b)(y)$. 

  \begin{lemnum}\label{lemma_bG}
      For each integer $m \geq 1$, the function $b(t)$ defined    for $t \in [0,\infty)$ by
    \[b(t)=t^{m-1}/(m-1)!\]   
    has Laplace transform   $B(s)=s^{-m}$ for $\Re(s)>0$. That is, for all $s$ with $\Re(s)>0$, the following identity holds, in which the integral is (in particular) convergent:
    \beq\label{generic_pole_identity} 
    s^{-m}=\int_{0}^{\infty}\frac{1}{(m-1)!}t^{m-1}\exp(-st)\,dt.
    \eeq
\end{lemnum} 
\begin{proof} Observe that for all real $\sig>0$, 
\[ (m-1)!= \Gamma(m) = \int_0^\infty t^{m-1} \exp(-t) dt = \int_0^\infty (\sig t)^{m} \exp(-\sig t) \frac{dt}{t} ,\]
in which the right-most integral is convergent.
This shows that (\ref{generic_pole_identity}) holds for all
holds for all real $s=\sig>0$. The integral in (\ref{generic_pole_identity}) is convergent, and holomorphic in $s$, for all $s$ with $\Re(s)>0$; by analytic continuation (\ref{generic_pole_identity}) also holds  for all $s$ in the mutual region of holomorphicity of both sides, which is the right half-plane $\Re(s)>0$.
\end{proof}

A general Dirichlet series  $D(s)=\sum_{n}d_n \lam_n^{-s}$ associated to complex coefficients $\{d_n\}_n$ can be expressed in terms of the Laplace transform of   the partial sum function 
\[ d(t):= \sum_{\lam_n \leq \exp(t)} d_n\]
in order to access functional analytic tools. Precisely, suppose $D(s)$ converges for all $s$ with $\Re(s)>\al$. 
  Express the general Dirichlet series as a Lebesgue-Stieltjes integral and apply integration by parts to see that for $\Re(s)>\al$,
\[
D(s)=\sum_{n=1}^{\infty}d_n \lam_n^{-s} = \int_{0}^{\infty}\exp(-st)\,d(d(t))\\
= \lim_{t\to \infty}\frac{d(t)}{\exp(st)}+s\int_{0}^{\infty}d(t)\exp(-st)\, dt.
\]
By Kronecker's lemma (Theorem \ref{thm:Kronecker} with $x=\exp(t)$), for  every $s$ with $\Re(s)>\al$,
\[
\lim_{t\to \infty}\frac{d(t)}{\exp(st)} = 0.
\]
 Thus for $\Re(s)>\al$,
\beq\label{Dir_series_as_Laplace_transf}
D(s)=\sum_{n=1}^{\infty}d_n \lam_n^{-s} = s\int_{0}^{\infty}d(t)\exp(-st)\, dt.
\eeq

 \subsection{Mellin transform}\label{sec_notation_Mellin}

There are two standard conventions for defining the Mellin transform. 

 Convention 1: Given a function $F(x)$, one convention  is to define the Mellin transform as a Riemann-Stieltjes integral (where convergent):
\beq\label{Mellin_dfn1}
\mathcal{M}(F)(s) = \int_0^\infty x^{-s}dF(x).
\eeq
This convention is used for example in \cite{BatDia04}, which applies this definition to the class of $F$ that are supported on $[1,\infty)$ (so that nominally the integral is over $[1,\infty)$),   right-continuous, and  locally of bounded variation. We will   use this convention, for functions in this class, in \S \ref{sec_hyp_weak_counterex_further}, and in \S \ref{sec_hyp_weak_counterex} when proving Theorem \ref{thm_hyp_weak_counterex}, in order to cite that source conveniently.
For functions in this class, an application of integration by parts alternatively expresses 
\[\mathcal{M}(F)(s) = s \int_1^\infty F(t) t^{-s-1}dt,\]
as long as the boundary terms vanish  (which occurs for example if $F(1)=0$ and  $|F(t)|=o(t^{\Re(s)}$)).
For example, given a standard Dirichlet series $D(s)=\sum_{n} d_n n^{-s}$ associated to complex coefficients $\{d_n\}_n$, with partial sum function $d(x)=\sum_{1 \leq n \leq x} d_n$, we have
\[ D(s)=\mathcal{M}(d)(s), 
\]
as shown in  \cite[\S 6.1 example (1)]{BatDia04}.
This normalization is particularly convenient for Dirichlet series, since the signed  exponent $-s$ in (\ref{Mellin_dfn1}) makes visible the convergence property of the integral, for $s$ in a right-half plane. 

Convention 2: Another standard convention employs the opposite sign $+s$. 
Given an appropriate function $\varphi$, define for $s \in \C$ the Mellin transform 
 by 
\beq\label{Mellin_dfn2}
\hat{\varphi}(s) :=  \int_0^{\infty} \varphi(u) u^{s-1}du .
\eeq
This convention is used for example in \cite[\S 5.1]{MonVau07}.
This convention is in some ways more natural when the Perron formula (Mellin inversion) is employed; we will apply this convention in the proof of Theorem \ref{thm:Strong} in \S \ref{sec_B_prelim} - \S \ref{sec_B_proof}.
 
 To relate this to the complementary convention (\ref{Mellin_dfn1}), note that $\hat{F}(s) = -\tfrac{1}{s}\Mcal(F)(-s)$. In the specific setting of   a standard Dirichlet series $D(s) = \sum_{n}d_n n^{-s} = \mathcal{M}(d)(s)$ with corresponding partial sum truncation $d(x) = \sum_{n \leq x} d_n$, note that
\[
(d(1/\cdot))\hat{\;}(s) =\int_0^\infty d(1/u) u^{s-1}du= \int_0^\infty \sum_{n =1}^\infty  d_n \mathbf{1}_{\leq 1}(nu) u^{s} \frac{du}{u} = D(s) \int_0^1 v^{s-1} dv=  \frac{D(s)}{s},
\]
under the change of variables $v=nu$.
The definition (\ref{Mellin_dfn2}) also relates to the Laplace transform   by a change of variables:
\[ (\varphi(-\log \cdot))\hat{\;}(s) = \int_0^\infty \varphi(-\log u) u^s \frac{du}{u} = \int_{-\infty}^\infty \varphi(t) \exp (-st) dt.\]
A similar change of variable relates the Mellin and Fourier transforms; see e.g. \cite[p.141]{MonVau07}. 

\subsubsection{Properties of the Mellin transform}
Here we record  useful properties of the Mellin transform (\ref{Mellin_dfn2}) and multiplicative convolution.
 \begin{lemnum} \label{lem:mellin_cpt_cts}
      If $f : (0,\infty) \to \C$ is compactly supported away from the origin and continuous, then  the Mellin transform
      \[
       \widehat{f}(s) = \int_0^{\infty} f(u) u^{s-1} du
      \]
      is an entire function of $s \in \C$.
  \end{lemnum}
   
  \begin{remark}
      In particular, $\widehat{f}(s)$ is equal to this integral representation for all $s \in \mathbb{C}$ without the need for any analytic continuation. 
  \end{remark}
    \begin{proof}
        By assumption, there exists $0 < a < b< \infty$ such that the integral  is equal to $\int_a^b f(u)u^{s-1}du$ for all $s \in \C$. Since this integrand is continuous on the compact interval $[a,b]$, the integral converges. This implies that $\widehat{f}(s) = \int_0^{\infty} f(u) u^{s-1} du$ for all $s \in \C$. 

        To show $\widehat{f}(s)$ is entire, observe that our assumptions imply that the map $(s,u) \mapsto f(u) u^{s-1}$ and its $\frac{\partial}{\partial s}$ partial derivative are continuous on $\mathbb{C} \times  [a,b]$. By Leibniz's integral rule, we may differentiate $\frac{\partial }{\partial s}$ under the integral sign for $\int_a^b f(u) u^{s-1} du = \int_0^{\infty} f(u) u^{s-1} du = \widehat{f}(s)$. This establishes that $\widehat{f}(s)$ is complex differentiable and hence entire. 
    \end{proof}

\begin{lemnum} \label{lem:convolution}
 Let $f : (0,\infty) \to \C$ and $g : (0,\infty) \to \C$ be bounded and $L^1$ on $(0,\infty)$ with respect to the Haar measure $d^{\times} u = du/u$. The convolution $f \ast g : (0,\infty) \to \C$ is defined by
 \[
 (f \ast g)(u) = \int_0^{\infty} f(u/w) g(w) \frac{dw}{w}
 \]
 for $u \in (0,\infty)$, and satisfies all of the following:
 \begin{enumerate}[label=(\roman*)]
    \item $f \ast g = g \ast f$ and $||f \ast g||_{\infty} \leq ||f||_{\infty} ||g||_1$. 
     \item  Fix $0 < a < b < \infty$ and $0 < c< d < \infty$. If $\mathrm{supp}(f) \subseteq [a,b]$ and $\mathrm{supp}(g) \subseteq [c,d]$ then   $\mathrm{supp}(f \ast g) \subseteq [ac, bd]$. 
     \item If $f$ or $g$ is continuous, then $f \ast g$ is continuous. 
     \item If $f$ and $g$ are compactly supported away from the origin and continuous, then the Mellin transform of $f \ast g$ is entire and, moreover,
     \[
     \widehat{f \ast g}(s) = \widehat{f}(s) \widehat{g}(s) \qquad \text{ for } s \in \C. 
     \]
     \item Fix $0 < a < b < \infty$. The convolution $f \ast \mathbf{1}_{[a,b]}$ is continuous and, if additionally $f$ is $C^j$ for some integer $j \geq 0$, then $f \ast \mathbf{1}_{[a,b]}$ is $C^{j+1}$. 
 \end{enumerate}
\end{lemnum}  
\begin{proof}
    Observe for $u \in (0,\infty)$ that 
    \[
    \int_0^{\infty} |f(u/w) g(w)| \frac{dw}{w} \leq ||f||_{\infty} \int_0^{\infty} |g(w)| \frac{dw}{w} = ||f||_{\infty} ||g||_1,
    \]
    so $(f \ast g)(u)$ is defined by Lebesgue's dominated convergence theorem. 
    
    (i) The claim $f \ast g = g \ast f$ follows from the substitution $w \mapsto u/w$. The $L^{\infty}$-norm claim follows from the above inequality. 
    
    (ii) For $u,w \in (0,\infty)$, observe that $f(u/w)g(w) \neq 0$ only if $w \in [c,d]$ and $u/w \in [a,b]$ and hence only if $u = u/w \cdot w \in [ac, bd]$. Thus, $(f \ast g)(u) = 0$ whenever $u \not\in [ac,bd]$.

    (iii) By (i), we need only consider the case when $f$ is continuous. Fix $u \in (0,\infty)$. Let $(u_k)_k$ be any sequence of points in $(0,\infty)$ such that $u_k \to u$ as $k \to \infty$. The function $ h_k(w) := f(u_k/w) g(w)$ is dominated by the function $w \mapsto ||f||_{\infty} |g(w)|$ which is integrable on $(0,\infty)$ with respect to $d^{\times} w = dw/w$. Setting $h(w) := f(u/w) g(w)$, we have  $h_k \to h$ pointwise as $k \to \infty$ by continuity of $f$. Thus, by Lebesgue's dominated convergence theorem,
    \[
    (f \ast g)(u_k) = \int_0^{\infty} h_k(w) dw \longrightarrow \int_0^{\infty} h(w) dw = (f \ast g)(u) 
    \]
    as $k \to \infty$. This proves continuity of $f \ast g$. 
 
    (iv) By (i), (ii), and (iii), the convolution $f \ast g$ is also bounded, compactly supported away from the origin, and continuous. By applying \cref{lem:mellin_cpt_cts} to $f \ast g$, then Fubini's theorem, and then \cref{lem:mellin_cpt_cts} again to $f$ and $g$, it follows that
    \begin{align*}
    \widehat{(f * g)}(s) 
    & = \int_0^{\infty} \Big( \int_0^{\infty} f(u/w) g(w) \frac{dw}{w} \Big) u^{s-1} du \\
    & = \int_0^{\infty} \int_0^{\infty} f(u/w) (u/w)^{s-1} g(w) w^{s-2} du dw  = \widehat{f}(s) \widehat{g}(s)
    \end{align*}
    for all $s \in \mathbb{C}$. 
    
    (v) Using the substitution $w \mapsto u/w$, we have that
    \[
    (f \ast \mathbf{1}_{[a,b]})(u) = \int_a^b f(u/w) \frac{dw}{w} = \int_{u/b}^{u/a} f(w) \frac{dw}{w}. 
    \]
    The continuity claim follows by the integrability of $f$ with respect to $d^{\times} w = dw/w$. If $f$ is $C^j$ for some $j \geq 0$, then by the fundamental theorem of calculus, 
    \[
    (f \ast \mathbf{1}_{[a,b]})'(u) = \frac{f(u/a)}{u/a} - \frac{f(u/b)}{u/b}
    \]
    for $u \in (0,\infty)$, so the derivative $(f \ast \mathbf{1}_{[a,b]})'$ is $C^j$ and hence $f \ast \mathbf{1}_{[a,b]}$ is $C^{j+1}$. 
\end{proof}

\section{Hypothesis \ref{hyp:Kronecker} }\label{sec_Kronecker}

A proof of Kronecker's lemma for standard Dirichlet series may be found in \cite[Thm. 2.12]{Hil05}.
 (See also a very general statement in \cite[Lemma 4.2.2]{QueQue13}.) Here we provide a proof for general Dirichlet series. It is useful to state a partial summation formula for sums indexed by an increasing sequence $\{\lam_n\}_n$, found in \cite[Thm. A]{Ing32}. 
\begin{lemnum}\label{lem_partial_summation_general}
Let $\phi$ be a continuously differentiable complex-valued function on $(0,\infty)$.  Let  $\{\lam_n\}_n$ be a non-decreasing sequence of positive real numbers tending to infinity as $n \maps \infty$. Let $\{c_n\}_n$ be a sequence of complex numbers, and define for any $0<x_1 \leq x_2$,  
 \[ C(x_1,x_2) = \sum_{x_1<\lam_n \leq x_2} c_n.\]
 Then 
 \[ \sum_{x_1< \lam_n \leq x_2} c_n \phi(\lam_n) = C(x_1,x_2)\phi(x_2)    - \int_{x_1}^{x_2} C(x_1,u) \phi'(u)du.\]
\end{lemnum}
\begin{proof}
As in   proofs of similar formulas, the proof begins by writing
\[C(x_1,x_2)\phi(x_2) - \sum_{x_1< \lam_n \leq x_2} c_n \phi(\lam_n) = \sum_{x_1< \lam_n \leq x_2} c_n (\phi(x_2) - \phi(\lam_n)) = \sum_{x_1<\lam_n \leq x_2} \int_{\lam_n}^{x_2} c_n \phi'(u) du.\]
Now to interchange the order of summation and integration, define auxiliary functions by setting $g_n(u) = c_n \phi'(u)$ for $u \geq \lam_n$ and $g_n(u)=0$ for $u < \lam_n$. Then certainly
\[ \sum_{x_1< \lam_n \leq x_2} \int_{x_1}^{x_2} g_n(u) du = \int_{x_1}^{x_2} \sum_{x_1< \lam_n \leq x_2} g_n(u) du.\]
Applying this to the previous expression, we can complete the proof:
\[\sum_{x_1<\lam_n \leq x_2} \int_{\lam_n}^{x_2} c_n \phi'(u) du = \int_{x_1}^{x_2}\sum_{x_1<\lam_n \leq u}  c_n \phi'(u) du
= \int_{x_1}^{x_2} C(x_1,u)\phi'(u)du.\]
\end{proof}

\subsection{Proof of  Theorem \ref{thm:Kronecker}}
To prove Theorem \ref{thm:Kronecker}, consider now the given general Dirichlet series $D(s) = \sum_n d_n \lam_n^{-s}$ with complex coefficients $\{d_n\}_n$, which converges for a given $s$ with $\Re(s)>0$. We will apply the partial summation lemma with $c_n=d_n \lam_n^{-s}$ and $\phi(u) = u^s$. Supposing $x \geq 1$ is given, for any $1 \leq N \leq x$ of our choice, we can trivially write 
\beq\label{D_decomp} \sum_{1 \leq \lam_n \leq x} d_n = \sum_{1 \leq \lam_n \leq N} d_n + \sum_{N<\lam_n \leq x} (d_n \lam_n^{-s})\phi(\lam_n) .
\eeq
By an application of partial summation,  the second term on the right-hand side evaluates to 
\[
x^s\left(\,\sum_{N< \lam_n \leq x} d_n \lam_n^{-s}\right) - \int_{N}^x \left(\,\sum_{N< \lam_n \leq u} d_n \lam_n^{-s}\right) \frac{d}{du}(u^s) du.
\]
By the assumption that $D(s) = \sum_n d_n \lam_n^{-s}$ converges, truncations of this series form a Cauchy sequence, so that  for every $\ep>0$ we may make a choice of $N=N_\ep$ sufficiently large (depending on $\ep, s$) such that  $|\sum_{N_\ep< \lam_n \leq u} d_n \lam_n^{-s} |\leq \ep$ for all $u\geq N_\ep$. Consequently, upon writing $\sig = \Re(s)$, for all $x \geq N_\ep$ the first term in the previous expression   is bounded above in absolute value by $\ep x^\sig$ while the integral term   may be estimated by 
\[ \int_{N_\ep}^x \left|\sum_{N_\ep< \lam_n \leq u} d_n \lam_n^{-s}\right| \left| \frac{d}{du}(u^s)\right| du \leq \ep |s| \int_{N_\ep}^x u^{\sig-1} du = \ep\frac{|s|}{\sig}(x^\sig-N_\ep^\sig) \leq \ep\frac{|s|}{\sig}x^\sig.\]
Applying these bounds in (\ref{D_decomp}) with $N=N_\ep$ shows that for every $x \geq N_\ep$, 
\[ \left|\sum_{1 \leq \lam_n \leq x} d_n\right| \leq \left|\sum_{1 \leq \lam_n \leq N_\ep} d_n\right| + \ep x^\sig + \ep \frac{|s|}{\sig} x^\sig.\]
With $\ep$ and $N_\ep$ fixed, we let $x \maps \infty$ to see that 
\[ 
\limsup_{x \maps \infty} x^{-\sig}\left|\sum_{1 \leq \lam_n \leq x} d_n\right| \leq   \ep  + \ep \frac{|s|}{\sig} .\]
Since $\ep>0$ may be taken arbitrarily small, we conclude that $\sum_{\lam_n \leq x} d_n = o(x^\sig)$ as $x \maps \infty$, completing the proof of Theorem \ref{thm:Kronecker}.

 \subsection{Proof of Proposition \ref{prop_limit_of_Kronecker}}\label{sec_PNT_counterexample}

The counterexample described in Proposition \ref{prop_limit_of_Kronecker} complements the following  Abelian theorem adapted from \cite[Ch. II \S 10]{Ing32}. 
\begin{lemnum}\label{lemma_abelian_asymptotic_implies_order_of_pole}
Suppose 
 $D(s) = \sum_n d_n n^{-s}$  is a (standard) Dirichlet series with complex coefficients $\{d_n\}_n$ that is absolutely convergent for real $s=\sig >1$.  
If 
 \beq\label{lim_known} \lim_{x \maps \infty} \frac{\sum_{n \leq x}d_n}{x} =c_0, 
 \eeq
 then 
 \beq\label{residue_known}
 \lim_{\sig \maps 1^+} (\sig-1)D(\sig)=c_0.
 \eeq
 \end{lemnum}
 The proof of this lemma, which we omit, follows from a mild adaptation of the elementary proof   of Chebyshev's theorem that if $ \lim_{x \maps \infty}\pi(x)/(x(\log x)^{-1})$ exists, it must be 1; details may be found in \cite[Ch. I Thm. 6]{Ing32}.
 The converse of the lemma, which would play the role of a Tauberian theorem, is false. This is a barrier to proving the Prime Number Theorem by simple real-variable methods.  Indeed, upon setting $D(s) = -\zeta'(s)/\zeta(s)$, then $d(x) = \sum_{n \leq x} \Lambda(n) = \psi(x)$ in Chebyshev's notation. Simple real-variable considerations show that $\lim_{\sig \maps 1^+} -(\sig-1) \zeta'(\sig)/\zeta(\sig) =1$, so that if the converse of the lemma were true, one could immediately deduce that $\psi(x) \sim x$,   equivalent to the Prime Number Theorem.

Now we prove Proposition \ref{prop_limit_of_Kronecker}, which 
shows that the converse of the lemma is false. 
For $\Re(s)>1$ define the Dirichlet series
\beq\label{Ingham_example}
A(s) = \zeta(s)  + \frac{1}{2}\zeta(s-i)+ \frac{1}{2}\zeta(s+i) =\sum_{n =1}^\infty \frac{1 + \cos (\log n)}{n^s}=: \sum \frac{a_n}{n^s}.
\eeq
The first representation shows that $A(s)$ has a meromorphic continuation to $\C$ with only a simple pole at each of $s=1, s=1+i$ and $s=1-i$, while the second makes visible that $a_n \in [0,2]$. 
Since $\zeta(1\pm i)$ is bounded at $s=1$, and $\zeta(s)$ has a simple pole of residue $1$ at $s=1$, (\ref{residue_known}) holds with $c_0=1$. 

But the limit analogous to (\ref{lim_known}) does not exist. This can be seen by computing $a(x) :=\sum_{n \leq x}a_n$ via   the Euler-Maclaurin summation formula \cite[Lemma 4.1]{IwaKow04}:  for integers $a<b$ and a $C^1$ function $\phi$ on $[a,b]$, 
\beq\label{EulerMaclaurin}
\sum_{a< n \leq b}\phi(n) = \int_a^b \phi(u)  du + \tfrac{1}{2}\left(\phi(b) - \phi(a)\right)+\int_a^b   \phi'(u)\left(\{u\}-\tfrac{1}{2}\right)du.
\eeq
Here $\{u\} = u - \lfloor u\rfloor$ is the fractional part of $u$. Applying this with $\phi(u) = 1+ \cos (\log u)$, 
so that $|\phi'(u)|\leq 1/u$, yields
\begin{align*}
   a(x) & =\sum_{1 \leq n \leq x}a_n
   = \int_1^x \left(1+\cos(\log u)\right) du + O(1) + O\left(\int_1^x \frac{du}{u}\right)
     \\
   &  = x + \tfrac{1}{2} x \left(\cos (\log x) + \sin (\log x)\right) + O(\log x) = x + \tfrac{1}{\sqrt{2}}x \cos (\log x - \pi/4) + O(\log x).
\end{align*}
In particular $\lim_{x \maps \infty} x^{-1}a(x)$ does not exist.
This completes the proof of Proposition \ref{prop_limit_of_Kronecker}.

To understand the underlying reason that this limit does not (indeed cannot) exist, the following fact is helpful:  if (\ref{lim_known}) is true for a Dirichlet series $D(s)$ that has a (simple) pole at $s=1$, then 
a short argument proves that $D(s)$ can have no pole on the line $\Re(s)=1$ aside from at $s=1$. This follows from the arguments presented in \cite[Ch. II \S 9-10]{Ing32}.
In particular, since the example $A(s)$ constructed in (\ref{Ingham_example}) has three poles on the line $\Re(s)=1$, this also confirms that $\lim_{x\maps \infty} a(x)/x=c_0$ cannot hold for any $c_0$.

\section{Hypothesis \ref{hyp:Weak}: Preliminaries for   Theorem \ref{thm:Weak}}\label{sec_thm_weak}

We prove \cref{thm:Weak} by adapting a special case of the   Ikehara-Delange method, developed in the original work of Delange \cite{Del54,Del55}, and also reported on by Narkiewicz \cite[Theorem 3.8 and Corollary]{Nar83}. 
In this section, we provide an intuition for how to express Theorem \ref{thm:Weak} in terms of Laplace transforms, and  formally state the essential ingredient,   Theorem \ref{thm_Delange}, which is  functional analytic in nature. We then deduce Theorem \ref{thm:Weak} and Corollary \ref{cor_thm_Delange_Stark} from Theorem \ref{thm_Delange}, and explain the intuition behind this theorem, before turning to the rigorous proof  
 in \S \ref{sec_Delange}. Delange's work proves far more general results than we explain in this paper  (see \S \ref{sec_lit_HypA}); our explanation can be used as a blueprint to understand the more general work.  
 
\subsection{Framing the problem with Laplace transforms}\label{sec_thm_weak_intuition}
Let $A(s) = \sum_n a_n \lam_n^{-s}$ satisfy  Hypothesis \ref{hyp:Weak} for $\al>0$ and an integer $m \geq 1$. Under this hypothesis,  $A(s) = g(s)/(s-\alpha)^m + h(s)$ where $g$ is a polynomial of degree at most $m-1$ with $g(\al) \neq 0$ and $h$ is holomorphic in $\Re(s) > \alpha$ and continuous in $\Re(s) \geq \alpha$.    Since by hypothesis $A(s)$ has a pole at $s =\al$ (so in particular $A(s) \not\con 0$) and the coefficients $\{a_n\}$ are non-negative, it follows that $A(\sig)$ is positive for all real $\sig>\al$. In particular,  $\lim_{\sig \maps \al^+}(\sig-\al)^mA(s) = g(\al) >0$.  

For $t \geq 0$, let 
\beq\label{a_partial_sums_exp}a(t):=\sum_{\lam_n\le \exp(t)}a_n.\eeq
 (Setting $x=\exp(t)$ reveals the original partial sum (\ref{eqn:original-sum}).) Since each $a_n \geq 0$, $a(t)$ is non-decreasing for $t\ge 0$ (and is non-negative).  
In the present notation, Theorem \ref{thm:Weak} (after the substitution $x=\exp(t)$) claims that 
\beq\label{thm_weak_claim}
a(t) = \left(c_0+o(1)\right)\exp(\al t)\frac{t^{m-1}}{(m-1)!} \qquad \text{as $t \maps \infty$, with $c_0:=g(\al)/\al$.}
\eeq

Recall from the comment after Hypothesis \ref{hyp:Weak} that the Dirichlet series $A(s)$ is   convergent for $\Re(s)>\al$, hence  by (\ref{Dir_series_as_Laplace_transf}), in this region
\beq\label{A_Laplace}
\frac{A(s)}{s} =\int_{0}^{\infty}a(t)\exp(-st)\, dt.
\eeq
 (Since $\alpha>0$, $1/s$ does not introduce a pole in the region $\Re(s)\geq \al$.) That is, the Laplace transform of $a(t)$ is $A(s)/s.$
 Roughly speaking, our goal is to invert the Laplace transform, so that properties of $a(t)$ can be deduced from properties of $A(s)$. 

Recall from   Lemma \ref{lemma_bG}  that the Laplace transform of the function $t^{m-1}/(m-1)!$ is $s^{-m}$, which we can think of as the generic function with an order $m$ pole at $s=0$.
This is familiar: upon setting $x=\exp(t)$, the computation in Remark \ref{remark_residue} shows that the highest-order contribution to the residue of $A(s)\exp(st)/s$ at $s=\al$ takes the form $c_0\exp(\al t)\frac{t^{m-1}}{(m-1)!}$.   This suggests that we should compare $A(s)/s$ to $c_0(s-\al)^{-m}$ in terms of Laplace transforms. Precisely,
evaluating the identity (\ref{generic_pole_identity}) at $s-\al$ shows that
\beq\label{generic_pole_identity_shifted}  \frac{1}{ (s-\al)^m} =\int_{0}^{\infty}\frac{1}{(m-1)!}t^{m-1}\exp(-(s-\al)t)\,dt \quad \text{for all $s$ with $\Re(s)>\al$.}
\eeq
Rescaling this by $c_0$ and differencing   with $A(s)/s$ as expressed in (\ref{A_Laplace}) shows that for $\Re(s)>\al$,
\begin{align}\label{A_diff}
     \frac{A(s)}{s} - \frac{c_0}{(s-\al)^m} & = \int_0^\infty \left( a(t)  - \frac{c_0}{(m-1)!} \exp(\al t) t^{m-1} \right) e(-st)dt  \\
     & =  \int_0^\infty b(t)\exp(\al t) \left[ \frac{a(t)}{b(t)\exp(\al t)}  - c_0 \right] e(-st)dt, \nonumber
\end{align}
with $b(t) :=  \frac{1}{(m-1)!}  t^{m-1},$ say. 

\subsubsection{Heuristic idea}  Suppose for the purpose of illustration that the Laurent series for $A(s)/s$ at $s=\al$ takes the form $\frac{c_0}{(s-\al)^m} + h(s)$ where $h(s)$ is holomorphic for $\Re(s)> \al$ and continuous for $\Re(s) \geq \al$. Then   the left-hand side of (\ref{A_diff}) is continuous for $\Re(s) \geq \al$; if the identity were verified  in the larger region $\Re(s) \geq \al$,  the right-hand side would be continuous there as well, and in particular would be a convergent integral at $s=\al$. This suggests the term in square brackets must be $o(1)$ in order for the integral to converge when $s=\al$. We recognize this as the claim (\ref{thm_weak_claim}) of Theorem \ref{thm:Weak}.
More generally, if the Laurent series expansion for $A(s)/s$ takes the form 
\[ \frac{c_0}{(s-\al)^{m}} +\frac{c_1}{(s-\al)^{m-1}}+ \cdots + \frac{c_{m-1}}{(s-\al)} + h(s),\]
for some constants $c_1,\ldots, c_{m-1}$, and with $h(s)$  holomorphic for $\Re(s)> \al$ and continuous for $\Re(s) \geq \al$,  then the left-hand side of (\ref{A_diff}) has a pole of order at most $m-1$, and we aim to arrive at the same conclusion. 
In either case, to extract the conclusion that $a(t) = (c_0+o(1))\exp(\al t)b(t)$ as $t \maps \infty$, thereby proving Theorem \ref{thm:Weak}, we will apply the functional analytic result stated next as Theorem \ref{thm_Delange}.

\subsection{The functional analytic result}

 \begin{thmnum}\label{thm_Delange}
Let $b(t)$ be a measurable real-valued function on $[0,\infty)$, bounded on any finite interval and moreover, with Laplace transform
\[
G(s) = \int_{0}^{\infty}b(t)\exp(-st)\,dt
\]
that is convergent for $\Re(s)>0$.  Suppose that for a fixed integer $m \geq 1$,   
    \beq\label{b_lim_hypothesis} 
    \lim_{t \maps \infty} \frac{b(t)}{t^{m-1}}  = \frac{1}{(m-1)!}= \frac{1}{\Gamma(m)}.
    \eeq
Let $a(t)$ be a   non-decreasing real-valued function on $[0,\infty)$, with Laplace transform 
\[
f(s) = \int_{0}^{\infty}a(t)\exp(-st)\,dt
\]
that is  convergent for $\Re(s)>\alpha$, for a fixed real $\al>0$.  Suppose that for $m \geq 1$ as above and real constant $c_0>0$, the function defined by
\[
F(s) = f(s) - c_0G(s-\alpha)
\]
has the property that
\[ F(s) = \frac{\tilde{g}(s)}{(s-\al)^{m-1}} + \tilde{h}(s)\]
for a polynomial $\tilde{g}(s)$ of degree at most $m-2$, and a function $\tilde{h}(s)$ that is holomorphic for complex numbers $s$ with $\Re(s)>\al$ and continuous for $\Re(s) \geq \al$. 
(By convention, $\tilde{g}(s)\con 0$ if $m=1$.) Then   
\[
a(t)=(c_0+o(1))\exp(\alpha t)b(t) \qquad \text{as $t\to\infty$.}
\]
\end{thmnum}
We can summarize the hypothesis on $F(s)$ (slightly loosely) as saying that $F(s)$ is holomorphic on $\Re(s) > \al$, and continuous on $\Re(s) \geq \al$, except for a pole of order at most $m-1$ at $s=\al$.

 We will prove Theorem \ref{thm_Delange} in \S \ref{sec_Delange}. In that section, we generally follow the strategy described in Narkiewicz's exposition from \cite[Thm. 3.8]{Nar83}; that presentation is itself closely based on \cite{Del54}.  We deviate from the strategy of \cite[Thm. 3.8]{Nar83} in two respects: first, we strengthen the theorem, since we only assume continuity on the line $\Re(s) = \al$, whereas the original result assumes holomorphicity in $\Re(s) \geq \al$ aside from a pole of specified order at $s=\al$. This strengthening allows us to obtain Corollary \ref{cor_thm_Delange_Stark}. Second, we fix the proof of the result we call Lemma \ref{lemma_h_convergence}; see Remark \ref{remark_Narkiewicz}.

 \subsection{Consequences of the functional analytic result}
Theorem \ref{thm_Delange} immediately implies Theorem \ref{thm:Weak}.

\begin{proof}[Deduction of Theorem \ref{thm:Weak}]
Define $b(t)=t^{m-1}/(m-1)!$ and its Laplace transform $G(s) = s^{-m}$ for $\Re(s)>0$ as in Lemma \ref{lemma_bG}.
Under Hypothesis \ref{hyp:Weak} for $\al>0$ and an integer $m \geq 1$, for $\Re(s)>\al$ we may write
\[ A(s) = \frac{g(s)}{(s-\al)^m}   + h(s)\] 
for a polynomial $g$ of degree at most $m-1$ with $g(\al) \neq 0$ and a function $h(s)$ that is holomorphic for $\Re(s)>\al$ and continuous for $\Re(s) \geq \al$. 
  Let $a(t)$ denote the partial sum as in (\ref{a_partial_sums_exp}),  and define the Laplace transform of $a(t)$ by 
\begin{equation}\label{integralA}
    f(s):=   \frac{A(s)}{s}=  \int_{0}^{\infty}a(t)\exp(-st)\,dt,
\end{equation}
which  is convergent for $\Re(s)>\alpha$ by Kronecker's lemma, as in (\ref{A_Laplace}). 

 As shown in \S \ref{sec_thm_weak_intuition}, $g(\al) = \lim_{\sig \maps \al^+} (\sig-\al)^m A(\sig)> 0,$ so we may define $c_0 := g(\al)/\al>0$. Consider 
\[
F(s):=f(s) - c_0G(s-\alpha) = \frac{g(s)}{  s(s-\alpha)^m} -\frac{c_0}{(s-\alpha)^m}+ \frac{h(s)}{s} = \frac{\al g(s)-s g(\al)}{\al s(s-\alpha)^m}  + \frac{h(s)}{s}.
\]  
Note that the function 
\[ g_1(s) := \frac{\al g(s) -sg(\al)}{s-\al}\]
is holomorphic in $\Re(s) \geq \al$ since the numerator is holomorphic in this region and vanishes at $s=\al$. (Indeed, since $g(s)$ is a polynomial, this reasoning shows $g_1(s)$ is entire, although we do not require the full strength of this property.) 
Thus the function $\tfrac{g_1(s)}{\al s}$ is also holomorphic for $\Re(s) \geq \al$, and its Taylor series about $s=\al$ may be written as $\sum_{j=0}^{m-2} c_j(\al)(s-\al)^j +\sum_{j \geq m-1} c_j(\al)(s-\al)^j=: \tilde{g}(s) + \tilde{h}(s)$, in which $\tilde{g}(s)$ is a polynomial of degree $m-2$ and $\tilde{h}(s)(s-\al)^{-(m-1)}$ is holomorphic for $\Re(s) \geq \al$. Consequently, using this notation,
\[ F(s) = \frac{\tilde{g}(s)}{(s-\al)^{m-1}} + \left( \frac{h(s)}{s} + \frac{\tilde{h}(s)}{(s-\al)^{m-1}}\right)\]
 satisfies the hypothesis of Theorem \ref{thm_Delange}.
 Thus by the theorem, 
\[ a(t) = (c_0+o(1))\exp(\al t)b(t) = (c_0+o(1)) \exp(\al t) \tfrac{t^{m-1}}{(m-1)!}.\]
Making the change of variable $x = \exp(t)$ yields the conclusion of Theorem \ref{thm:Weak}:
\[ \sum_{\lam_n \leq x} a_n \sim \tfrac{g(\al)}{\al (m-1)!} x^\al (\log x)^{m-1} \qquad \text{as $x \maps \infty$.}\] 

 \end{proof}
 
We also state another immediate corollary of Theorem \ref{thm_Delange}.
 
\begin{cornum}\label{cor_thm_Delange_Stark}
    Let $a(t)$ be a non-decreasing real-valued function on $[0,\infty)$ with Laplace transform
    \[ f(s) = \int_0^\infty a(t) \exp(-st)dt.\]
Let $\al>0$ and an integer $m \geq 1$ be given, as well as a polynomial 
\[P(s)= c_{0} + c_{1}(s-\al)+ c_{2}(s-\al)^2 + \cdots +c_{m-1}(s-\al)^{m-1}\]
with $c_{0} \neq 0.$
If $f(s) - P(s)(s-\al)^{-m}$ is holomorphic for $\Re(s)>\al$ and continuous for $\Re(s) \geq \al$ then 
\[ a(t) = \left(\tfrac{c_{0}}{(m-1)!} + o(1)\right) \exp(\al t)t^{m-1} \qquad \text{as $t \maps \infty$.}\]
\end{cornum}
\begin{remark}[A result of Stark]\label{remark_Stark}
This Tauberian theorem has been used in the literature, and attributed to unpublished notes of H. Stark (which we have not seen). See a statement without proof in \cite[``Tauberian Theorem (Stark)'', \S 2.3]{DGH03}. To derive exactly the statement recorded there, set $x=\exp(t)$ in this corollary  so that in their notation, $a(t):= S(e^t) = S(x)$, and the conclusion is $S(x) \sim \tfrac{c_{0}}{(m-1)!} x^\al (\log x)^{m-1}$ as $x \maps \infty$.
\end{remark}

   \begin{remark}\label{remark_Stark_constant}
   The assumptions of \cref{cor_thm_Delange_Stark} imply $c_0 > 0$.   Since $a(t)$ is non-decreasing on $[0,\infty)$, then writing $a(t) = (a(t) - a(0)) + a(0) =: \tilde{a}(t) + a(0)$ provides a non-negative function $\tilde{a}(t)$ on $[0,\infty)$. Accordingly, we can write $f(s) = \tilde{f}(s) + u(s)$ where $\tilde{f}$ is the Laplace transform of $\tilde{a}(t)$ and $u(s)=a(0)/s$ is the Laplace transform of $a(0)$ (by Lemma \ref{lemma_bG} for $m=1$). In particular, $u(s)$ is holomorphic for $\Re(s)>0$. Combined with the hypotheses of the corollary, this implies that $\tilde{f}(s) - P(s)(s-\al)^{-m}$ is holomorphic for $\Re(s)>\al$ and continuous for $\Re(s) \geq \al$. Then \[
    c_0 = P(\alpha) = \lim_{\sigma \to \alpha^+} P(\sigma) = \lim_{\sigma \to \alpha^+}  (\sigma-\alpha)^m \tilde{f}(\sigma)  \geq 0 , 
    \]
    because the non-negativity of $\tilde{a}(t)$ implies $\tilde{f}(\sigma) \geq 0$ for all real $\sigma > \alpha$. Finally, since the corollary assumes $c_0 \neq 0$, we may conclude $c_0>0$. 
\end{remark}

\begin{proof}[Proof of Corollary \ref{cor_thm_Delange_Stark}]
Define $b(t)=t^{m-1}/(m-1)!$ and its Laplace transform $G(s) = s^{-m}$ for $\Re(s)>0$ as in Lemma \ref{lemma_bG}.
Define $h(s) = f(s) - P(s)(s-\al)^{-m}$ so that $h(s)$ is holomorphic for $\Re(s)>\al$ and continuous in $\Re(s) \geq \al$. Also define 
\[ F(s) = f(s)  - c_0G(s-\al)
= \frac{c_{1}}{(s-\al)^{m-1}} + \cdots + \frac{c_{m-1}}{(s-\al)} + h(s).\]
 Then $F(s)$ satisfies the hypotheses of Theorem \ref{thm_Delange} with constant $c_0$, and the corollary is proved. 
 \end{proof}

\subsection{Initial reductions}\label{sec_reductions} 
We prepare for the proof of Theorem \ref{thm_Delange} by making two  reductions. 

The first reduction is that it suffices to prove the theorem in the case that $a(t)$ is non-negative on $[0,\infty)$; this was observed by Delange \cite[footnote (1) p. 20]{Del54}.
For suppose that $a(t)$ is a non-decreasing real-valued function on $[0,\infty)$; then the function $\tilde{a}(t):=a(t)-a(0)$ is a   non-decreasing real-valued function that is moreover non-negative on $[0,\infty)$. Define $\tilde{f}(s)$ to be the Laplace transform of $\tilde{a}(t)$, so that $\tilde{f}(s) = f(s) - a(0)/s$ (as in Remark \ref{remark_Stark_constant}).  Suppose that $f(s), G(s)$ and $F(s) = f(s) - c_0 G(s-\al)$ satisfy the hypotheses of Theorem \ref{thm_Delange}. Then   $\tilde{F}(s):=\tilde{f}(s) - c_0G(s-\al)$ does as well, since $\tilde{F}(s) = f(s) - c_0 G(s-\al) - a(0)/s$, in which $a(0)/s$ is holomorphic for $\Re(s) \geq \al >0$.  Thus by applying the special case of the theorem to $\tilde{f}(s)$ we deduce an asymptotic  $\tilde{a}(t) \sim c_0 \exp(\al t)b(t)$ as $t \maps \infty$, and this implies the same asymptotic holds for $a(t)$ as $t \maps \infty$ (since $\al>0$ and the functions $a(t)$ and $\tilde{a}(t)$ only differ by a constant). 
 
The second reduction is that it suffices to prove the theorem in the case $c_0=1$. 
For if the hypotheses of the theorem hold for a given $c_0>0$, and $a(t)$ is non-negative and non-decreasing, let $\tilde{a}(t) = a(t)/c_0$ (which is again non-negative and non-decreasing). Then set $\tilde{f}(s)=f(s)/c_0$ to be the Laplace transform of $\tilde{a}(t)$. Under the hypothesis that $F(s):=f(s) - c_0G(s-\al)$ satisfies  the hypotheses of Theorem \ref{thm_Delange}, it follows that
 $ \tilde{F}(s):= F(s)/c_0 = \tilde{f}(s) - G(s-\al)$ does as well. 
By applying Theorem \ref{thm_Delange} to $\tilde{F}(s)$, we learn that $\tilde{a}(t) = (1+o(1))\exp(\al t) b(t)$ as $t \maps \infty$. Thus $a(t) = (c_0+o(1))\exp(\al t) b(t)$ as $t \maps \infty$,  as desired. 

From now on we will consider only the case of Theorem \ref{thm_Delange} when the given function $a(t)$ is non-negative and non-decreasing, and the constant $c_0=1$. 

\subsection{Strategy for the proof of  Theorem \ref{thm_Delange}}\label{sec_Delange_intuition}

\subsubsection{Setup}\label{sec_setup_ThmDelange}
We recall the strategy for proving Theorem \ref{thm_Delange} indicated in (\ref{A_diff}). 
Under the hypotheses of the theorem, we now write this more carefully. Define $v(t)$ in the range $t \in [0,\infty)$ as the truncated Laplace transform
\beq\label{v_dfn}
v(t) :=  \int_0^{1} u^{m-1} \exp (-tu)du,
\eeq
which  is a non-negative, decreasing function of $t$. 
  The following lemma shows that we may think of $tv(t)$ as the ``inverse'' of $b(t)$ in the limit. 
 \begin{lemnum}\label{lemma_v_limit}
For   $m \geq 1$ and $v(t)$ as defined above, 
\beq\label{v_limit} 
\lim_{t \maps \infty} t^mv(t)=\Gamma(m).
\eeq
 Consequently, $ \lim_{t \maps \infty}  t  v(t) b(t)= 1$ if and only if
\[ \lim_{t \maps \infty} \frac{b(t)}{t^{m-1}}  =  \frac{1}{\Gamma(m)}.\] 
    \end{lemnum}
    
\begin{proof}
For the first claim,
\[
\lim_{t \maps \infty} t^m v(t) = \lim_{t \maps \infty}  \int_0^{1} (tu)^{m} \exp(-tu) \frac{du}{u} =  \lim_{t\maps \infty} \int_0^{t} v^{m-1} \exp (-v) dv =   \Ga(m).\]
The second claim follows, upon considering $\lim_{t \maps \infty}tv(t)  b(t)  =\lim_{t \maps \infty}t^mv(t) \cdot b(t)/t^{m-1}$.
\end{proof}
 
Define 
\[J(t) := tv(t) \exp(-\al t) a(t).\]
Then to prove Theorem \ref{thm_Delange}, it suffices to prove
$
\lim_{t \maps \infty} J(t)=1.
$
\subsubsection{Heuristic ideas, part I}\label{sec_motivate_h}
 In this subsection, we motivate the construction of a function $H(s)$; we will justify each step later in \S \ref{sec_Delange}.
The hypothesis of the theorem provides a property of 
\[ F(s) = f(s) - G(s-\al) = \int_0^\infty \exp(-st)[a(t) - \exp(\al t)b(t)]dt.\] 
Now, being more careful than (\ref{A_diff}), we will not divide by $b(t)$ within the integrand, but instead multiply by $tv(t)$. We can do so in two steps. First, differentiate this expression once with respect to $s$ to pull down a factor of $t$:
\[ F'(s) = -\int_0^\infty \exp(-st)t[a(t) - \exp(\al t)b(t)]dt,\] 
which is valid in the region $\Re(s)>\al$ where $F(s)$ is holomorphic. Second, to introduce $v(t)$ to the integrand we integrate $F'(s+u)$ against $u^{m-1}$ over $u \in [0,1]$. Precisely, for any $\Re(s)>\al$ we are motivated to define 
\[ H(s) := - \int_0^{1} F'(s+u)u^{m-1}du
= \int_0^\infty \exp(-st) tv(t)[a(t) -\exp(\al t)b(t)]dt.\]
Here the interchange of the order of integration will be justified, since for a fixed $s$ with $\Re(s)>\al$, the integral expressing $F'(s)$ converges uniformly for $u \in [0,1]$. Now we recognize that 
\beq\label{h_J_v} 
H(s)= \int_0^\infty \exp(-(s-\al)t)[tv(t)\exp(-\al t) a(t)  -tv(t)b(t)]dt \eeq
so that in particular for $s=\al+\ep$ for a given $\ep>0$,
\[  
H(\al+\ep) = \int_0^\infty \exp(-\ep t)[J(t)  -tv(t)b(t)]dt.\]
Recalling that $tv(t) b(t) \sim 1$, we would like to extract from properties of $H(\al+\ep)$ that $J(t) -1 = o(1)$ as $t \maps \infty$.  

\subsubsection{Heuristic ideas, part II}  In this subsection, we continue to proceed heuristically; we will justify each step, and the convergence of integrals, in the rigorous proof in \S \ref{sec_Delange}. By hypothesis, $F(s)$ has   a pole of order at most $m-1$ at $s=\al$, so $H(s)$ could have a pole  at $s=\al$, and we cannot easily study $H(\al +\ep)$ as $\ep \maps 0^+$ along the real line. Instead,   Delange's strategy uses three ideas. First: avoid the pole by taking limits along horizontal lines strictly above and below the real line. That is, show that for each real $y \neq 0$, $\lim_{\ep \maps 0^+}H(\al + \ep +iy)$ is well-defined, and in fact defines a function $H_0(y)$ that is continuous for $y \neq 0$ and integrable on $[-L,L]$. Second: dilute the effect of the pole by averaging along vertical lines. That is, study the average
\[ \int_{-L}^L H(\al + \ep +iy) W(y) dy,\] 
for a continuous kernel function $W$  that is compactly supported in some interval centered at the origin, which we have depicted here by $(-L,L)$. 
Third: use the Dominated Convergence Theorem to pass the limit $\ep \maps 0^+$ inside this average. This is the essential way Delange's method ``works around'' the pole at $s=\al$, without assuming we can work in any region to the left of $\Re(s)=\al$. This step requires showing that there is an integrable function in $y$ that dominates $H(\al+\ep+iy)$, uniformly for all $\ep>0$. Once all of these steps are rigorously carried out, the conclusion is that 
\[ \lim_{\ep \maps 0^+} \int_{-L}^L H(\al + \ep +iy) W(y) dy
 = \int_{-L}^L \lim_{\ep \maps 0^+} H(\al + \ep +iy)W(y)dy = \int_{-L}^L H_0(y) W(y)dy. \] 
Plugging in the definition (\ref{h_J_v}) for $H(\al + \ep +iy)$ and rearranging, we will interpret this as the statement that 
 \begin{multline*}  \lim_{\ep \maps 0^+} \int_{-L}^L \int_0^\infty \exp( -(\ep+iy)t)J(t) dt W(y)dy
\\  = \int_{-L}^{L}H_0(y)W(y)dy +  \lim_{\ep \maps 0^+} \int_{-L}^L \int_0^\infty \exp(-(\ep+iy)t) tv(t)b(t)dt W(y)dy.\end{multline*}
As in \S \ref{sec_Fourier_transform}, denote the Fourier transform $\Fscr(W)(t) = \int_{-L}^L W(y)\exp(-ity)dy$.
Suppose for the purpose of illustration that we may freely interchange the order of integration in the identity above; then this reveals
 \begin{multline*} 
 \lim_{\ep \maps 0^+}  \int_0^\infty \Fscr(W)(t) \exp( -\ep t)J(t) dt  
\\  = \int_{-L}^{L}H_0(y)W(y)dy +  \lim_{\ep \maps 0^+}  \int_0^\infty \Fscr(W)(t) \exp(-\ep t) tv(t)b(t)dt .\end{multline*}
We define $ \int_0^\infty \Fscr(W)(t) J(t) dt$ to be the expression on the left-hand side.
If we further assume that $W$ is defined appropriately so that $\Fscr(W)$ is an integrable function, then using that $tv(t)b(t) \sim 1$ as $t \maps \infty$ (so in particular it is a bounded function) we can apply the Dominated Convergence Theorem in the last term, replacing this term by $\int_0^\infty \Fscr(W)(t)  tv(t)b(t)dt$. At this point the argument yields:
\[  \int_0^\infty \Fscr(W)(t) J(t) dt
 =\int_{-L}^{L}H_0(y)W(y)dy+\int_0^\infty \Fscr(W)(t)  tv(t)b(t)dt.\]

We still need to conclude that $J(t) \sim 1$ as $t \maps \infty$. Let us exploit the remaining hypothesis that $tv(t)b(t) \sim 1$. One idea is to ``push $t$ to infinity'' at every point in the region of integration on the last term on the right-hand side: for example, replace $t=u+\xi$ for a fixed real $\xi$ that we will later let tend to infinity.
To carry this out, it is convenient to denote $\Fscr(W)(t)=K(t-\xi)$ for some function $K$ that we assume is integrable.
Then, for this fixed $\xi$,
\[ \int_0^\infty \Fscr(W)(t)  tv(t)b(t)dt
 = \int_{0}^\infty K(t-\xi) tv(t)b(t)dt = 
 \int_{-\infty}^\infty K(u) \om(u,\xi)du,\] 
 in which we simply define $u=t-\xi$ and $\om(u,\xi) = (u+\xi)v(u+\xi)b(u+\xi)$ for each $u \geq -\xi$, and $\om(u,\xi)=0$ for $u<-\xi$. By hypothesis, as $\xi \maps \infty$, $\om(u,\xi) \sim 1$ for each fixed $u$. Since $K(u)$ is an integrable function,  the Dominated Convergence Theorem shows
 \[ \lim_{\xi \maps \infty} \int_{-\infty}^\infty K(u) \om(u,\xi)du =   \int_{-\infty}^\infty K(u) du.\]  
On the other hand, according to the definition $\Fscr(W)(t)=K(t-\xi)$   we can take the Fourier inverse as in \S \ref{sec_Fourier_transform} to compute that  \[ W(y) = \Fscr^{-1}(K(\cdot - \xi))(y) = \Fscr^{-1}(K)(y) \exp(i\xi y).\]
Taking stock of every property we have assumed so far, this line of argument leads to the conclusion  that if $K(y)$ is an integrable function such that $\Fscr^{-1}(K)$ is a continuous function of compact support in $(-L,L)$, for the continuous function $H_0$ we have:
\[ \lim_{\xi \maps \infty} \int_{-\infty}^\infty K(t-\xi)J(t)dt =  \lim_{\xi \maps \infty} \int_{-L}^L H_0(y) \Fscr^{-1}(K)(y) \exp(i\xi y) dy + \int_{-\infty}^\infty K(u)du.\] 
Since $H_0$ is continuous and integrable, the Riemann-Lebesgue lemma shows that the first term on the right-hand side vanishes. 
 (In some sense, this application of the Riemann-Lebesgue lemma can be seen as a reason that this approach does not yield lower-order terms in the asymptotic.)
If we add one more condition on the kernel $K$, that $\int K=1$, this line of argument leads to the conclusion
\beq\label{KJ_explain}
\lim_{\xi \maps \infty} \int_{-\infty}^\infty K(t-\xi)J(t)dt =  1.\eeq
This additional condition $\int K=1$ is in fact compatible with the well-known theory of approximations to the identity, which are  needed for the final step: to  pick out the value of $J(t)$ (for large $t$) from the average (\ref{KJ_explain}).  

\subsubsection{Approximations to the identity} 
An approximation to the identity  on $\R$ is a family of kernel functions $K_\lam$   defined according to a real parameter $\lam>0$, with the following three properties:
 \begin{enumerate}[(i)]
\item $\int_\R K_\lam(x)dx=1$ for all $\lam>0$.
\item $\int_\R |K_\lam(x)| dx \ll 1$ uniformly in $\lam>0$.
\item For every $R>0$, $\int_{|x| \geq R}|K_\lam(x)|dx \maps 0$ as $\lam \maps \infty$.
\end{enumerate} 
A classical property is that for an approximation to the identity $K_\lam$, for each integrable function $f$, 
\beq\label{approx_id_app}
\lim_{\lam \maps \infty} \int_{-\infty}^\infty K_\lam(x-t)f(t)dt = f(x), \qquad \text{for almost every $x$;}
\eeq
see \cite[Ch. 3 Thm. 2.1]{SSReal}. (In the additive convolution notation of \S \ref{sec_additive_convolution}, this is the statement $\lim_{\lam \maps \infty} (K_\lam \star f)(x) = f(x)$ for a.e. $x$.) The idea is that as $\lam \maps \infty$,  the average becomes more concentrated or peaked at $t=x$, so that  the average  better approximates the   value of $f$ at the point $x$. Let us pause to observe that there are many approximations to the identity; for example:

\begin{lemnum}[Approximation to the identity]\label{lemma_approx_id}

 (I)  Let $K(\xi) =    \frac{1}{\pi} \tfrac{ \sin^2 \xi}{  \xi^2}$ denote the normalized Fej\'{e}r kernel, and define $K_\lam(\xi) := \lam K(\lam \xi)$. Then $K_\lam$ is an  approximation to the identity.
 (II) Let  $K$ be a Schwartz function with $\int K(\xi)d\xi=1$, and define $K_\lam(\xi) := \lam K(\lam \xi)$. Then $K_\lam$ is an approximation to the identity.
     \end{lemnum}
     \begin{proof}
The arguments run in parallel. To verify  property (i), rescaling shows that  $\int K_\lam(\xi ) d\xi = \int \lam K( \lam \xi) d\xi = \int K(\xi) d\xi $. This evaluates to 1 by applying (\ref{Fejer_integral}) for the Fej\'er kernel, or by the hypothesis, in  the Schwartz function case.  For (ii), the Fej\'er kernel is non-negative, so (i) implies (ii). For the Schwartz function, it is integrable so that  $\int |K_\lam (\xi)|d\xi = \int |K(\xi)|d\xi$ is bounded. For (iii),  for every $R>0$,  
  \[ \int_{|\xi| \geq R} |K_\lam(\xi)| d\xi =  \int_{|u| \geq \lam R}|K(u)| du.\]
  For the normalized Fej\'er kernel or the Schwartz function, $K$ is an integrable function so that this tail vanishes as $\lam \maps \infty$.
     \end{proof}

\subsubsection{Heuristic ideas, part III} 
 In this subsection, we continue with the heuristic sketch of the main strategy.
We return to   (\ref{KJ_explain}), in which our goal is to pick out the value of $J(t)$ for large $t$.
We cannot simply apply the result (\ref{approx_id_app}) for an approximation to the identity, since $J(t)$ need not be integrable. Nevertheless, we will apply the identity (\ref{KJ_explain}) with an approximation to the identity $K_\lam$ in place of $K$,   so that $K_\lam(t-\xi)$ becomes more concentrated or peaked at $t=\xi$ as $\lam \maps \infty$, and (\ref{KJ_explain}) should eventually reveal that $J(\xi) \sim 1$ as $\xi \maps \infty$. We will impose that $K_\lam$ is even, so that $K_\lam(t-\xi) = K_\lam(\xi-t)$. We will also require that $K_\lam$ is non-negative, so that we can later apply   the monotonicity of $a(t)$ and $v(t)$ to extract an asymptotic for $J(t)$ from (\ref{KJ_explain}) by hand. Recall that earlier in the argument we have also required that $\Fscr^{-1}(K_\lam)$ should be continuous of compact support. Fortunately,  approximations to the identity with all these properties do exist.

\subsubsection{Allowable approximations to the identity} 
We define that an approximation to the identity $K_\lam$ is \emph{allowable} if 
 \beq\label{allowable_dfn}
 K_\lam(\xi) = \lam K(\lam \xi)
 \eeq 
 for an even, non-negative real-valued function $K$, such that  $\Fscr^{-1}(K)(x)$ is a continuous function of compact support. Note that by rescaling property (i) of an approximation to the identity, $\int K(\xi) d\xi=1$. 
 Let us establish once and for all the notation that for an allowable approximation to the identity, 
 \[ k(x):= \Fscr^{-1}(K)(x), \qquad  k^\lam(x) : =\Fscr^{-1}(K_\lam)(x)= k(x/\lam).\]
 In particular, if $k(x)$ is supported in $[-\del_0,\del_0]$ then $k^\lam(x)$ is supported in $[-\del_0\lam,\del_0\lam]$.

  To prove Theorem \ref{thm_Delange}, we may work with any allowable approximation to the identity $K_\lam$, and we provide a large class of  options:
\begin{lemnum}[Allowable approximation to the identity]\label{lemma_allowable_approx_id}

(I) Let $K(\xi)= \frac{1}{\pi} \tfrac{ \sin^2 \xi }{\xi^2}$ denote the normalized Fej\'{e}r kernel, and define  $K_\lam (\xi)=\lam K(\lam \xi)$. Then $K_\lam$  is an allowable approximation to the identity, with $k^\lam$ supported in $[-2\lam,2\lam]$.

(II) Given any $\del_0>0$, there is an allowable approximation to the identity $K_\lam$ such that $K$ is a Schwartz function, and $k^\lam$ is a Schwartz function supported in $[-\del_0 \lam,\del_0 \lam]$. 
\end{lemnum}

\begin{proof}
(I) By Lemma \ref{lemma_approx_id}, $K_\lam$ is an approximation to the identity, and visibly it is an even, non-negative function.   By the function pair recorded in (\ref{Fejer_dfn}), $k^\lam(x) = \tfrac{1}{2}(1 - \tfrac{|x|}{2\lam})\mathbf{1}_{[-2\lam,2\lam]}(x)$, which is a continuous function of compact support in $[-2\lam,2\lam]$.

(II) Given any $\del_0>0$, and any even, non-negative Schwartz function $\eta$ compactly supported in $[-\del_0/2,\del_0/2]$,   Lemma \ref{lemma_build_kernel} provides a compactly supported, even, non-negative Schwartz function $k$ supported in $[-\del_0,\del_0]$ whose Fourier transform $K= \Fscr(k)$ is an even, non-negative Schwartz function with $\int K=1$. By Lemma \ref{lemma_approx_id}, it follows that $K_\lam(\xi) = \lam K(\lam \xi)$ is an approximation to the identity; by  construction $K_\lam$ is moreover allowable.
\end{proof}

\section{Hypothesis \ref{hyp:Weak}:  Proof of Theorem \ref{thm_Delange}}\label{sec_Delange}

We now turn to a rigorous proof of Theorem \ref{thm_Delange}. We will follow the strategy of Delange \cite[Thm. 1]{Del54} as presented in Narkiewicz \cite{Nar83}, except where they employ the Fej\'er kernel, we will employ a fixed allowable approximation to the identity, as defined in (\ref{allowable_dfn}) and provided by  Lemma \ref{lemma_allowable_approx_id}.

\subsection{The main asymptotic via an allowable approximation to the identity}\label{sec_main_asymptotic_weak}
Recall from the reductions in \S \ref{sec_reductions} that we need only prove Theorem \ref{thm_Delange} when the given function $a(t)$ is non-negative and non-decreasing, and the constant $c_0=1$. 
Recall from the setup in 
\S \ref{sec_setup_ThmDelange} the definition of the truncated Laplace transform  $v(t)$  in (\ref{v_dfn}), and the definition \beq\label{J_dfn} 
J(t):=tv(t)\exp(-\al t)a(t).
\eeq
To prove Theorem \ref{thm_Delange}, it suffices to prove
\beq\label{J_limit}
\lim_{t \maps \infty} J(t)=1.
\eeq
We now prove this rigorously, assuming the following result, which we verify later in this section.
 \begin{propnum}\label{prop_JK_integral}
 Let $K_\lam$ be an allowable approximation to the identity, as defined in (\ref{allowable_dfn}).
Fix $\lam>0$.   

 (I) For every real $\xi$,   
\[\int_0^\infty K_\lam (t-\xi) J(t)dt = \int_0^\infty K_\lam (\xi-t) J(t)dt\]
is a convergent integral.

(II) Moreover, 
     \[ \lim_{\xi \maps \infty} \int_0^\infty K_\lam (t-\xi) J(t)dt =1.\] 
\end{propnum}

 Note that in (I), the two expressions are equivalent since $K_\lam$ is an even function; thus the task is to prove convergence of the integral.
\subsubsection{Proof of Theorem \ref{thm_Delange}: Upper bound for the lim sup}
 Assuming Proposition \ref{prop_JK_integral}, we begin by showing  $\limsup_{t \maps \infty} J(t) \leq 1.$ Since $K_\lam$ is an allowable approximation to the identity, $K_\lam(x) = \lam K(\lam x)$ where $K$ is an even, non-negative, integrable function with $\int K(u)du=1$.
First, $J(t) \geq 0$ for all $t \geq 0$, since $a(t)$ is assumed to be non-negative and $v(t)$ is non-negative by definition. 
Fix a real number $h>0$ and consider any $\xi>h$. By the non-negativity of $J(t)$ and $K_\lam$, 
\[ \int_0^\infty  K_\lam (t-\xi) J(t) dt \geq \int_{\xi - h}^{\xi+h} K_\lam (t-\xi)J(t)dt.\]
Within the interval of integration (which is contained in $(0,\infty)$), since $v(t)$ is decreasing and $a(t)$ is non-decreasing,
\[ J(t) \geq (\xi - h) v(\xi+h) \exp( -(\xi+h)\al) a(\xi-h)
 = J(\xi-h) \exp(-2\al h) \frac{v(\xi+h)}{v(\xi-h)}.\]
Plugging this into the integral on the right, we obtain the inequality 
\begin{align*} \int_0^\infty  K_\lam (t-\xi) J(t) dt 
& \geq  J(\xi-h)
\exp(-2\al h)\frac{v(\xi+h)}{v(\xi-h)} \int_{\xi-h}^{\xi+h} K_\lam(t-\xi)dt \\
& = J(\xi-h)
\exp(-2\al h)\frac{v(\xi+h)}{v(\xi-h)} \int_{-\lam h}^{\lam h} K(u)du.
\end{align*}
 Rearranging shows 
 \[ J(\xi-h) \leq \exp(2\al h) \frac{v(\xi-h)}{v(\xi+h)} (A)^{-1}(B),\]
 say, in which 
\[
     (A) : = \int_{-\lam h}^{\lam h} K(u)du,\qquad
     (B) : = \int_0^\infty  K_\lam (t-\xi) J(t) dt .
\]
By (\ref{v_limit}), 
 \beq\label{v_limit_ratio} 
 \lim_{\xi \maps \infty} \frac{v(\xi-h)}{v(\xi+h)} = \lim_{\xi \maps \infty} \frac{(\xi+h)^m}{(\xi-h)^m} =1.
 \eeq
 Taking $\xi \maps \infty$ (and applying Proposition \ref{prop_JK_integral} to the limit of term (B)), and 
then taking $\lam \maps \infty$ (and applying $\int K=1$ to the limit of term (A)), shows that 
\[ \limsup_{\xi \maps \infty} J(\xi -h) \leq \exp (2\al h) .\]
By taking $h>0$ arbitrarily small, 
we may conclude that $\limsup_{t \maps \infty} J(t) \leq 1$.  For future reference, we note that since $J(t)$ is locally bounded, this implies that $J(t)=O(1)$ for all $t\geq 0$.

\subsubsection{Proof of Theorem \ref{thm_Delange}: Lower bound for the lim inf}
To complete the proof of the theorem, we must verify that $\liminf_{t \maps \infty} J(t) \geq 1$. 
Now the idea is to fix an $h>0$ and $\xi>h$ and study $J(\xi +h)$. By the previous step, we have shown that $J(t)$ is uniformly bounded above by $O(1)$ for all $t \geq 0$; let us accordingly define the finite positive real number
\[ M = \sup \{ J(t) : t \geq 0\}.\]
Then using the non-negativity of $K_\lam$ and $J(t)$,
\beq\label{KMJ}
\int_0^\infty  K_\lam (t - \xi) J(t) dt 
 \leq M \int_{|t-\xi| > h} K_\lam (t-\xi) dt  + \int_{\xi-h}^{\xi+h} K_\lam (t-\xi) J(t) dt =: \mathrm{I} + \mathrm{II},
 \eeq
 say.
 The first term may be bounded above (independent of $\xi$) by 
 \[ \mathrm{I} \leq 2M \int_{\lam h}^\infty K(u) du.\]
  By the non-negativity and integrability of $K(u)$, as $\lam \maps \infty$ this contribution  vanishes (as the tail of a convergent integral), so that $\lim_{\lam \maps \infty} \mathrm{I} =0.$
  
Next, within  the interval of integration in the term $\mathrm{II}$ on the right-hand side of (\ref{KMJ}), since $v(t)$ is a decreasing function and $a(t)$ is non-decreasing, 
 \[ J(t) \leq (\xi+h) v(\xi-h) \exp(-(\xi-h)\al) a(\xi+h) 
     = J(\xi+h) \exp (2\al h) \frac{v(\xi-h)}{v(\xi+h)}.\]
Thus 
\[ \mathrm{II} \leq J(\xi+h) \exp (2\al h) \frac{v(\xi-h)}{v(\xi+h)} \int_{-\lam h}^{\lam h} K(u)  du.\]
Applying these upper bounds in (\ref{KMJ}) and rearranging shows that 
\[ J(\xi+h) \geq \exp (-2\al h) \frac{v(\xi+h)}{v(\xi-h)} (A)^{-1} (B') ,\]
say, 
in which $(A)$ is defined exactly as before, and now
\[
    (B') := \int_0^\infty  K_\lam (t - \xi) J(t) dt - \mathrm{I} \geq   \int_0^\infty  K_\lam (t - \xi) J(t) dt- 2M \int_{\lam h}^\infty K(u) du.
\]
Now taking $\xi \maps \infty$ (and applying Proposition \ref{prop_JK_integral} and the previous observation (\ref {v_limit_ratio})), followed by taking $\lam \maps \infty$ (and applying $\int K=1$ in (A) and the vanishing of the term $\mathrm{I}$ in (B')), shows that 
\[ \liminf_{\xi \maps \infty} J(\xi+h) \geq \exp(-2\al h).\]
By taking $h>0$ arbitrarily small, 
we may conclude that $\liminf_{t \maps \infty} J(t) \geq 1$.
 This completes the proof of Theorem \ref{thm_Delange}, pending the verification of Proposition \ref{prop_JK_integral}.
 
\subsection{Preliminaries to prove Proposition \ref{prop_JK_integral}: the function $H(s)$}\label{sec_K_integral_proof}
Let $K_\lam$ be an allowable approximation to the identity, as defined in (\ref{allowable_dfn}). In the proof of Proposition \ref{prop_JK_integral} (I), it is marginally more convenient  to work with the second expression (which we recall is equivalent to the first expression since $K_\lam$ is an even function).  The method of proof   is the same for any allowable approximation of the identity $K_\lam$, but for notational convenience we suppose $k(x)=\Fscr^{-1}(K)(x)$ is supported in $[-2,2]$ so that 
\[k^\lam = \Fscr^{-1}(K_\lam)\] 
is supported in $[-2\lam,2\lam]$. (That is, if we represent the compact support of $\Fscr^{-1}(K)(x)$ by $[-\del_0,\del_0]$, we assume for notational simplicity that $\del_0=2$. Indeed, the Fej\'{e}r kernel, or any choice from Lemma \ref{lemma_allowable_approx_id} with $\del_0=2$, satisfies this; for any other allowable approximation to the identity with $\Fscr^{-1}(K)(x)$ by $[-\del_0,\del_0]$ one may simply replace every occurrence of $2\lam$ throughout \S \ref{sec_Delange} by $\del_0\lam$.)

Recall   the hypotheses of Theorem \ref{thm_Delange}, for the function $F(s)$ defined for $\Re(s)>\al$ (and $c_0=1$) by
 \[
 F(s) =  f(s) - G(s-\al)= \int_0^\infty  \exp(-st) [a(t)  - \exp(\al t) b(t)] dt.
\]
For $\Re(s)>\al$, by differentiating this expression (which by hypothesis is holomorphic in this region),
 \beq\label{Fprime}
 F'(s) = - \int_0^\infty \exp(-st) t[a(t) - \exp(\al t)b(t)] dt.
 \eeq
 We now define 
 \beq\label{dfn_h}
 H(s)  :=   -\int_0^{1} F'(s+u)u^{m-1} du.
 \eeq
\begin{lemnum}\label{lemma_h_initial_ppties}
The function $H(s)$ is holomorphic for $\Re(s)>\al$, and in this region 
 \beq\label{h_second_integral_dfn}
 H(s)  = \int_0^\infty \exp(-st) t v(t)[a(t)  - \exp(\al t) b(t) ]dt.
 \eeq
 For every real $L>0$ and fixed $\ep>0$, for any continuous function $\phi(y)$ on $[-L,L]$,
\[\int_{-L}^{L} H(\al + \ep + iy) \phi(y) dy\] 
is a convergent integral.
In particular, for every real $\lam>0$ and fixed $\ep>0$, 
\[\int_{-2\lam}^{2\lam} H(\al + \ep + iy) k^\lam(y) \exp (i\xi y) dy=\int_0^\infty  \exp(-(\al+\ep)t)K_\lam(\xi-t) t v(t) [a(t) - \exp(\al t)b(t)]dt,\]
and both integrals are convergent.
\end{lemnum}
\begin{proof}
 To observe that $H(s)$ is holomorphic for $\Re(s)>\al$, 
 recall that by hypothesis $F(s)$ is holomorphic for $\Re(s) > \al$ and continuous on $\Re(s) \geq \al$, aside from a pole at $s=\al$ of order at most $m-1$. Let $\Omega:=\{s: \Re(s)>\al\}$. Then  $F'(s+u)u^{m-1}$ is defined for $(s,u) \in \Omega \times [0,1]$ and is  holomorphic in $s \in \Omega$ for each  $u \in [0,1],$ and is jointly continuous for $(s,u) \in \Omega \times [0,1]$. Thus $H(s)$ is holomorphic on $\Omega$ \cite[Ch. 2, Thm. 5.4]{SSComp}.
 After substituting the  expression (\ref{Fprime}) for $F'$ and rearranging, for $\Re(s)>\al$, 
\[
 H(s)  = \int_0^\infty \exp(-st) t v(t)[a(t)  - \exp(\al t) b(t) ]dt,
\]
with $v(t)$ defined as before in (\ref{v_dfn}).
For clarity, the interchange of the order of integration is justified, since for a fixed $s$ with $\Re(s)> \al$,   the integral expressing $F'(s+u)$ converges, uniformly in $u\in [0,1]$. This has verified (\ref{h_second_integral_dfn}).

For any fixed real $L>0$ and $\ep>0$, the integral expression (\ref{h_second_integral_dfn}) for $H(s)$ is   convergent, uniformly for $s$ in the horizontal half-strip $\Re(s)\geq \al+\ep$, $\Im(s) \in [-L,L].$
Hence for each fixed $\ep>0$, integrating $H(\al+\ep +iy)$ against a continuous function $\phi(y)$ on the compact interval $y \in [-L,L]$ yields a  convergent integral.

In particular, for each fixed $\ep>0$, integrating $H(\al+\ep +iy)$ against the continuous function $k^\lam(y)\exp(i\xi y)$ on the compact interval $y \in [-2\lam,2\lam]$ yields an absolutely convergent double integral
\[
\int_{-2\lam}^{2\lam} H(\al + \ep + iy) k^\lam(y) \exp (i\xi y) dy = \int_{-2\lambda}^{2\lambda} \int_0^{\infty} (\cdots) \, dt\, dy
\]
upon inserting the integral expression (\ref{h_second_integral_dfn}). We may therefore    interchange  the order of integration, so the integral on the left is identical to the (convergent) integral
\[\int_0^\infty  \exp(-(\al+\ep)t)K_\lam(\xi-t) t v(t) [a(t) - \exp(\al t)b(t)]dt.\]
\end{proof}

 To prepare to take the limit as $\ep \maps 0^+$ in the last conclusion of the above lemma, we isolate the key convergence properties of $H$.
\begin{lemnum}\label{lemma_h_convergence}
 \begin{enumerate}[(i)]
     \item\label{h_interval} 
   For any closed finite-length interval $I$ on the real line with $0 \not\in I$, the limit of $H(\al+\ep +iy)$ as $\ep \maps 0^+$ converges uniformly for $y \in I$ to a continuous function on $I$.
  
 \item\label{h_H}      For any $\lam>0$, there exists an integrable (non-negative) function $\mathcal{H}(y)$ on $[-2\lam,2\lam]$ such that uniformly for all $0<\ep<1$, 
      \[|H(\al+\ep +iy)| \leq \mathcal{H}(y).\]

\item\label{g_exists} For real $y \neq 0$, the limit 
 \[  \lim_{\ep \maps 0^+} H(\al + \ep +iy) =: H_0(y)\]
 exists. The function $H_0(y)$ is integrable over $[-2\lam,2\lam]$ for any real $\lam>0$, and $H_0(y)$ is a continuous function  for $y \neq 0$. 
 
\item\label{h_int} For each fixed $\lam>0$, and for any  continuous function $\phi(y)$ on $[-2\lam,2\lam]$,
  \[ \lim_{\ep \maps 0^+} \int_{-2\lam}^{2\lam} H(\al + \ep + iy) \phi(y) dy  =\int_{-2\lam}^{2\lam}\lim_{\ep \maps 0^+}  H(\al + \ep + iy) \phi(y) dy = \int_{-2\lam}^{2\lam} H_0(y) \phi(y)dy.\]
  \end{enumerate}
 \end{lemnum}
  \begin{remark}\label{remark_Narkiewicz}
 In Narkiewicz \cite[Ch. III Lemma 3.3]{Nar83},  the text appears to check that $H(\al+\ep+iy)$ (called $h(\al+\ep+iy)$ in his notation) is integrable in $\ep$ with unspecified behavior in $y$. However, it appears that what is required is that $H(\al+\ep+iy)$ is dominated, uniformly in $\ep$, by a function that is integrable in $y$, to enable an application of the Dominated Convergence Theorem in \cite[Ch. III Cor. 1]{Nar83}. (This defect does not seem to be apparent in Delange's original method, which is quite involved (see \cite[\S2.2, Lemme 2]{Del54}).) In this lemma, we present a different proof that verifies integrability in $y$. 
 \end{remark}

Because $F(s)$ has a pole of at most order $m-1$ at $s=\al$, the discussion  naturally breaks into two cases: $m=1$ and $m \geq 2$.
It is useful to note that the hypothesis on $F(s)$ in Theorem \ref{thm_Delange} yields a corresponding quantitative upper bound we will apply now: for any $R >0$, there exists a constant $C_R>0$ such that 
\beq\label{F_pole_quantitative}
|F(\al +s)| \leq C_R|s|^{-(m-1)} \qquad \text{for all $|s| \leq R$ with $\Re(s)>0$.}
\eeq

 \subsubsection{Proof of Lemma \ref{lemma_h_convergence} \ref{h_interval}}

 If $m=1$, then  by  the definition (\ref{dfn_h}), for $\Re(s)>\al,$
 \beq\label{h_F_difference}
 H(s) = - \int_{0}^{1} F'(s+u)du =  F(s)-F(s+1).
 \eeq
Consequently for any $\ep>0$ and any $y \in \R$,
\[
   H(\al+\ep+iy) =  F(\al+\ep+iy)-F(\al +(\ep +1)+iy).
   \]
 In the case $m=1$, $F(s)$ is holomorphic in $\{s: \Re(s)>\al\}$ and continuous in $\{s : \Re(s) \geq \al\}$ (since $F(s)$ has no pole at $s=\al$). For any   closed finite-length interval $I$, the function $(\epsilon,y) \mapsto H(\al+\ep +iy)$ is consequently continuous as a function of $(\epsilon,y) \in [0,1] \times I$. (This is true even if $I$ contains the origin, again since $F(s)$ has no pole at $s=\al$.) Since $[0,1] \times I$ is compact, the convergence of the limit
 \[
 H_0(y) := H(\alpha+iy) = \lim_{\epsilon \to 0^+} h(\alpha+\epsilon+iy) \quad \text{ for } y \in I
 \]
 is uniform, and the limit function is continuous as a function of $y \in I$.

Next suppose that $m \geq 2$. We suppose $I$ is a closed finite-length interval not containing the origin, and again consider
\[
 H(s) =- \int_0^{1} F'(s+u) u^{m-1} du .
\]
Integration by parts shows that for a fixed $s$ with $\Re(s)>\al$,
  \[ H(s) =  -F(s+1) +(m-1)\int_0^{1} F(s+u)u^{m-2}du.\]
 Thus in particular for $y \in I$,
  \beq\label{h_two_terms}
  H(\al +\ep +iy) =  -F(\al+(\ep+1)+iy) +(m-1)\int_0^{1} F(\al+\ep+iy+u)u^{m-2}du.
  \eeq
  For the first term, since $F(s)$ is holomorphic in $\Re(s) \geq \alpha + 1$, the function $(\epsilon,y) \mapsto F(\alpha+ (\epsilon + 1) + iy) $ is continuous as a function of $(\epsilon,y) \in [0,1] \times I$. Since $[0,1]\times I$ is compact, the convergence of the limit $ F(\alpha+1+iy) = \lim_{\ep \maps 0^+}F(\al + (\ep + 1)+iy)$  is uniform, and the limit function  is continuous as a function of $y \in I$.

  For the second term in (\ref{h_two_terms}), we will apply the Dominated Convergence Theorem. We need to show that  an integrable function in $u$ dominates the integrand, uniformly for all $0< \ep \leq 1$. Under the hypotheses of Theorem \ref{thm_Delange} we may apply (\ref{F_pole_quantitative}) for $R\ll \max\{1,\max_{y\in I}|y|\}$,  for all  $0<\ep \leq 1$, $y \in I$ and $u \in [0,1]$,
\[
  |F(\al + \ep +iy +u) u^{m-2}| \ll_I \frac{u^{m-2}}{|\ep+iy+u|^{m-1}} 
  \leq \frac{u^{m-2}}{u^{m-2}|y|} \leq \frac{1}{|y| }  ,
\]
since $|\ep + iy+u| \geq u$ and $|\ep + iy+u| \geq |y|$. 
The upper bound  $|y|^{-1}$, which holds uniformly for all $0 \leq \ep \leq 1$, is an integrable function of $u$ (indeed, a constant function with respect to $u$) on the interval $[0,1]$.  
  In conclusion, we may take the limit as $\ep \maps 0^+$ inside the integral in (\ref{h_two_terms}), obtaining
  \[ \lim_{\ep\maps 0^+} \int_0^{1} F(\al+\ep+iy+u)u^{m-2}du
   = \int_0^{1} F(\al+iy+u)u^{m-2}du. \]
  The right-hand side integral is of the form  $\tilde{G}(y):=\int_0^{1}G(u,y)du$ for 
   $G(u,y):=F(\al+iy+u)u^{m-2}$. By hypothesis, the function $F(s+u)$    is holomorphic in $s \in \Omega:=\{\Re(s) > \al, \Im(s) \in I\}$ for each fixed $u \in [0,1]$, and continuous on the closure $\overline{\Omega} \times [0,1]$. Thus $G(u,y)$ is uniformly continuous for $(u,y)$ in the compact set $[0,1]\times I$. Thus as the integral of a jointly continuous function over a compact interval, the function $\tilde{G}(y)$ is a continuous function for $y \in I$.
 This completes the proof that as $\ep \maps 0^+$, the second term in (\ref{h_two_terms}) has a limit that is a continuous function of $y \in I$.
 This completes the proof of   Lemma \ref{lemma_h_convergence}\ref{h_interval}, for $m \geq 2$. 

 \subsubsection{Proof of Lemma \ref{lemma_h_convergence} \ref{h_H}}

For the second claim of the lemma, we must confirm that there is a function $\mathcal{H}(y)$, integrable on $[-2\lam,2\lam]$ such that  for each $0<\ep<1$, $|H(\al +\ep +iy)|\ll_{\lambda} \mathcal{H}(y)$ for all $y \in [-2\lam,2\lam]\setminus \{0\}$.  The proof again naturally divides into two cases. In the case $m=1$,  recall from the proof of Lemma \ref{lemma_h_convergence} \ref{h_interval} that the map $(\epsilon,y) \mapsto H(\al+\ep +iy)$ is continuous as a function of $(\epsilon,y) \in [0,1] \times [-2\lam,2\lam].$ It follows that
\[
     |H(\al + \ep+ iy)|  \ll_\lam 1 \qquad \text{for all $y \in [-2\lam,2\lam]$ and  $0\leq \ep\leq 1$.}
\]
Thus for each $\ep>0$, $|H(\al + \ep+ iy)| \ll_{\lambda} \mathcal{H}(y) :=1,$ which is integrable in $y$ over $[-2\lam,2\lam]$.  This completes the proof   for $m=1$.

We next consider the case $m \geq 2$. In the first term on the right-hand side in (\ref{h_two_terms}), for all $y \in [-2\lam,2\lam]$ the argument $s=\al+(\ep+1)+iy$ of $F$ lies in the region of holomorphicity of $F(s)$. Thus this term  is bounded by $O_{\lambda}(1)$ for all $y \in [-2\lam,2\lam]$, uniformly in $0 \leq \ep \leq 1$. This suffices to check integrability in $y$ (over $[-2\lam, 2\lam]$) of this term.  We next focus on the second term in (\ref{h_two_terms}).
By applying (\ref{F_pole_quantitative}) with $R \ll \max\{1,\lam\}$, we have uniformly for $y \in [-2\lambda,2\lambda]$ and $0 < \epsilon \leq 1$, 
\[
(m-1)\int_0^{1} F(\al+\ep+iy+u)u^{m-2}du \ll_{\lambda}\int_0^{1} \frac{u^{m-2}}{|\epsilon+u+iy|^{m-1}} du \leq \int_0^{1} \frac{u^{m-2}}{|u+iy|^{m-1}} du. 
\]
There is now no dependence on $\epsilon$.
We make two elementary observations: first, $0 \leq u \leq |u+iy|$, and second, $|u+iy|^{-1} \leq \sqrt{2}(u+|y|)^{-1}$. (The second follows from $(A+B)^{2}\leq 2 (A^2+B^2)$ for any real $A,B$.) Applying these in sequence shows the above integral is
\beq\label{step_show_H_integrable}
\leq \int_0^{1} \frac{1}{|u+iy|}du \leq \sqrt{2}\int_0^{1} \frac{1}{u+|y|} du = \sqrt{2}( \log (1+|y|)-\log|y|) =: \mathcal{H}(y), 
\eeq
say, for $y \neq 0$. Also define $\mathcal{H}(0) = 1$, say. Note that $\mathcal{H}(y)$ is a non-negative function.
The final claim is that the even function $\mathcal{H}(y)$ is integrable on $[0,2\lambda]$ and hence on $[-2\lambda,2\lambda]$. Indeed, the first term in the definition of $\mathcal{H}(y)$ is integrable on $[0,2\lam]$. 
The second term is integrable on $[0,2\lam]$ as well; the only subtlety to check is integrability for $y$ near the origin, so it suffices to verify integrability on $[0,1]$, say.  Defining $f_n(y) = |\log y| \mathbf{1}_{[1/n,1]}(y)$ gives a sequence of integrable functions that converge pointwise to $-\log y$ on $(0,1]$ and with $0 \leq f_1 \leq f_2 \leq \cdots$, so that by the monotone convergence theorem,
\[ - \int_0^{1} \log y dy = \lim_{n \maps \infty} \int_0^{1} |\log y |\mathbf{1}_{[1/n,1]}(y)dy
 =- \lim_{n \maps \infty} \int_{1/n}^1\log y dy= 1. \]
In summary, we can conclude that $H(\alpha+\epsilon+iy)$ is dominated, uniformly for all $\epsilon > 0$ and $y \in [-2\lambda,2\lambda] \setminus \{0\}$, by $\mathcal{H}(y)$ which is an integrable function of $y$ on $[-2\lambda,2\lambda]$.

   \subsubsection{Proof of Lemma \ref{lemma_h_convergence} \ref{g_exists}}

 Since \ref{h_interval} holds for any closed interval $I$ not containing the origin,  we may conclude that for real $y \neq 0$ the limit $H_0(y):=\lim_{\ep \maps 0+}H(\al+\ep + iy)$ exists, and is a continuous function (for $y \neq 0$). Fix any $\lam>0$; it follows by (ii) and the Dominated Convergence Theorem that $H_0(y)$ is integrable on $[-2\lambda, 2\lambda] \setminus \{0\}$ and therefore integrable on $[-2\lambda,2\lambda]$.

   \subsubsection{Proof of Lemma \ref{lemma_h_convergence} \ref{h_int}}
To verify \ref{h_int} we apply the Dominated Convergence Theorem. Note that the given continuous function $\phi(y)$ is necessarily uniformly bounded for $y \in [-2\lam,2\lam]$, and so is integrable on this finite-length interval. By \ref{h_H},   $H(\al+\ep +iy)$ is dominated by an integrable function $\mathcal{H}(y)$ on $[-2\lam,2\lam]$, uniformly for all $0<\ep< 1$, so we may apply the Dominated Convergence Theorem. Combined with the limit function $H_0(y)$ constructed in \ref{g_exists}, this proves the identity in \ref{h_int}.
 This completes the proof of Lemma \ref{lemma_h_convergence}.

 \subsection{Proof of Proposition \ref{prop_JK_integral} (I)}
  We begin with a lemma, which immediately implies Proposition \ref{prop_JK_integral} (I).
 \begin{lemnum}\label{lemma_prep_JK_integral}
 Let $K_\lam$ be an allowable approximation to the identity, as defined in (\ref{allowable_dfn}) and $J(t)$ as in (\ref{J_dfn}).
Fix $\lam>0$ and real $\xi$.  For every $\ep>0$, the integral
\[  \int_0^\infty \exp(-\ep t)K_\lam(\xi-t) J(t)dt\] 
is convergent. For $\ep \maps 0^+$, we may interchange the limit and integration, so that
\[ \lim_{\ep \maps 0^+} \int_0^\infty \exp(-\ep t)K_\lam(\xi-t) J(t)dt= \int_0^\infty \lim_{\ep \maps 0^+}  \exp(-\ep t)K_\lam(\xi-t) J(t)dt.\]
\end{lemnum}

For the first claim in the lemma, apply the definition (\ref{J_dfn}) of $J(t)$ to write for fixed $\ep>0$:
\begin{align}
\int_0^\infty  \exp(-\ep t)K_\lam(\xi-t) J(t)dt
& =\int_0^\infty  \exp(-\ep t)K_\lam(\xi-t) t v(t) e(-\al t) a(t)dt \nonumber\\
& =\int_0^\infty  \exp(-(\al +\ep) t)K_\lam(\xi-t) t v(t)   a(t)dt.\label{J_ep_idea}
\end{align}
We recall that the hypotheses of Theorem \ref{thm_Delange} provide properties of the function $F(s),$
expressed for $\Re(s)>\al$ by
 \beq\label{F_integral}
 F(s) =  f(s) - G(s-\al)= \int_0^\infty  \exp(-st) [a(t)  - \exp(\al t) b(t)] dt.
 \eeq
 Thus consider  the expression
\beq\label{ep_two_integrals}
 \int_0^\infty  \exp(-(\al+\ep)t)K_\lam(\xi-t) t v(t) [a(t) - \exp(\al t)b(t)]dt + \int_0^\infty \exp(-\ep t) K_\lam(\xi-t) t v(t)b(t)dt .
\eeq 
To prove the first claim in the lemma, it suffices to show that for fixed $\ep>0$, each of these is a convergent integral, so that  their sum, which  is equal to (\ref{J_ep_idea}), is convergent.

For each $\ep>0$, the first integral in (\ref{ep_two_integrals}) is   convergent 
by Lemma \ref{lemma_h_initial_ppties}, which also has shown this integral is identical to 
\beq\label{reveal_h} 
\int_{-2\lam}^{2\lam} H(\al+\ep+iy)k^\lam(y)\exp(i\xi y)dy.
\eeq
 For each $\ep>0$, the second integral in (\ref{ep_two_integrals}) is   convergent (in fact uniformly for all $\ep \geq 0$), since $K_\lam(\xi-t)$ is an integrable function and by hypothesis   $\lim_{t \maps \infty} t b(t)v(t)=1$ (as in Lemma \ref{lemma_v_limit}).

Now we focus on the second claim in the lemma; it again suffices to consider each term in (\ref{ep_two_integrals}) individually, so for each term we will show that we can interchange the limit $\lim_{\ep \maps 0^+}$ with the integral. 
For the first term (namely (\ref{reveal_h})), this is true by Lemma \ref{lemma_h_convergence} \ref{h_int}, applied with the continuous function $\phi(y) = k^\lam(y) \exp(i \xi y).$ 
Moreover,  by Lemma \ref{lemma_h_convergence} we have shown that 
\[ \lim_{\ep \maps 0+} \int_0^\infty  \exp(-(\al+\ep)t)K_\lam(\xi-t) t v(t) [a(t) - \exp(\al t)b(t)]dt =\int_{-2\lam}^{2\lam} H_0(y) k^\lam (y) \exp(i \xi y)dy, \] 
for a function $H_0(y)$ that is integrable on $[-2\lam,2\lam]$ (and continuous for $y \neq 0$), so that this is convergent for each fixed $\xi$. 

For the second term in (\ref{ep_two_integrals}), recall it is   convergent,  uniformly for all $\ep \geq 0$, since $K_\lam(\xi-t)$ is an integrable function and by hypothesis   $\lim_{t \maps \infty} t b(t)v(t)=1$. 
Thus there is an integrable function of $t$ that dominates the integrand (uniformly for all $\ep \geq 0$). Thus  by the Dominated Convergence Theorem, we may take the limit inside the integral, obtaining
 \beq\label{b_integral_limit}
 \lim_{\ep\maps 0^+} \int_0^\infty \exp(-\ep t) K_\lam(\xi-t) t v(t)b(t)dt=\int_0^\infty  K_\lam(\xi-t) t v(t)b(t)dt ,
 \eeq
 and this integral is convergent for each fixed $\xi$. 
 This completes the proof of Lemma \ref{lemma_prep_JK_integral}. 
 
 In the discussion above, note that we have also proved an explicit expression:
 \begin{lemnum}\label{lemma_prep_JK_integral_explicit}
     For each $\lam>0$ and real $\xi$, 
     \[ \int_0^\infty K_\lam(\xi-t)J(t)dt
     = \int_{-2\lam}^{2\lam} H_0(y) k^\lam (y) \exp(i \xi y) dy + \int_0^\infty  K_\lam(\xi-t) t v(t)b(t)dt .\]
 \end{lemnum}

 \subsection{Proof of Proposition \ref{prop_JK_integral} (II)}

For the second claim of Proposition \ref{prop_JK_integral}, we take the limit as $\xi \maps \infty$. It suffices to take the limit as $\xi \maps \infty$ in  each term individually on the right-hand side of the identity in Lemma \ref{lemma_prep_JK_integral_explicit}, and show both terms have a finite limit, say $L_1$ and $L_2$ respectively, with the property that $L_1 + L_2=1$. In the first integral on the right-hand side, $H_0(y)k^\lam(y)$ is an integrable function,  continuous for $y \neq 0$, with support in a bounded set. Thus in the limit as $\xi \maps \infty$, the first integral vanishes by the Riemann-Lebesgue Lemma \cite[Lemma 7.5]{BatDia04} (that is to say, $L_1=0$). 

For the second integral on the right-hand side, setting $u=t-\xi$ (and recalling $K_\lam$ is an even function) expresses it as 
 \[ \int_{-\infty}^\infty K_\lam (u) \om(u;\xi) du\]
 in which $\om(u;\xi) = (u+\xi)v(u+\xi)b(u+\xi)$ if $u \geq -\xi$ and $\om(u;\xi)=0$ for $u< -\xi$. Since $\lim_{t \maps \infty} tv(t)b(t)=1$ (as in Lemma \ref{lemma_v_limit}), we may conclude that  $\om(u;\xi)$ is bounded, uniformly in $\xi$ and $u$, and moreover for each $u$, $\lim_{\xi \maps \infty} \om(u;\xi)=1$. Thus by the integrability of $K_\lam(u)$, the Dominated Convergence Theorem applies, leading to 
 \[ \lim_{\xi \maps \infty}\int_{-\infty}^\infty K_\lam (u) \om(u;\xi) du = \int_{-\infty}^\infty K_\lam (u)du=\int_{-\infty}^\infty K(v)dv = 1,\]
since $K_\lam$ is an allowable approximation to the identity. That is to say, $L_2=1$. 
 This confirms the second claim in Proposition \ref{prop_JK_integral}, and completes the proof of the proposition. Consequently, Theorem \ref{thm_Delange} is also proved.

  \begin{remark}[Hypothesis \ref{hyp:Weak} and non-negativity]
We remark on the hypothesis in Theorem \ref{thm:Weak} that the coefficients $a_n$ are non-negative. This is used to guarantee that the function $a(t) = \sum_{n \leq \exp(t)} a_n$ is non-decreasing, so that Theorem \ref{thm_Delange} can be applied.
Observe that in the proof of Theorem \ref{thm_Delange}, the hypothesis that $a(t)$ is   non-decreasing is used twice in an essential way in \S \ref{sec_main_asymptotic_weak}, when proving  the main asymptotic for the function $J(t)$: both in the proof that $\limsup_{t\maps \infty} J(t) \leq 1$  and   in the proof that $\liminf_{t\maps \infty} J(t) \geq 1$. 
\end{remark}

\section{Examples with Theorem \ref{thm:Weak}}\label{sec_A_corollaries}

 \subsection{Notes on \cref{cor_HypA_PNT}}
 By a classical result for the logarithmic derivative of the Riemann zeta function $\zeta(s)$,    $A(s)=-\zeta'(s)/\zeta(s)$ can be meromorphically extended to all of $\C$. In general, for any meromorphic function $f(s)$, its logarithmic derivative $f'(s)/f(s)$ has a simple pole at each pole or zero of $f(s)$, with the residue being: the negative of the order of the pole,  or the order of the zero, respectively \cite[\S 3.4]{SSComp}. Thus since $\zeta(s)$, originally defined by a Dirichlet series for $\Re(s)>1$, can be meromorphically extended to $\C$ then so can $A(s)$. Since $\zeta(s)$ is holomorphic for $\Re(s) \geq 1$ aside from a simple pole at $s=1$, and does not vanish in the region $\Re(s)\geq 1$, then $A(s)$ is holomorphic for $\Re(s) \geq 1$ with its only singularity being a simple pole at $s=1$ of residue 1 (see for example \cite[Ch. 7]{SSComp}, or \cite[Thm. 10]{Ing32}). Thus an application of Theorem \ref{thm:Weak} to $A(s)$ shows that  $\sum_{n \leq x} \Lambda(n) \sim x$. This implies $\pi(x) \sim x/\log x$  by a version of partial summation; see for example \cite[Ch. 7]{SSComp} or \cite[Thms. 3 and 12]{Ing32}.

  \subsection{Notes on \cref{cor_HypA_dSG}}\label{sec_A_corollaries_group}
This example is \cite[Thm. 1.1]{SauGru00}. In that work, du Sautoy and Grunewald showed that for the zeta function $\zeta_G(s)$ of the group $G$ defined as in (\ref{zeta_group_dfn}), and abscissa of convergence $\al_G>0$ defined as in (\ref{zeta_group_abscissa_dfn}),  $\zeta_G(s)$ admits a meromorphic continuation to $\Re(s) \geq \al_G$ and in this region the only singularity of $\zeta_G(s)$ is a pole at $s=\al_G$ of a finite order $m$ for some integer $m \geq 1$. They then apply an appropriate Tauberian theorem, which they state as \cite[Thm. 4.19]{SauGru00}, and credit to Ikehara as proved by Delange in \cite[p. 62]{Del55}. Here it suffices to apply Theorem \ref{thm:Weak}.
 
 As remarked by du Sautoy and Grunewald \cite[p. 823]{SauGru00}, in order to prove the meromorphic continuation of $\zeta_G(s)$, they employed Artin $L$-functions. While it is expected that an Artin $L$-function $L(\rho,s)$ has only a simple pole at $s=1$ if the representation $\rho$ is trivial, and is entire if $\rho$ is nontrivial, this is not known unconditionally (in general). Thus the methods of \cite{SauGru00} do not rule out that $\zeta_G(s)$ could have a pole for some $s$ with $\Re(s)< \al_G$ arbitrarily close to $\al_G$. Thus while they state in \cite[Thm. 1.1]{SauGru00} a meromorphic continuation in a half-plane strictly larger than $\Re(s) \geq \al_G$, it is inexplicit, and Theorem \ref{thm:Weak} remains the appropriate type of Tauberian theorem to apply.

 \subsection{Notes on  \cref{cor_HypA_Klu}}
     Kl\"uners proved the result in \cref{cor_HypA_Klu} by proving an identity of the form
\[A(s) = G(s) H(s) \zeta_{k(\zeta_\ell)}(ds)^e ,\]
in which $G(s)$ is a finite Euler product (hence convergent), $H(s)$ is a convergent Euler product for $\Re(s) \geq 1/d$, and $\zeta_{k(\zeta_\ell)}(s)$ is the Dedekind zeta function of the cyclotomic extension. Since $\zeta_{k(\zeta_\ell)}(s)$  has a simple pole at $s=1$, it follows that $\zeta_{k(\zeta_\ell)}(ds)^e $ has a pole of order $e$ at $s=1/d$, and the example follows from Theorem \ref{thm:Weak} with $\al=1/d$ and $m=e$.
In fact, Kl\"uners proves a more general result in \cite[Lemma 2.2]{Klu22} that allows $m$ to be any positive integer, and then the statement of \cref{cor_HypA_Klu} holds with the rational number $e:=m/[k(\zeta_\ell):k]$; this requires the application of a more general Tauberian theorem described in \S \ref{sec_lit_HypA}.

\subsection{When an upper bound suffices}\label{sec_remark_upper_bound}
We noted after Theorem \ref{thm:Weak} that a simple corollary of the theorem is the statement that $\sum_{\lam_n \leq x} a_n = O(x^\al(\log x)^{m-1})$. In fact, if only such an upper bound is required, it can be obtained from a slightly weaker hypothesis than Hypothesis \ref{hyp:Weak}. This is explained in \cite[\S 7.2.2]{BatDia04} for the case of a standard Dirichlet series with a simple pole ($m=1$). Briefly, suppose that $F(x)$ is a real-valued monotone non-decreasing  function  supported in $[1,\infty)$ (continuous from the right, locally of bounded variation) with Mellin transform $\mathcal{M}(F)(s)$. (For example, $F(x)=\sum_{n \leq x} a_n $ with $a_n \geq 0$ and associated standard Dirichlet series $A(s)$.) Suppose that
$\mathcal{M}(F)(s)$ has abscissa of convergence $\al >0$ and there exists a real number $a$ and a function $\phi,$ continuous on the closed half-plane $\{s: \sig \geq \al\}$, such that 
 \[ \mathcal{M}(F)(s) = \int x^{-s} dF(x) = \frac{a}{s-\al} + \phi(s)\]
  in the horizontal half-strip $\{ s: \sig \geq \al, |t|< 2\lam_0\}$, for some fixed $\lam_0>0$. 
 Then $F(x) = O(x^\al).$ The proof is obtained by following the proof of the Ikehara--Wiener theorem (for $m=1$ in this case), until an upper bound is obtained, without claiming a lower bound. (In \cite{BatDia04}, this corresponds to repeating the proof of Theorem 7.3 until relation (7.10); in the proof we gave here for Theorem \ref{thm:Weak} this would be analogous to proving only that $\limsup_{t \maps \infty} J(t) = O(1)$, and not pursuing a $\liminf$ statement.)

  \subsection{Reasons for caution under Hypothesis \ref{hyp:Weak}: further examples}\label{sec_hyp_weak_counterex_further}
Theorem \ref{thm_hyp_weak_counterex} (whose proof we defer to \S \ref{sec_hyp_weak_counterex}) provides  one example, specific to Dirichlet series, that shows the conclusion of Theorem \ref{thm:Weak} cannot be improved under Hypothesis \ref{hyp:Weak}; there are far broader classes of examples in a more general setting. 
For a pole of order $m=1$, Theorem \ref{thm:Weak} is called an Ikehara--Wiener theorem, and a version can be stated in terms of Mellin transforms. We can summarize a basic form of this result as the following statement, in which we consider a simple pole at $s=1$ for simplicity.  Let $S(x)$ be a real-valued non-decreasing function  supported on $(1,\infty)$. Suppose that the Mellin transform defined in (\ref{Mellin_dfn1}),
\[ \mathcal{M}(S)(s):= \int_1^\infty  x^{-s}dS(x)= s\int_1^\infty S(x) x^{-s-1}dx,\]
converges for $\Re(s)>1$, and that there exists a real $a\geq 0$ such that $\mathcal{M}(S)(s) -a( \frac{s}{s-1})$ admits a continuous extension to the half-plane $\Re(s) \geq 1$. Then an Ikehara--Wiener theorem states
\beq\label{S_expansion}
S(x) = ax + o(x) \qquad \text{as $x \maps \infty$.}
\eeq
(See for example \cite[Thm. 1]{DebVin18}.)

We record three observations. First, a converse is true. Indeed, suppose that   a function $S(x)$  is complex-valued, supported on $(1,\infty)$, right-continuous and locally of bounded variation. Then if (\ref{S_expansion}) holds, integration by parts shows that uniformly in   $\Im (s)$,
\beq\label{MSs_relation}
\mathcal{M}(S)(s) =  
s\int_{1}^\infty x^{-s}(a + o(1)) dx  =a \left( \frac{s}{s-1}\right) + |s| o\left(\frac{1}{\Re(s)-1}\right), \qquad \text{as $\Re(s) \maps 1^+.$}
\eeq
From this, a brief argument shows that if  $\mathcal{M}(S)$ can be extended as a meromorphic function to a region containing the closed half-plane $\{ s: \Re(s)\geq 1\}$, then within this region $\mathcal{M}(S)$ has only a single simple pole with residue $a$ at $s=1$ (if $a \neq 0$); or no singularities in the closed half-plane (if $a=0$). This is an ``Abelian theorem,'' and is a generalization of the earlier statement in Lemma \ref{lemma_abelian_asymptotic_implies_order_of_pole} that (\ref{lim_known}) implies (\ref{residue_known}). (For example, the truth of the Prime Number Theorem  implies that $\zeta'(s)/\zeta(s)$ has no singularities on the line $\Re(s)=1$ aside from a simple pole at $s=1$, so that the Riemann zeta function has no zeroes on this line.) See \cite[Lemma 7.1, Cor. 7.2]{BatDia04} for these considerations, or a more general Abelian theorem  in \cite[Ch. II \S 7.1 Thm. 2]{Ten15}.  

Second, (\ref{S_expansion}) need not hold if $\mathcal{M}(S)(s) - a(\frac{s}{s-1})$ is only known to be well-behaved in a horizontal right half-strip (of bounded height), rather than a right half-plane.  For observe that if (\ref{S_expansion}) holds, then (\ref{MSs_relation}) holds, and this implies that for $s=\sig+it$ at any height $t \neq 0$,
\begin{multline*}\limsup_{\sig \maps 1^+}(\sig-1)|\mathcal{M}(S)(\sig+it)| \\
\ll_a  \limsup_{\sig \maps 1^+}(\sig-1)\left[  \frac{(\sig^2+t^2)^{1/2}}{((\sig-1)^2+t^2)^{1/2}} + (\sig^2+t^2)^{1/2}o\left(\frac{1}{\sig-1}\right)\right]=0.
\end{multline*}
But suppose $S$ is a function such that $\mathcal{M}(S)(s)$ has a pole at $s=1+it_0$ for some $t_0 \neq 0$. Then 
 \[ \limsup_{\sig \maps 1^+} (\sig-1)|\mathcal{M}(S)(\sig+it_0)|>0,\]
 contradicting the line above, so the claim (\ref{S_expansion}) must be false in this case, even though (\ref{MSs_relation}) could still hold for $|\Im(s)|<t_0$. (This generalizes the earlier statement that (\ref{residue_known}) need not imply (\ref{lim_known}), so the converse to Lemma \ref{lemma_abelian_asymptotic_implies_order_of_pole} is false.)
See \cite[\S 7.2]{BatDia04} for such considerations. Nevertheless, as observed earlier in \S \ref{sec_remark_upper_bound},   ``good behavior'' inside a half-strip (of bounded height) can still provide the weaker conclusion $S(x) = O(x)$.

Third,  the $o(x)$ remainder term in (\ref{S_expansion}) cannot essentially be improved, in the following sense. Let $\rho$ be any positive function defined on $[1,\infty)$. Fix $a>0$ and $0<\del<1$. If \emph{every} non-decreasing function $S(x)$  on $[1,\infty)$   such that $\mathcal{M}(S)(s) -a(\frac{s}{s-1})$ admits a continuous extension to $\Re(s) >1-\del>0$, satisfies the property 
\[ S(x) = ax + O(x \rho(x)),\]
then $\rho(x) = \Omega (1)$.  That is to say, for any presumed explicit rate of decay for the $o(x)$ term in (\ref{S_expansion}), there is some $S(x)$ in this class for which the remainder term decays slower than that presumed rate. This has been proved by Debruyne and Vindas  via a nonconstructive argument using ideas from abstract functional analysis; see \cite[Thm. 2]{DebVin18}. An alternative method constructs explicit counterexamples: given any positive function $\rho(x)$ defined on $[1,\infty)$ that tends to $0$ as $x \maps \infty$, there is a construction of a non-decreasing function $S(x)$ on $[1,\infty)$ such that $\mathcal{M}(S)(s)$ converges for $\Re(s)>1$, $\mathcal{M}(S)-(s-1)^{-1}$ extends to an entire function, and $S(x) = x + \Omega(x\rho(x)).$ See \cite[Thm. 1]{BDV21}.

\section{Hypothesis \ref{hyp:Strong}: Preliminaries about why and how to smooth}\label{sec_B_prelim}

\subsection{Motivation for smoothing}\label{sec_smooth}
Our next interest lies in obtaining an asymptotic for the partial sum  $\sum_{\lam_n \leq x} a_n$  \emph{with an explicit remainder term},  under \cref{hyp:Strong} for the general Dirichlet series $A(s)$. A standard approach to such a Tauberian theorem is to replace the sharp cut-off weight $\mathbf{1}_{\leq x}(u)$ with a smooth approximation $\varphi(u)$. These ideas can be traced back to  Landau \cite{Landau1912,Landau1917,Landau1918}  when counting integral ideals of a number field and other similar problems. The first step is to establish a general form of Mellin inversion for such smooth approximations. When working under Hypothesis \ref{hyp:Strong},  we define the Mellin transform $\hat{\phi}$  of any appropriate function $\phi$ using the convention (\ref{Mellin_dfn2}). Below is one example of a Mellin inversion statement, which is appropriate for our needs. 
\begin{lemnum}\label{lemma_Mellin_inversion}
Assume Hypothesis \ref{hyp:Strong} for $\alpha,\delta,\kappa,m$. Fix an integer $\ell \geq \lceil \kappa \rceil + 3$. If $\varphi$ is a real-valued $C^{\ell-1}$ function compactly supported within the open interval $(0,\infty)$, then 
\beq\label{smooth_argument_asymptotic_prep}
\sum_{n=1}^{\infty} a_n \varphi(\lambda_n)  = \Res_{s=\alpha} [  A(s) \hat{\varphi}(s) ] + \frac{1}{2\pi i} \int_{(\alpha-\delta)} A(s)  \hat{\varphi}(s) ds.
\eeq

\end{lemnum}
Here we use the standard convention that for a real number $c$, $\int_{(c)}$ denotes the integral up the vertical line $\Re(s)=c$.
The proof of Lemma \ref{lemma_Mellin_inversion} follows a standard sequence of ideas using the inverse Mellin transform, which we momentarily postpone to Section \ref{subsec:Prelim}. In this subsection, our aim is instead to motivate the basic strategy of smoothing and unsmoothing.  

Let $\varphi$ be a smooth weight according to Lemma \ref{lemma_Mellin_inversion}, so $\varphi$ is a real-valued $C^{\ell-1}$-function that is compactly supported on $(0,\infty)$, where $\ell \geq \lceil \kappa \rceil + 3$.  For a parameter $x \geq 3$, if the smooth weight $\varphi(u)$ is a good approximation to the sharp cut-off weight $\mathbf{1}_{\leq x}(u)$, then certainly the smooth weight $\varphi$ must depend on $x$ in some way. For convenience, we will additionally assume $\varphi$ is compactly supported inside the smaller open interval $(0,x)$. We begin with a simple estimate for the Mellin transform of such a weight.

\begin{lemnum}\label{lemma_smooth_Mellin_bound}
Fix an integer $k \geq 0$, and real $x \geq 3$. Let $\phi$ be a real-valued continuous function that is compactly supported on $(0,x)$, and is continuously differentiable $k$ times. Then  the Mellin transform $\hat{\phi}(s)$ is entire, and for each $1 \leq j \leq k$, 
\beq\label{smooth_argument_phi_bound}
\hat{\varphi}(s)  = \int_0^x \varphi^{(j)}(u) \frac{u^{s+j-1}}{s(s+1)\cdots(s+j-1)} du \ll_{j} x^{\Re(s)} \cdot \frac{x^{j}  }{|s|^{j}} \|\varphi^{(j)}\|_{\infty} \quad \text{for $\Re(s) > 0$}. 
\eeq 
\end{lemnum}
In other words, if $\phi$ is $C^k$ for a sufficiently large $k$, the Mellin transform $\hat{\varphi}(s)$ decays rapidly for $s$ along vertical lines. (A similar bound holds for $j=0$; by convention, if $\phi \in C^0$ we mean $\phi$ is continuous.)
\begin{proof}
Since $\varphi(u)$ is continuous and compactly supported in $(0,x)$ (and in particular $\varphi$ vanishes at $0$), the Mellin transform $\hat{\varphi}(s)$ is entire (Lemma \ref{lem:mellin_cpt_cts}). For a fixed $j$ with $1 \leq j \leq k$, integrating by parts $j$ times gives \eqref{smooth_argument_phi_bound}. Nominally, the last step uses that $\Re(s)+j-1 \neq -1$, which is true since $\Re(s) > 0$.
\end{proof}

In   the setting of Lemma \ref{lemma_Mellin_inversion} where $\phi$ is $C^{\ell-1}$, the choice $\ell \geq \lceil \kappa \rceil + 3$ ensures that  $\varphi$ is sufficiently smooth that its Mellin transform $\hat{\varphi}(s)$ will decay \textit{faster} than the Dirichlet series $A(s)$ grows, according to \cref{hyp:Strong}. Indeed, applying (\ref{smooth_argument_phi_bound}) with $j=\ell-1$ and $\ell=\lceil \kappa \rceil + 3$ to bound $\hat{\phi}(s)$, and \cref{hyp:Strong} to bound the Dirichlet series $A(s)$, leads to a bound  for the remainder term in (\ref{smooth_argument_asymptotic_prep}):  
\begin{align*}
 \frac{1}{2\pi i} \int_{(\alpha-\delta)} A(s)  \hat{\varphi}(s) ds 
 & \ll_{\ell} x^{\alpha-\delta} \cdot x^{\ell-1} \|\varphi^{(\ell-1)}\|_{\infty} \int_{\alpha-\delta-i\infty}^{\alpha-\delta+i\infty} \frac{|s|^{\kappa} \log^{m-1}(3+|\Im(s)|)}{|s|^{\kappa+2}} |ds|  \\
 & \ll_{\ell,m}  x^{\alpha-\delta} \cdot x^{\ell-1} \|\varphi^{(\ell-1)}\|_{\infty}. 
\end{align*}
This leads to an asymptotic statement of the following shape:
	\beq\label{smooth_argument_conclusion}
	\sum_{n=1}^{\infty} a_n \varphi(\lambda_n) = \Res_{s=\alpha}\Big[A(s) \hat{\varphi}(s) \Big] + O_{\ell,m}( x^{\alpha-\delta} \cdot x^{\ell-1} \|\varphi^{(\ell-1)}\|_{\infty}   ). 
	\eeq
At first glance, the remainder  term appears quite frightening with the expression $x^{\ell-1} \|\varphi^{(\ell-1)}\|_{\infty},$ but this quantity is not so severe. This is more easily seen by temporarily rescaling the smooth weight $\varphi$ with compact support in the open interval $(0,x)$ to a new smooth weight $\Phi$ with compact support in the open interval $(0,1)$. Define
\[
\Phi(t) := \varphi(xt) \quad \text{ for } t \in \R, \quad \text{ so } \quad \hat{\Phi}(s) = \hat{\varphi}(s) x^{-s} \quad \text{and} \quad \|\Phi^{(\ell-1)}\|_{\infty} = x^{\ell-1} \|\varphi^{(\ell-1)}\|_{\infty} 
\]
by performing the substitution $u=xt$ in the formula \eqref{Mellin_dfn2}. Thus, \eqref{smooth_argument_conclusion} is equivalent to
\beq\label{smooth_argument_conclusion_rescaled}
\sum_{n=1}^{\infty} a_n \Phi\Big(\frac{\lambda_n}{x}\Big) = \Res_{s=\alpha}\Big[A(s) x^s \hat{\Phi}(s) \Big] + O_{\ell,m}( x^{\alpha-\delta} \|\Phi^{(\ell-1)}\|_{\infty}   ),
\eeq
which appears to possess the best power savings that one could expect under \cref{hyp:Strong}.

Notice, however, that there are two defects with \eqref{smooth_argument_conclusion_rescaled} and equivalently with \eqref{smooth_argument_conclusion}. First, the residue depends on $\Phi$ and does not equate to the natural or exact residue $\Res_{s=\alpha}[A(s) x^s/s ]$ that leads to the main term in \cref{thm:Weak}. (We can see this explicitly by modifying the computation in Remark \ref{remark_residue}, replacing $x^s/s$ by $\hat{\phi}(s) = x^s \hat{\Phi}(s)$. For example, if the pole at $s=\al$ is of order $m=1$, the ``smoothed'' residue is $g(\al) x^{\alpha}\hat{\Phi}(\al),$ which is of the right order of magnitude, but differs from the exact residue $g(\al)x^\al/\al$ by a factor of $\hat{\Phi}(\alpha)-1/\alpha$.)
As a second defect, the remainder term depends on the $L^{\infty}$-norm of the $(\ell-1)$-th derivative of $\Phi$. 
 
  To  obtain nevertheless an asymptotic for the sharply cut-off sum $\sum a_n \mathbf{1}_{\leq x}(\lam_n)$, one must make a  choice for $\varphi$ (or equivalently for $\Phi$); unfortunately these defects are not cooperative. For example, if one chooses $\Phi$ to be \textit{independent} of $x$, then this forces the remainder term to be $O_{\ell}(x^{\alpha-\delta})$ but the smoothed residue will poorly approximate the natural residue  since the differing factor $\hat{\Phi}(\alpha)-1/\alpha$ will necessarily be a non-zero constant independent of $x$.  On the other hand, if one chooses $\Phi$ to be \textit{dependent} on $x$ then, as the approximation of the sharp cut-off weight $\mathbf{1}_{\leq x}( \, \cdot \, )$ by the smooth weight $\varphi( \, \cdot \, ) = \Phi( \, \cdot \, / x)$ improves, the residue in \eqref{smooth_argument_conclusion_rescaled} becomes closer to the natural residue  but the remainder term in \eqref{smooth_argument_conclusion_rescaled} inflates. This inflation occurs because the derivatives of $\Phi$ must grow with $x$ to better approximate the sharp cut-off. Our proof of Theorem \ref{thm:Strong} in \S \ref{sec_B_proof} will carefully balance these issues, by constructing a specific type of smooth weight $\varphi$ whose rescaled version $\Phi$ will necessarily depend on the parameter $x$. We also call upon the non-negativity assumption $a_n \geq 0$.
We will reduce to considering a sum of $a_n$ over an interval $I=[y,x]$ for some $y>0$, and then construct \emph{two} weights, a majorant $\phi^+$ of $\mathbf{1}_{[y,x]}$ that is supported on an enlargement $I^+\supset I$, and a minorant $\phi^-$ that is supported on a contraction $I^- \subset I$ (with enlargement/contraction controlled by a parameter $\ep$). Each of $\phi^+$ and $\phi^-$ takes values in $[0,1]$, so by the assumption that $a_n \geq 0$,
\[ \sum_{n=1}^\infty a_n \phi^-(\lam_n) \leq \sum_{y< \lam_n \leq x} a_n \leq \sum_{n=1}^\infty a_n \phi^+(\lam_n).\] 
We will show that the left-most sum and the right-most sum both contribute the same main term, with remainder terms depending on the parameter $\ep$ (which encode the contribution of the short intervals $I^+\setminus I$ and $I\setminus I^-$). The identity of the main terms, followed by an optimal choice of $\ep$, then leads to Theorem \ref{thm:Strong}. For more on non-negativity, see Remark \ref{remark_non-negativity}.

\subsection{Proof of Lemma \ref{lemma_Mellin_inversion}} \label{subsec:Prelim} 
For clarity, we organize the ideas to prove  Lemma \ref{lemma_Mellin_inversion} into four ingredients: Mellin inversion for our smooth weight $\varphi$, Perron's formula for the smoothed sum $\sum_n a_n \varphi(\lambda_n)$, a growth estimate for the Dirichlet series $A(s)$ in a vertical strip, and finally a contour shift. Recall that these deductions will be made assuming \cref{hyp:Strong} with $\alpha,\delta,\kappa,m$. 
\subsubsection{Mellin inversion for well-behaved functions}\label{sec_Mellin_justification}
We begin with an explicit statement of inversion of  the Mellin transform defined in \eqref{Mellin_dfn2}.
\begin{lemnum}
Let $\phi$ be a real-valued $C^2$ function compactly supported in $(0,\infty)$. 
For any fixed $\sigma > 0$,  
\begin{equation} \label{Mellin_inversion_smooth}
\varphi(u) = \frac{1}{2\pi i} \int_{\sigma-i\infty}^{\sigma+i\infty} \hat{\varphi}(s) u^{-s} ds \qquad \text{ for } u \in (0,\infty), 
\end{equation}
\end{lemnum}
 The proof of such inversion formulas can be delicate, depending on the nature of the function $\varphi$; see \cite[Section 5.1]{MonVau07} for various scenarios. Luckily, in the setting of Lemma \ref{lemma_Mellin_inversion}, the function $\varphi$ is compactly supported in $(0,\infty)$ and $C^{\ell-1}$ with $\ell \geq 3$ (and hence $C^2$), so the present lemma suffices.
\begin{proof}

Fix $u \in (0,\infty)$. For $T \geq 3$, it follows from \eqref{smooth_argument_phi_bound} with $j=2$ and Fubini's theorem that 
\begin{align*}
\frac{1}{2\pi i} \int_{\sigma-iT}^{\sigma+iT} \hat{\varphi}(s) u^{-s} ds 
& = \frac{1}{2\pi i}  \int_{\sigma-iT}^{\sigma+iT} \Big( \int_0^{\infty} \varphi''(t) \frac{t^{s+1}}{s(s+1)} dt \Big) u^{-s} ds\\ 
& =  \int_0^{\infty}  t \varphi''(t) \Big( \frac{1}{2\pi i} \int_{\sigma-iT}^{\sigma+iT} \frac{(u/t)^{-s}}{s(s+1)}ds \Big) \, dt. 
\end{align*} 
Note the application of Fubini's theorem is trivial since $\varphi$ is compactly supported in $(0,\infty)$. Next, as the inner integral $ds$ converges absolutely and uniformly as $T \to \infty$, we may take the limit of both sides as $T \to \infty$ and, by Lebesgue's Dominated Convergence Theorem, we may interchange this limit with the integral $dt$ over the compact support of $\varphi$.
For any fixed $\sig > 0$ (see \cite[Ch. 7 Lemma 2.4]{SSComp}), 
\[
\lim_{T \to \infty} \frac{1}{2\pi i} \int_{\sigma-iT}^{\sigma+iT} \frac{y^{-s}}{s(s+1)} ds = \begin{cases} 1-y & \text{if } 0 < y \leq 1, \\  0 & \text{if } 1 \leq y < \infty. \end{cases}
\]
Combining the previous observations with this identity at $y = u/t$, we conclude that 
\[
\frac{1}{2\pi i} \int_{\sigma-i\infty}^{\sigma+i\infty} \hat{\varphi}(s) u^{-s} ds = \int_u^{\infty} t \varphi''(t) \Big(1-\frac{u}{t} \Big) dt = \varphi(u)
\]
after integrating by parts. This completes the verification of \eqref{Mellin_inversion_smooth}. 
\end{proof}

\subsubsection{Perron's formula for the smoothed sum} 

For the second ingredient, we state an
  identity that is a version of Perron's formula and another example of Mellin inversion.
  \begin{lemnum}
  Let $A(s)=\sum_{n} a_n \lam_n^{-s}$ be a general Dirichlet  series with non-negative coefficients $a_n \geq 0$. 
Suppose $A(s)$ converges (equivalently, converges absolutely) for $\Re(s)>\al>0$. 
  Let $\phi$ be a real-valued $C^2$ function compactly supported in $(0,\infty)$. 
  For any fixed $\sigma > \alpha$,  
\beq\label{Perron_demo_general}
\sum_{n=1}^{\infty} a_n \varphi(\lam_n) = \frac{1}{2\pi i }\int_{\sigma-i\infty}^{\sigma+i\infty} A(s) \hat{\varphi}(s) ds. 
\eeq
\end{lemnum}
Note that under Hypothesis \ref{hyp:Strong} for $\al,\del,\kappa,m$ (indeed even under Hypothesis \ref{hyp:Weak}), $A(s)$ has abscissa of convergence (and absolute convergence) $\sig_c =\sig_a= \al$ (see \S \ref{sec_convergence}), so that $A(s)$ satisfies the hypotheses. We will apply this when $\phi$ is $C^{\ell-1}$, with $\ell = \lceil \kappa \rceil +3$. 
\begin{proof}
  Since $A(s)$   converges absolutely for $\Re(s)>\alpha$ (and uniformly in $\Re(s) \geq \sigma,$ for any fixed $\sigma> \alpha$),  for a fixed $\sig>\al$ we may write
\[
A(s) = \sum_{n=1}^{\infty} a_n \lambda_n^{-s} \quad \text{ uniformly for } \Re(s) = \sigma. 
\]
Thus, by applying the inverse Mellin transform \eqref{Mellin_inversion_smooth} to each term $\phi(\lam_n)$ in (\ref{Perron_demo_general}), it suffices to show that 
\[
\frac{1}{2\pi i} \sum_{n=1}^{\infty}   \int_{\sigma-i\infty}^{\sigma+i\infty} a_n \lambda_n^{-s} \hat{\varphi}(s) ds  = \frac{1}{2\pi i }\int_{\sigma-i\infty}^{\sigma+i\infty} \sum_{n=1}^{\infty} a_n \lambda_n^{-s} \hat{\varphi}(s) ds,
\]
or in other words, we may interchange the integral and sum. This follows immediately from Fubini's theorem because the coefficients $a_n$   are non-negative, and  we may apply \eqref{smooth_argument_phi_bound} with $j=2$ to deduce the uniform upper bound 
\[
\int_{\sigma-i\infty}^{\sigma+i\infty} \sum_{n=1}^{\infty} |a_n \lambda_n^{-s} \hat{\varphi}(s)| |ds|  = A(\sigma)  \int_{\sigma-i\infty}^{\sigma+i\infty}   |\hat{\varphi}(s)| |ds| \ll A(\sigma) x^\sig \cdot x^2 \| \varphi''\|_{\infty}, 
\] 
where $x > 0$ is any fixed real number satisfying $\supp\varphi \subseteq (0,x)$.  This completes the proof. 
 \end{proof}

\subsubsection{Growth in a strip}
Hypothesis \ref{hyp:Strong} states a condition on the growth of $A(s)$ on precisely the vertical line $\Re(s) = \al - \del$. This can be combined with the fact that $A(s)$ is absolutely convergent for $\Re(s)>\al$ to deduce a bound for $A(s)$ throughout a vertical strip by applying the Phragm\'{e}n-Lindel\"{o}f convexity principle, which we quote from \cite[Thm. 2]{Rad59} (see also \cite[Thm. 5.53]{IwaKow04}). 
\begin{lemnum}[Phragm\'{e}n-Lindel\"{o}f]\label{lemma_PH}
    Let $f(s)$ be a function holomorphic on the strip $a \le \Re(s) \le b$, for some real numbers $a<b$, and assume that    for some $M_1,M_2\ge0,$ 
\[|f(s)| \leq M_1 \exp(|s|^{M_2}) \qquad \text{for all $s$ with $a\le \Re(s) \le b$. }\]
If  
\[|f(a+it)| \le C_a(1+|t|)^A, \qquad |f(b+it)| \le C_b(1+|t|)^B \qquad \text{for $t \in \mathbb{R}$,}\]
with $A \geq B$, then
    \beq\label{Phragmen_Lindelof}
    |f(\sigma + it)|\le C_a^{\ell(\sigma)}C_b^{1-\ell(\sigma)}(1+|t|)^{A\ell(\sigma)+B(1-\ell(\sigma))}
    \eeq
    for all $s$ in the strip $a \le \Re(s) \le b$, where $\ell(\sig) = \frac{b-\sig}{b-a}$ is the linear function such that $\ell(a)=1$ and $\ell(b)=0$.
    Note that the conclusion in (\ref{Phragmen_Lindelof}) is independent of $M_1$ and $M_2$. 
\end{lemnum}

We shall require the following consequence for $A(s)$. We will apply this both  during the change of contour to prove  Lemma \ref{lemma_Mellin_inversion},  in order to show the remainder  terms (integrals over  horizontal lines at height $\pm i T$) vanish as $T \maps \infty$, and in \S \ref{sec_B_proof}.
\begin{lemnum}\label{lemma_convexity}
     Assume Hypothesis \ref{hyp:Strong} for $\al,\del,\kappa,m$,   so that $A(s)$ converges absolutely for $\Re(s) > \al $. Fix any $\eta>0$. For   $\alpha-\delta \le \Re(s) \le \alpha + \eta$,
\[
 	\Big|\frac{(s-\alpha)^m}{(s+\alpha)^m}A(s)\Big| \ll_{\eta, A(\alpha+\eta)} (1+|\Im(s)|)^{(\kappa+\eta) ( \frac{\alpha  - \Re(s) +\eta }{\delta+\eta} )}.
 	\]
\end{lemnum}
\begin{proof}
To apply Lemma \ref{lemma_PH}, let 
\[
f(s) := \displaystyle\frac{(s-\alpha)^m}{(s+\alpha)^m}A(s).
\]
  In the strip $\alpha - \delta \le \Re(s) \le \alpha + \eta$,  $f(s)$ is holomorphic under the hypothesis that $A(s)$ has only the order $m$  pole at $s=\alpha$ (and  $(s+\al)^{-m}$ does not introduce a pole in this strip). Moreover, under Hypothesis \ref{hyp:Strong}, $|f(s)|\leq M_1'\exp(|s|^{M_2'})$ throughout the strip, for certain $M_1',M_2' \geq 0$ (in which $M_1'$ may depend on $\al,\del,m$). Also, there exists a constant $C_1\ge 1$ such that for all $t \in \R,$ 
\beq\label{f_bound_left}
|f(\alpha-\delta+it)|\le C_1(1+|t|)^{\kappa+\eta}.
\eeq
  On the other hand, under Hypothesis \ref{hyp:Strong}, $A(s)$   converges absolutely for $\Re(s)>\alpha$ (and uniformly in $\Re(s) \geq \al+\eta,$ for any fixed $\eta>0$). In this region of absolute (and uniform) convergence,  
\beq\label{f_bound_right}
|f(\alpha+\eta+it)| \ll \sum_{n=1}^{\infty}\frac{|a_n|}{|\lam_n^{\alpha+\eta+it}|} = \sum_{n=1}^{\infty}\frac{a_n}{\lam_n^{\alpha+\eta}} = A(\alpha+\eta) =: C_2,
\eeq
say, where $C_2$ is a positive constant. The linear function satisfying $\ell(\alpha-\delta)=1$ and $\ell(\alpha+\eta)=0$ is $\ell(\sigma)=(-\sigma + \alpha + \eta)/(\delta +\eta)$. Thus by (\ref{Phragmen_Lindelof}),
\[
|f(s)|  \le C_1^{\frac{-\Re(s) + \alpha + \eta}{\delta +\eta}}C_2^{\frac{\Re(s)-\alpha+\delta}{\delta +\eta}}(1+|\Im(s)|)^{(\kappa+\eta)\left(\frac{-\Re(s) + \alpha + \eta}{\delta + \eta}\right)}  \ll_{\eta, A(\alpha+\eta)}  (1+|\Im(s)|)^{(\kappa+\eta)\left(\frac{\alpha-\Re(s) +\eta}{\delta+\eta}\right)}
\]
for all $s$ with $\al-\del \leq \Re(s) \leq \al+\eta$. The verifies Lemma \ref{lemma_convexity}. 
\end{proof}

\subsubsection{Shifting the contour}\label{sec_shifting} Finally, we may complete the proof of Lemma \ref{lemma_Mellin_inversion}. Let $\phi$ be as in the lemma. From Perron's formula \eqref{Perron_demo_general}, we have shown for any fixed $\sigma > \alpha$ that
\[
\sum_{n=1}^{\infty} a_n \varphi(\lambda_n)  = \lim_{T \to \infty} \frac{1}{2\pi i} \int_{\sigma-iT}^{\sigma+iT} A(s) \hat{\varphi}(s) ds. 
\]

Let $\mathcal{R}_T$ denote a rectangular contour with vertical components     from $\sigma-iT$ to $\sigma+iT$, and  from $\al-\del+iT$ to $\al-\del-iT$, and with horizontal components $\mathcal{H}_T^+$ from $\sigma+iT$ to $\al-\del +iT,$ and $\mathcal{H}_T^-$ from $\al-\del -iT$ to  $\sigma-iT$ (with standard counterclockwise orientation). Since $A(s)$ has a pole of order $m$ at $s=\alpha$ from \cref{hyp:Strong} and the Mellin transform $\hat{\phi}(s)$ is entire, it follows by the residue theorem that 
\[  
\frac{1}{2\pi i } \int_{\sigma-iT}^{\sigma+iT} A(s) \hat{\varphi}(s)  ds = \Res_{s=\al}[A(s) \hat{\phi}(s)] + \frac{1}{2\pi i }\Big( \int_{\alpha-\delta-iT}^{\alpha-\delta+iT} -  \int_{\mathcal{H}_T^+}  - \int_{\mathcal{H}_T^-} \Big) A(s) \hat{\phi}(s) ds .
\]
Thus to verify (\ref{smooth_argument_asymptotic_prep}) and complete the proof, it suffices to show the integrals over the horizontal segments vanish as $T \maps \infty$. We will consider $\mathcal{H}_T^+$, and the argument to handle $\mathcal{H}_T^-$ is analogous.
We may assume $\sigma < \alpha+1$, and apply Lemma \ref{lemma_convexity} to bound $A(s)$ with a small fixed choice of $\eta \in (\sigma-\alpha,1)$ to deduce that
\[\int_{\mathcal{H}_T^+}  = \int_{\sigma+iT}^{\al-\del +iT} A(s) \hat{\phi}(s)  ds
\ll \int_{\al-\del}^{\sigma} T^{(\kappa+\eta)( \frac{\al - u +\eta}{\del + \eta})} | \hat{\phi}(u + iT)|  du
\ll  T^{\kappa+\eta} \int_{\al-\del}^{\sigma} |\hat{\phi}(u + iT) | du.
\]
Let $x > 0$ be a fixed real number such that $\supp \varphi \subseteq (0,x)$. By \eqref{smooth_argument_phi_bound} with $j=\ell-1$, the expression above is at most 
\[
\ll_\ell T^{\kappa+\eta}\|\phi^{(\ell-1)}\|_{\infty} \int_{\alpha-\delta}^{\sigma} x^{u+\ell-1} |u+iT|^{-\ell+1} du \ll_{\ell,\sigma,x} T^{\kappa+1+\eta-\ell}. 
\]
Since $\ell \geq \lceil \kappa \rceil +3$ by assumption, the exponent is negative and hence the estimate tends to zero as $T \to \infty$, as required. This completes the proof of Lemma \ref{lemma_Mellin_inversion}.  

\section{Hypothesis \ref{hyp:Strong}: Proof of \cref{thm:Strong}}\label{sec_B_proof}

\subsection{Preparing the smooth weights} 
 
To prove Theorem \ref{thm:Strong}, we will use the approach described in \S \ref{sec_B_prelim} of smoothing the sum with $\varphi$, an approximation of the indicator function $\mathbf{1}_{[0,x]}$ on the interval $[0,x]$. The crux of the problem is to make a careful choice of weights $\varphi$ in Lemma \ref{lemma_Mellin_inversion} that will allow us to recover the same residue as in Theorem \ref{thm:Weak} and produce a strong power-saving remainder term. We first establish basic properties about indicator functions, with the Mellin transform as defined in (\ref{Mellin_dfn2}).

  \begin{lemnum} \label{lem:indicator}
      Let $0 < a < b < \infty$. The Mellin transform $\chi(s) := \widehat{\mathbf{1}}_{[a,b]}(s)$ of the indicator function $\mathbf{1}_{[a,b]}$ is an entire function satisfying:
      \begin{enumerate}[label=(\alph*)]
          \item $\chi(s) = (b^s-a^s)/s$ for $s \in \C^{\times}$ and $\chi(0) = \log(b/a)$.
          \item $|\chi(s)| \leq \max\{ a^{\Re(s)}, b^{\Re(s)} \}  \cdot \min\{ \log (b/a), 2/|s|\}$ for $s \in \C$.
      \end{enumerate}
  \end{lemnum}
  \begin{remark}
      The singularity of  $(b^s-a^s)/s$ at $s=0$ is removable as shown by (a). Instead of writing $\chi(s)$ in this piecewise manner, we will write $\chi(s) = (b^s-a^s)/s$ for all $s \in \C$, including at its removable singularity $s=0$. We will always adhere to this convention with removable singularities for any   function of a complex variable. 
  \end{remark}
  \begin{proof}
      (a) Note $\chi(s) = \int_a^b u^{s-1} du = (b^s-a^s)/s$ for $s \neq 0$ and $\chi(0) = \int_a^b u^{-1} du = \log(b/a)$.  

      (b) Thus, by the triangle inequality, $|\chi(s)| \leq (a^{\Re(s)} + b^{\Re(s)})/|s| \leq \max\{ a^{\Re(s)}, b^{\Re(s)} \} \cdot (2/|s|)$ for $s \neq 0$. The bound also holds when $s=0$ by interpreting $2/|s| = \infty$. On the other hand, $|\chi(s)| \leq \int_a^b u^{\Re(s)-1} du \leq \max\{ a^{\Re(s)}, b^{\Re(s)} \} \int_a^b u^{-1} du = \max\{ a^{\Re(s)}, b^{\Re(s)} \} \log(b/a)$ for $s \in \C$. Combining both estimates proves (b).
  \end{proof}

 We are ready to introduce our desired weights and derive their key properties.   These constructions approximate the indicator function $\mathbf{1}_{[y,x]}$ and are common in many analytic problems; our inspiration comes from some examples familiar to us, such as Weiss \cite[Lemma 3.2]{Wei83}, Zaman \cite[Lemma 2.6]{Zam16}, and Thorner--Zaman \cite[Lemma 2.2]{TZ19}.

\begin{lemnum} \label{lem:weight} 
Fix an integer $\ell \geq 1$, and $0<\ep<1$.  Fix real numbers $x,y > 0$ with $x>e^{4\ep}y$. There exist compactly supported $(\ell-1)$-times continuously differentiable functions \[\varphi^{+} :  (0,\infty) \mapsto [0,1], \qquad \varphi^{-} :  (0,\infty) \mapsto [0,1]
\]
with the following properties:
	 \begin{enumerate}[(i)]
		\item \label{weight_ppty_maj_min} The functions $\varphi^-$ and $\varphi^+$ are a minorant and a majorant, respectively, of the characteristic function on $[y,x]$; that is, for all $u \geq 0$, 
		\[ \varphi^- (u)\leq \mathbf{1}_{[y,x]}(u) \leq \varphi^+(u) . \]
	Moreover,  we have that   $\phi^+ \con 1$ on $[y,x]$ and $\varphi^+$ is supported in $[e^{-2\ep}y,e^{2\ep}x]$; and also we have that $\phi^-\con 1$ on $[e^{2\ep} y, e^{-2\ep}x]$ and $\varphi^-$ is supported on $[y,x]$.
     
        \item \label{weight_ppty_Mellin} 
        The Mellin transforms $\hat{\varphi}^-(s)$ and $\hat{\varphi}^+(s)$ are entire. 
	
		\item\label{weight_ppty_bound} 
		For all $s$ with $\Re(s) \geq 0,$
		\begin{align*}
|\hat{\varphi}^+(s)| & \leq (xe^{2\epsilon})^{\Re(s)} \cdot \min\Big\{ \frac{2}{|s|}, \log(e^{2\epsilon}x/y) \Big\} \cdot \min \Big\{ 1, \Big(\frac{\ell}{\epsilon |s|} \Big)^{\ell} \Big\}, \\
		|\hat{\varphi}^-(s)| & \leq x^{\Re(s)} \cdot \min\Big\{ \frac{2}{|s|}, \log(e^{-2\epsilon}x/y) \Big\} \cdot \min \Big\{ 1, \Big(\frac{\ell}{\epsilon |s|} \Big)^{\ell} \Big\}.
\end{align*}

		\item \label{weight_ppty_approx_id} Define  $\psi^{\pm} := \varphi^{\pm} - \mathbf{1}_{[y,x]}$. 
        The Mellin transforms of $\psi^{+},\psi^{-}$ are entire functions  such that for $s \in \C$,
		\EQN{
		\hat{\varphi}^{\pm}(s) & =  \frac{x^s-y^s}{s} + \hat{\psi}^\pm(s)  
		}
		and 
    \[
    |\hat{\psi}^{\pm}(s)| \leq 2 \epsilon e^{2\epsilon |\Re(s)|} (x^{\Re(s)} +y^{\Re(s)}).  
    \]
	\end{enumerate}
\end{lemnum} 
\begin{remark} By (i),   $\varphi^{\pm}$ can be extended to be defined on $[0,\infty)$ by setting $\varphi^{\pm}(0) =0$. All of the other properties will remain since $\varphi^{\pm} \equiv 0$ on the interval $[0, e^{-2\epsilon} y]$. 
\end{remark}
\begin{remark}
    Estimate (iv) will be applied when $|s-\al| \ll 1/\log x$, which will imply $|\widehat{\psi}^{\pm}(s)| \ll \epsilon x^{\alpha}$. This bound can be interpreted as the error between the residue at $s=\alpha$ for the chosen weights $\varphi^{\pm}$ compared to the sharp cutoff $\mathbf{1}_{[y,x]}$. This error will be acceptable as we will ultimately choose $\epsilon \ll x^{-\delta/(\kappa+1)}$. 
\end{remark}
\begin{proof}
	We begin by defining three sharp cut-off weights denoted $\omega, \phi_0^+$ and $\phi_0^-$ on $(0,\infty)$. In each case the Mellin transform may be evaluated explicitly via \cref{lem:indicator}, and we list it next to the definition of the weight: 
    \begin{equation} \label{weight_step_functions}
	\begin{aligned}
	\omega(t) & = \frac{\ell}{2\epsilon}\mathbf{1}_{[e^{-\epsilon/\ell},  e^{\epsilon/\ell} ]}(t), \qquad 
	    	\hat{\omega}(s) = \frac{e^{\epsilon s/\ell} - e^{-\epsilon s/\ell}}{2  \epsilon s/\ell},\\
	\varphi_0^{+}(t) & = \mathbf{1}_{[e^{-\epsilon}y, e^{\epsilon}x]}(t), \qquad 
	    	\hat{\varphi}_0^+(s) = \frac{(e^{\epsilon} x)^s- (e^{-\epsilon}y)^s}{s},\\
	\varphi_0^-(t) &  = \mathbf{1}_{[e^{\epsilon}y,e^{-\epsilon}x]}(t), \qquad 
		\hat{\varphi}_0^-(s) = \frac{(e^{-\epsilon} x)^s- (e^{\epsilon} y)^s}{s}.
	\end{aligned}
    \end{equation}
	For each integer $j \geq 1$, define the functions $\varphi_j^{\pm}$ iteratively to be the  multiplicative convolution of  $\varphi_{j-1}^{\pm}$ and  $\omega$; that is, 
\[
	\varphi_j^{\pm}(t) = (\varphi_{j-1}^{\pm} \ast \omega)(t) = \int_0^{\infty} \varphi_{j-1}^{\pm}(u) \omega(t/u) \frac{du}{u} = \int_0^{\infty} \varphi_{j-1}^{\pm}(t/u) \omega(u) \frac{du}{u}.
	\]
    For all $j \geq 0$, we verify that $0 \leq \phi_j^\pm(u) \leq 1$ for all $u$. The case $j=0$ is immediate. For $j \geq 1$, this follows by induction since $\phi_j^{\pm}$ is an integral of  non-negative functions and since
\[
\varphi_j^{\pm}(t) =\int_0^{\infty} \varphi_{j-1}^{\pm}(t/u) \omega(u) \frac{du}{u} \leq \int_0^{\infty} \omega(u) \frac{du}{u} = 1. 
\]
For each $j\geq 1$, we verify that $\varphi_j^{\pm}$ is $C^{j-1}$. By \cref{lem:convolution}(v), $\varphi^{\pm}_1$ is continuous and hence $C^0$, establishing $j=1$. By induction and \cref{lem:convolution}(v) again, the functions $\varphi^{\pm}_j = \varphi^{\pm}_{j-1} \ast \omega$ are $C^{j-1}$ for $j \geq 1$. Additionally, for all $j \geq 0$, we will establish by induction that 
\begin{equation} \label{weight_support}
\begin{aligned}
    \mathrm{supp}(\varphi_j^{+}) &\subseteq [e^{-\epsilon - j\epsilon/\ell} y, e^{\epsilon + j\epsilon/\ell} x], 
    \qquad 
    \varphi_j^+ \equiv 1 \text{ on } [e^{-\epsilon + j \epsilon/\ell} y, e^{\epsilon -j \epsilon/\ell} x],\\
   \mathrm{supp}(\varphi_j^-) & \subseteq  
    [e^{\epsilon - j\epsilon/\ell}y, e^{-\epsilon + j\epsilon/\ell} x], 
    \qquad 
    \varphi_j^- \equiv 1 \text{ on } [e^{\epsilon + j \epsilon/\ell} y, e^{-\epsilon -j \epsilon/\ell} x].
\end{aligned}    
\end{equation}
Indeed, the case $j=0$ is by definition \eqref{weight_step_functions}. The support claims for $j \geq 1$ follow inductively by \cref{lem:convolution}(ii) and \eqref{weight_step_functions}. To establish the claims $\varphi^{\pm}_j \equiv 1$   for $j \geq 1$, observe 
\[
\varphi_j^{\pm}(t) =\int_0^{\infty} \varphi_{j-1}^{\pm}(u) \omega(t/u) \frac{du}{u} = \frac{\ell}{2\epsilon} \int_{t e^{-\ep/\ell}}^{t e^{\ep/\ell}} \varphi_{j-1}^{\pm}(u) \frac{du}{u}
\]
for any $t \in (0,\infty)$. If $t \in [e^{\mp \epsilon + j\epsilon /\ell} y, e^{\pm \epsilon - j \epsilon /\ell}x]$ then $[t e^{-\epsilon/\ell}, t e^{\epsilon/\ell}] \subseteq [e^{\mp \epsilon + (j-1)\epsilon /\ell} y, e^{\pm \epsilon - (j-1) \epsilon /\ell}x]$ so by induction $\varphi_{j-1}^{\pm} \equiv 1$ on the interval $[t e^{-\epsilon/\ell}, t e^{\epsilon/\ell}]$ and the above integral evaluates to $1$. This completes the induction and establishes \eqref{weight_support}. 

This proves that $\varphi^{\pm}_j$ are bounded, compactly supported, and continuous for all $j \geq 1$. Thus, by \cref{lem:convolution}(iv), their Mellin transforms $\widehat{\varphi}^{\pm}_j(s)$ are entire and satisfy
\begin{equation} \label{weight_mellin}
    \widehat{\varphi}_j^{\pm}(s) =  \widehat{\varphi}^{\pm}_0(s) \big[ \widehat{\omega}(s)]^{j}
\end{equation}
by induction on $j \geq 0$, as $\varphi_j^{\pm} = \varphi_{j-1} \ast \omega$. 

Now define  $\varphi^{\pm} := \varphi^{\pm}_{\ell}$ as our candidate weights. Our combined observations above have already shown that $\varphi^{\pm} : (0,\infty) \to [0,1]$ is a compactly supported $C^{\ell-1}$ function. We verify (i)--(iv). 

(i) We have already proved the claims in \eqref{weight_support} with $j=\ell$. The identity $\varphi^- \leq \mathbf{1}_{[y,x]} \leq \varphi^+$ follows directly from these properties and the fact that $0 \leq \varphi^{\pm} \leq 1$.

(ii) We have already established $\widehat{\varphi}^{\pm}(s)$ is entire when verifying  \eqref{weight_mellin} for the case $j=\ell$. 

(iii) By \cref{lem:indicator}(b) and \eqref{weight_step_functions}, we find for all $s \in \C$ that
\begin{equation} \label{weight_step_bounds}
\begin{aligned} 
|\widehat{\varphi}^{\pm}_0(s)| 
& \leq \max\{ (e^{\pm \epsilon} x)^{\Re(s)}, (e^{\mp \epsilon} y)^{\Re(s)} \} \cdot \min\Big\{ \log(e^{\pm 2\epsilon} x/y) , \frac{2}{|s|} \Big\}, \\
|\widehat{\omega}(s)| & \leq \max\{ (e^{\epsilon/\ell})^{\Re(s)}, (e^{-\epsilon/\ell})^{\Re(s)} \} \cdot \min\Big\{ 1, \frac{\ell}{\epsilon|s|} \Big\}.
\end{aligned}
\end{equation}
For $s \in \C$ with $\Re(s) \geq 0$, these inequalities imply (upon recalling $x> e^{4\ep}y$):
\begin{align*}
|\widehat{\varphi}^{\pm}_0(s)| 
& \leq   (e^{\pm \epsilon} x)^{\Re(s)}  \cdot \min\Big\{ \log(e^{\pm 2\epsilon} x/y) , \frac{2}{|s|} \Big\}, \\
|\widehat{\omega}(s)| & \leq    (e^{\epsilon/\ell})^{\Re(s)} \cdot \min\Big\{ 1, \frac{\ell}{\epsilon|s|} \Big\}. 
\end{align*}
Inserting these bounds in \eqref{weight_mellin} with $j=\ell$ gives the desired estimate. 

(iv) Recall $\psi^{\pm} := \varphi^{\pm} - \mathbf{1}_{[y,x]}$.  The functions $\psi^{\pm}(u)$ are continuous for all $u \in (0,\infty)$ except possibly $u \in \{y,x\}$ since $\varphi^{\pm}$ and $\mathbf{1}_{[y,x]}$ enjoy the same property. Also, since $0 \leq \varphi^{\pm} \leq 1$, it follows that $-1 \leq \psi^{\pm} \leq 1$ and hence $\psi^{\pm}$ are bounded. By (i), $\psi^+$ is supported in $[e^{-2\epsilon} y, y] \cup [x,e^{2\epsilon} x]$ and $\psi^-$ is supported in $[y, e^{2\epsilon} y] \cup [e^{-2\epsilon} x, x]$. Thus, in either case we may write the associated Mellin transform as
\begin{equation} \label{weight_penultimate}
\widehat{\psi}^{\pm}(s) = \int_0^{\infty} \psi^{\pm}(u) u^{s-1} du
\end{equation}
for all $s \in \C$, since the integrand is compactly supported away from the origin, continuous everywhere except 2 points, and bounded. Since $\psi^{\pm} = \varphi^{\pm} - \mathbf{1}_{[y,x]}$ is a linear combination of compactly supported, bounded, and integrable functions, it follows by linearity of the integral that 
\[
\widehat{\psi}^{\pm}(s) = \int_0^{\infty} (\varphi^{\pm}(u) - \mathbf{1}_{[y,x]}(u) )u^{s-1} du = \widehat{\varphi}^{\pm}(s) - \widehat{\mathbf{1}}_{[y,x]}(s)  =  \widehat{\varphi}^{\pm}(s) - \frac{x^s-y^s}{s}
\]
for $s \in \C$. Thus, by (ii) and  \cref{lem:indicator}(a), the function $\widehat{\psi}^{\pm}(s)$ is entire. It remains to estimate $|\widehat{\psi}^{\pm}(s)|$. From our earlier observations leading to \eqref{weight_penultimate}, we deduce for all $s \in \C$ that
\begin{align*}
|\widehat{\psi}^{+}(s) | \leq \int_0^{\infty} |\psi^{+}(u)| u^{\Re(s)-1} du 
& \leq \int_{e^{-2\epsilon} y}^{y} u^{\Re(s)-1} du + \int_{x}^{e^{2\epsilon} x} u^{\Re(s)-1} du  \\
& = \widehat{\mathbf{1}}_{[e^{-2\epsilon}y,y]}(\Re(s)) + \widehat{\mathbf{1}}_{[x, e^{2\epsilon}x]}(\Re(s)) 
\end{align*}
 by the definition of the Mellin transform for indicator functions. Therefore, by \cref{lem:indicator}(b), the above yields
 \begin{align*}
 |\widehat{\psi}^{+}(s) | & \leq 2\epsilon \cdot \max\{ (e^{-2\epsilon}y)^{\Re(s)}, y^{\Re(s)} \} + 2\epsilon \cdot \max\{ x^{\Re(s)}, (e^{2\epsilon}x)^{\Re(s)} \} \\
  & \leq 2 \epsilon e^{2\epsilon |\Re(s)|} (x^{\Re(s)} +y^{\Re(s)}). 
 \end{align*}
The same bound can be similarly deduced for $\widehat{\psi}^-$. This completes the proof. 

\end{proof}

\subsection{Reduction from $[0,x]$ to $[y,x]$}   Before utilizing the smooth weights from \cref{lem:weight}, we replace the  sharp cut-off $\mathbf{1}_{[0,x]}$ in \cref{thm:Strong} with another sharp cut-off $\mathbf{1}_{[y,x]}$ for $y > 0$. The latter is compactly supported away from the origin, which will be preferable in light of \cref{lem:mellin_cpt_cts} and the smooth weights from \cref{lem:weight}. 

Assume $A(s)$ satisfies Hypothesis \ref{hyp:Strong} for $\al,\del,\kappa,m$. We claim it suffices to establish the following estimate: for all $x, y > 0$ with $x > e^4 y$, 
\beq\label{dyadic_version}
\sum_{y < \lambda_n \leq x} a_n =  \Res_{s=\alpha} \Big[ A(s) \frac{x^s-y^s}{s} \Big] + O\Big(   x^{\alpha - \frac{\delta}{\kappa+1}} (\log x)^{m-1} \Big).
\eeq

\begin{proof}[Proof of \cref{thm:Strong} assuming \eqref{dyadic_version}] Since $\lambda_1 > 0$, observe that 
\[
\sum_{y < \lambda_n \leq x} a_n = \sum_{\lambda_n \leq x} a_n \quad \text{ for } 0 < y < \lambda_1. 
\]
Since the remainder term in \eqref{dyadic_version} does not depend on $y$, by linearity of residues it therefore suffices to prove  that 
\[
\Res_{s=\alpha} \Big[ A(s) \frac{y^s}{s} \Big] \longrightarrow  0  \quad \text{ as } y \to 0^+. 
\]
By Hypothesis \ref{hyp:Strong}, $A(s)$ has a pole of order $m \geq 1$ at $s=\al$, so
\[
\Res_{s=\alpha} \Big[ A(s) \frac{y^s}{s} \Big] = y^{\alpha} P_{m-1}(\log y),
\]
where $P_{m-1}(\log y)$ is a polynomial of degree $m-1$ in $\log y$ (see \cref{remark_residue} for details). As $\alpha > 0$ is fixed, it follows by routine  calculus limits that 
\[
y^{\alpha} P_{m-1}(\log y) \longrightarrow 0 \quad \text{ as } y \to 0^+. 
\]
This yields \cref{thm:Strong} assuming \eqref{dyadic_version} holds. 

\end{proof}

\subsection{Application of the smooth weights in the Mellin inversion argument} 
With the weights $\phi^+(u)$ and $\phi^-(u)$ from Lemma \ref{lem:weight} in hand, we are ready to establish \eqref{dyadic_version} and hence conclude the proof of Theorem \ref{thm:Strong}. Assume \cref{hyp:Strong} holds with $\alpha,\delta,\kappa,$ and $m$. Let $x, y > 0$ be arbitrary with $x > e^4 y$. Let $0 < \varepsilon < 1$ be a real number and fix the integer $\ell = \lceil \kappa \rceil +  3$. Define $\varphi^{\pm}$ to be the weights constructed in \cref{lem:weight}; these weights depend on $x,y,\ep,\ell$ and notice $x > e^4 y > e^{4\ep} y$. We will later choose $\ep$ in (\ref{T_ep_choices}) to depend on $x$ in an optimal way.

Both weights $\varphi^{\pm}$ from Lemma \ref{lem:weight} satisfy the conditions of Lemma \ref{lemma_Mellin_inversion}.
Therefore,
\EQNlabel{\label{eqn:ShiftedContour}}{
\sum_{n=1}^{\infty}  a_n \varphi^{\pm}(\lambda_n) = \Res_{s=\alpha} [  A(s) \hat{\varphi}^{\pm}(s) ] + \frac{1}{2\pi i} \int_{(\alpha-\delta)} A(s)  \hat{\varphi}^{\pm}(s) ds. 
}
We first remove the dependence of the residue on the functions $\hat{\varphi}^{\pm}$. Let $\psi^\pm(s)$ be the entire   function   given in \cref{lem:weight}\ref{weight_ppty_approx_id}, so that 
\begin{equation} 
    \label{eqn:ResidueUnsmoothed}
\Res_{s=\alpha}\left[A(s)\hat{\varphi}^{\pm}(s)\right] = \Res_{s=\alpha}\left[A(s)\frac{x^s-y^s}{s}\right] +  \Res_{s=\alpha}\left[A(s)\psi^\pm(s)\right].
\end{equation}
We estimate the second residue on the right-hand side. By Lemma \ref{lemma_convexity} (with $\eta = \delta$) and  \cref{lem:weight}\ref{weight_ppty_approx_id}, we have that
\beq\label{A_bounded_above_logx}
|A(s)|\ll \frac{1}{|s-\al|^m} \quad \text{ and } \quad |\widehat{\psi}^{\pm}(s)| \ll \epsilon x^{\Re(s)} \quad \text{ for } |s-\alpha| \leq \frac{\delta}{2}.
\eeq  
 We emphasize that the implied constants $\delta$ and $A(\alpha+\delta)$ from Lemma \ref{lemma_convexity} are suppressed by our  convention stated after \Cref{thm:Strong}. 
Also, our application of \cref{lem:weight}\ref{weight_ppty_approx_id} used that $\epsilon \in (0,1)$ and $\Re(s) \geq 0$ inside the disk $|s-\alpha| \leq \frac{\delta}{2}$ in order to absorb the factors of $e^{2\epsilon |\Re(s)|}$ and $y^{\Re(s)}$. The circle $|s-\alpha| = r := \min\{ \frac{\delta}{2},  \frac{1}{\log x}\}$ lies inside this region. Thus, by the residue theorem and holomorphy of $\psi^{\pm}(s)$, we may express the second residue on the right-hand side of (\ref{eqn:ResidueUnsmoothed}) as
\begin{equation} \label{psi_contribution_epsilon}
\begin{split}
\Res_{s=\alpha}\left[A(s)\psi^\pm(s)\right] 
& = \frac{1}{2\pi i} \oint_{|s-\alpha|=r} A(s) \psi^\pm(s) ds \\
& \ll  \oint_{|s-\alpha|=r} \frac{\ep x^{\Re(s)}}{|s-\al|^m}\, |ds|   
 \ll \varepsilon x^\alpha (\log x)^{m-1},
\end{split}
\end{equation}
which follows by applying  the estimate \eqref{A_bounded_above_logx} as well as the identity $x^{1/\log x} =e$.  

Next, we estimate the integral in \eqref{eqn:ShiftedContour} on the line $\Re(s) = \alpha -\delta$ by dividing it into two pieces using  a height parameter $T  =  T(\epsilon,\ell) \geq 2$ (to be chosen momentarily), writing
\[ 
\int_{(\al-\del)}  A(s)  \hat{\varphi}^{\pm}(s) ds=
\int_{\substack{\Re(s)=\alpha - \delta \\ |\Im(s)| \leq T}}   +\int_{\substack{\Re(s)=\alpha - \delta \\ |\Im(s)| > T}}   .
\]
For both integrals, we will apply \cref{lem:weight}\ref{weight_ppty_bound} and \cref{hyp:Strong} on the line $\Re(s)=\al-\del$. 
For $\Re(s) = \alpha - \delta > 0$ the bounds in \ref{weight_ppty_bound} simplify to 
\begin{equation} \label{weight_simplified}
|\hat{\varphi}^{\pm}(s)|  \ll \frac{x^{\alpha-\delta}}{|s|}  \min \left\{ 1, \left(\frac{\ell}{\epsilon |s|} \right) \right\}^{\ell}.   
\end{equation}
Note the $e^{2\epsilon}$ factors are absorbed by the implied constant since $0 < \epsilon < 1$.  For the integral with $\Re(s) = \alpha-\delta > 0$ and $|\Im(s)| \leq T$, we apply these bounds (choosing the first option in the minimum) to obtain an estimate of the form 
$|\hat{\varphi}^{\pm}(s)| \ll \frac{x^{\alpha-\delta}}{1+|\Im(s)|}.$
By \cref{hyp:Strong}, this yields for the portion of bounded height:
\beq\label{bounded_height_portion}
\int_{\substack{\Re(s)=\alpha - \delta \\ |\Im(s)| \leq T}}    \ll  x^{\alpha-\delta} \int_{|t| \leq T} (1+|t|)^{\kappa-1} (\log(3+|t|))^{m-1} dt \ll  x^{\alpha-\delta} T^{\kappa} (\log T)^{m-1}. 
\eeq
The line immediately above is the only place in the proof of Theorem \ref{thm:Strong} where we apply the assumption $\kappa>0$; for $\kappa=0$ see Remark \ref{remark_kappa_zero}.

For the integral with $\Re(s) = \alpha-\delta > 0$ and $|\Im(s)| > T$, we apply \eqref{weight_simplified} to obtain an estimate of the form $|\hat{\varphi}^{\pm}(s)| \ll \frac{x^{\alpha-\delta}}{(1+|\Im(s)|)^{\ell+1}} (\ell/\epsilon)^{\ell}$. Thus, by \cref{hyp:Strong}, the portion of unbounded height satisfies
\beq\label{unbounded_height_portion}
\int_{\substack{\Re(s)=\alpha - \delta \\ |\Im(s)| > T}}  
	 \ll x^{\alpha-\delta} (\ell/\epsilon)^{\ell} \int_{|t| > T} (1+|t|)^{\kappa-\ell-1} (\log(3+|t|))^{m-1} dt \ll  x^{\alpha-\delta} (\ell/\epsilon)^{\ell} T^{\kappa-\ell} (\log T)^{m-1}, 
\eeq
since $\ell = \lceil \kappa \rceil + 3 > \kappa$. It follows from these conclusions and (\ref{psi_contribution_epsilon}) that
\begin{multline}\label{total_remainder}
\Res_{s=\alpha}\left[A(s)\psi^{\pm}(s)\right] +
 \frac{1}{2\pi i}\int_{(\al-\del)}  A(s)  \hat{\varphi}^{\pm}(s) ds 
\\
\ll \epsilon x^{\alpha} (\log x)^{m-1}  +  T^{\kappa} (\log T)^{m-1} x^{\alpha-\delta}  +  (\ell/\epsilon)^{\ell} T^{\kappa-\ell} (\log T)^{m-1} x^{\alpha-\delta}  . 
\end{multline}
The selections
\beq\label{T_ep_choices}
  \qquad T = \ell/\epsilon, \qquad  \epsilon = \min\left\{ \tfrac{1}{10},\, x^{-\frac{\delta}{\kappa+1}}  \right\} 
\eeq
minimize the remainder. Combining these observations with  \eqref{eqn:ResidueUnsmoothed}   in \eqref{eqn:ShiftedContour}, we conclude that 
\beq\label{hyp:Strong_nonnegative}
\sum_{n=1}^{\infty}  a_n \varphi^{\pm}(\lambda_n) = \Res_{s=\alpha} \Big[  A(s) \frac{x^s-y^s}{s}\Big]  + O\big( x^{\alpha - \frac{\delta}{\kappa+1}} (\log x)^{m-1} \big).
\eeq
Recall that $0 \leq \varphi^-(u) \leq \mathbf{1}_{[y,x]}(u) \leq \varphi^+(u) \leq 1$ for all $u \geq 0$ by \cref{lem:weight}\ref{weight_ppty_maj_min}. Since the coefficients $a_n$ are non-negative, we may squeeze the original sum by 
\[ \sum_{n=1}^{\infty}  a_n \varphi^{-}(\lambda_n)  \leq \sum_{y< \lam_n \leq x}  a_n    \leq \sum_{n=1}^{\infty}  a_n \varphi^{+}(\lambda_n) 
\]
and then obtain the desired result (\ref{dyadic_version}) from two applications of (\ref{hyp:Strong_nonnegative}).   This completes the proof of Theorem \ref{thm:Strong}.

  \begin{remark}\label{remark_kappa_zero}[The case $\kappa=0$]
 When $\kappa=0$, (\ref{bounded_height_portion}) is instead estimated by $\ll x^{\al-\del}(\log T)^m$ while the estimate (\ref{unbounded_height_portion}) remains valid. In this case we still choose $T$ and $\epsilon$ as in \eqref{T_ep_choices}, and then the lefthand side of  (\ref{total_remainder})   is dominated by 
 \[
 \ll \ep x^{\al} (\log x)^{m-1} + x^{\al-\del}(\log (\ell/\ep))^{m} \ll x^{\alpha-\delta} (\log x)^m.
 \]
 This leads to a version of Theorem \ref{thm:Strong} for $\kappa=0$, with remainder term $O(x^{\al-\del}(\log x)^m)$.
 \end{remark}

 \begin{remark}
    We emphasize that all implied constants, notably in \eqref{A_bounded_above_logx}, depend only on the constants $\alpha,\delta, \kappa, m, A(\al+\del)$ and $C$ appearing in \cref{hyp:Strong}, and do not depend on the Dirichlet series $A(s)$ in any other way. This observation explains why we rely on \cref{lemma_convexity} to estimate the second residue appearing in  \eqref{eqn:ResidueUnsmoothed} and \eqref{psi_contribution_epsilon}. 
\end{remark} 

\begin{remark}[Hypothesis \ref{hyp:Strong}, and non-negativity] \label{remark_non-negativity}
We remark on the hypothesis that the coefficients $a_n$ are non-negative, in the proof of Theorem \ref{thm:Strong}. We   used this fact when we compared the expression  (\ref{hyp:Strong_nonnegative}) for the weight $\phi^+$ versus the weight $\phi^-$, in order to squeeze the original partial sum between these expressions.  This positivity crucially allows us to handle the ``short sum problem'' of estimating sums like $\sum_{x-h < \lambda_n \leq x} a_n$  where $h = x^{\theta}$ for some fixed $0 < \theta < 1$. Indeed, by \cref{lem:weight}(i) and \eqref{hyp:Strong_nonnegative}, non-negativity implies that
\[
 \sum_{y < n < e^{2\epsilon} y} a_n +  \sum_{e^{-2\epsilon} x < n < x} a_n  \leq \sum_{n=1}^{\infty} a_n \varphi^+(\lambda_n) - \sum_{n=1}^{\infty} a_n \varphi^-(\lambda_n) \ll \epsilon x^{\alpha} (\log x)^{m-1} 
\]
for our choice of $\epsilon$ in \eqref{T_ep_choices} or any larger value of $\epsilon$. By taking $\epsilon = h/x$, it follows that 
\[
\sum_{x-h < n \leq x} a_n \ll h x^{\alpha-1} (\log x)^{m-1}
\]
for $x^{1-\frac{\delta}{\kappa+1}} \leq h \leq x/10$. (One can equivalently prove this estimate by applying \cref{thm:Strong} twice with  $\lambda_n \leq x$ and $\lambda_n \leq x-h$, and then taking the difference.) This short sum estimate can therefore be viewed as a key step in the proof of \cref{thm:Strong}. 

\end{remark}

\subsection{Proof of \cref{cor:Pointwise}} 
For any $N \geq 1$, we have 
\[
a_N = \lim_{\epsilon \to 0^+} \Big(\sum_{\lambda_n \leq \lambda_N} a_n - \sum_{\lambda_n \leq \lambda_N - \epsilon} a_n \Big).
\]
By \cref{thm:Strong}, it follows for any $\epsilon > 0$ sufficiently small that
\[
\sum_{\lambda_n \leq \lambda_N} a_n - \sum_{\lambda_n \leq \lambda_N - \epsilon} a_n = \Res_{s=\alpha}\Big[ A(s) \frac{\lambda_N^s}{s} \Big] - \Res_{s=\alpha}\Big[ A(s) \frac{(\lambda_N-\epsilon)^s}{s} \Big]  + O\Big( \lambda_N^{\alpha - \frac{\delta}{\kappa+1} } (\log \lambda_N)^{m-1}\Big). 
\]
Since the error term does not depend on $\epsilon > 0$, it suffices to show the difference of residues tends to zero as $\epsilon \to 0^+$. 
From \cref{remark_residue}, the difference of residues is equal to 
\[
\lambda_N^{\alpha} P_{m-1}(\log \lambda_N) - (\lambda_N-\epsilon)^{\alpha} P_{m-1}(\log(\lambda_N-\epsilon)) 
\]
for some polynomial $P_{m-1}$ of degree $m-1$. Taking $\epsilon \to 0^+$, this difference tends to zero by continuity. Combining our observations establishes the corollary.

\begin{remark}\label{remark_CLT}
We pause to remark on \cite[Appendix A, Thm. A.1]{CLT01}, which claims that under Hypothesis \ref{hyp:Strong} for a given $\al,\del,\kappa,m$, the conclusion of Theorem \ref{thm:Strong} holds but with the better remainder term $O(x^{\al-\del})$, independent of the growth parameter $\kappa$. By Theorems \ref{thm:Example} and \ref{thm:Example_bounded}, this general claim cannot be true. Note that the remainder term stated in \cite[Thm. A.1]{CLT01} does not affect the main theorems in \cite{CLT01}, which are either stated as asymptotics with no specified remainder, or as an asymptotic with a power-saving remainder term, but no precise exponent giving for the savings (see \cite[Cor. 4.4.8, Cor. 5.2.6]{CLT01}).

Let us explain a gap in the proof of Theorem A.1 in that work. In their notation, $\Theta := \lim_{s \maps \al}(s-\al)^mA(s)>0.$ The statement of \cite[Lemma A.2]{CLT01}  shows correctly that under Hypothesis \ref{hyp:Strong} for $\al,\del,\kappa,m$, there exists for each integer $k > \kappa $ a polynomial $Q_k$ of degree $m-1$ with leading coefficient $k! \Theta /(\al^{k+1}(m-1)!)$ such that for every $\mu< \del$, 
\beq\label{CLT_Phi}
\Phi_k(x) :=\sum_{\lam_n \leq x}a_n (\log (x/\lam_n))^k = x^\al Q_k(\log x) + O(x^{\al - \mu}).
\eeq
 The factor $(\log (x/\lam_n))^k$ acts like a smooth weight (see \cite[\S 5.1.2]{MonVau07}), so that this is a reasonable remainder term for this smoothed sum (compare to (\ref{smooth_argument_conclusion_rescaled})).
The goal is then to show that a statement similar to this also holds for $k=0$. Beginning with $k>\kappa$ it suffices to prove that if a statement of the form (\ref{CLT_Phi}) holds for $k$ with a savings $\del_k$ in the remainder term $O(x^{\al - \del_k})$, then it holds for $k-1$ with a savings $\del_{k-1}$, and proceed downward by induction. First, the argument correctly proves  a sandwich result: for each $\eta \in (0,1),$
\[ \frac{ \Phi_k(x(1-\eta)) - \Phi_k(x)}{\log (1-\eta)} \leq k \Phi_{k-1}(x) \leq  \frac{ \Phi_k(x(1+\eta)) - \Phi_k(x)}{\log (1+\eta)}.\]
By the true statement  (\ref{CLT_Phi}) for $k>\kappa$, 
\beq\label{CLT_induction}
|\Phi_k(x) - x^\al Q_k(\log x)| \ll x^{\al - \del_k},
\eeq
for some $\del_k < \del$. From this, one can deduce the following property, for sufficiently small $u \in (-1,1)$ (say of the form $u=\pm x^{-\ep}$ for some small $\ep>0$):
\[ \frac{\Phi_k(x(1+u)) - \Phi_k(x)}{\log (1+u)}  = x^\al [ \al Q_k(\log x) +  Q'_k(\log x)]+ O(x^\al (\log x)^{m-2}u +x^{\al - \del_k}/u).\]
We wish to take $u$ to be   $u=\pm x^{-\ep}$ for some small $\ep>0$ and then apply this in the sandwich result; however we must check that we obtain savings in both parts of the remainder term. Balancing the remainder terms, say $x^{\al -\ep} \approx x^{\al - \del_k +\ep}$, shows we should take $\ep = \del_k/2$. As a result, we can deduce (\ref{CLT_induction}) for $k-1$ and $\del_{k-1}:=\del_k/2$ (up to factors of $\log x$) if it is known for $k$ and $\del_k$. (This is where the proof of \cite[Thm. A.2]{CLT01} must be corrected, since it seems to assume one could take $\del_{k-1} = \del_k$ at each step.) To begin the induction we start with an integer $k>\kappa$ as small as allowed, and then iterate until we reach $k=0$. Set  $K_0 = \kappa+1$ if $\kappa \in \Z$, and $K_0 = \lceil \kappa \rceil$ if $\kappa \not\in \Z$. The corrected argument proves a version of Theorem \ref{thm:Strong} with remainder term of size $O(x^{\al - \frac{\del}{2^{K_0}}+\ep})$, for any $\ep>0$.

\end{remark}

\section{Examples with Theorem \ref{thm:Strong}}\label{sec_B_corollaries}
In this section, we will call upon the standard convexity upper bound for the Riemann zeta function and for Dirichlet $L$-functions, which we record here, as in \cite[Thm. 1.9]{Ivi85} and \cite[Thm. 3]{Rad59} (see also \cite[Thm. 5.23]{IwaKow04}).  
\begin{lemnum}[Convexity bounds]\label{lemma_standard_convexity_zeta}
For all $s$ with $0 \leq \Re(s) \leq 1$,
 \[ \left|\frac{(s-1)}{(s+1)}\zeta(s)\right| \ll (1+|\Im(s)|)^{\frac{1}{2}(1-\Re(s))}\log (3+|\Im(s)|).\]
Let $\eta \in (0,1/2]$ be given. For each integer $q > 1$ and primitive Dirichlet character $\chi$ of conductor $q$, for all $s$ with $-1/2 \leq -\eta \leq \Re(s) \leq 1+\eta \leq 3/2$, 
 \[|L(s,\chi)| \leq \left( \frac{q|1+s|}{2\pi}\right)^{\frac{1}{2}(1-\Re(s)+\eta)}\zeta(1+\eta).\]

\end{lemnum}

\subsection{Notes on \cref{cor:Strong_1}}
 In \cite[Ex. 2.1]{Alb24xb}, Alberts shows that $A(s) = \prod_p (1+2p^{-s}) = \zeta(s)^2 H(s)$
 for 
 \[ H(s) = \prod_p (1-3p^{-2s}+2p^{-3s}).\]
The example \cite[Ex. 4.2]{Alb24xb} shows $H(s)$ is absolutely convergent for $\Re(s)>1/2$ (and hence uniformly bounded for $\Re(s) \geq 1/2 + \ep_0$ for any fixed $\ep_0>0$). 
 By the convexity bound in Lemma \ref{lemma_standard_convexity_zeta}, for any fixed $\del< 1/2$, on the line $\Re(s)=1-\del$, 
 \[ |\zeta^2(s)| \ll ((1+|\Im(s)|)^{\frac{1}{2} - \frac{1}{2}(1-\del) +\ep})^2 \ll (1+|\Im(s)|)^{\del +2\ep}
 \]
 for any $\ep>0$. Additionally, Lemma \ref{lemma_standard_convexity_zeta} shows that (\ref{PH_prep_hyp_B}) holds (with $M_2\leq 1$, say).
 Consequently $\zeta(s)^2$ satisfies Hypothesis \ref{hyp:Strong} with $\al=1$, $m=2$, any $\del< 1/2$,   and  any $\kappa > \del$; it suffices to take $\kappa=1/2$. It follows that $A(s)$ satisfies the hypothesis with the same parameters, and hence by Theorem \ref{thm:Strong}, 
 \[ \sum_{n \leq x} a_n  = \Res_{s=1} \left[ \frac{\zeta(s)^2 H(s) x^s}{s}\right] + O(x^{1 - 2\del/3+\ep})\]
 for any $\ep>0$.
 Recalling any $\del<1/2$ suffices, we can write the remainder term as $O(x^{2/3+\ep})$ for any $\ep>0$. 
To compute the residue, we may write $\zeta(s) = (s-1)^{-1} + h(s)$ for a function $h(s)$ that is holomorphic for $\Re(s) \geq 1$, since $s=1$ is the only pole of $\zeta(s)$ in this region. 
 The residue is 
 \begin{align*}
\Res_{s=1} \left[ \frac{\zeta(s)^2 H(s) x^s}{s}\right]
& = \lim_{s \maps 1}\left(\frac{d}{ds}\left[ (s-1)^2 \left[ ((s-1)^{-1}+ h(s))^2 \frac{H(s)}{s} x^s\right] \right]\right)\\
& =\lim_{s \maps 1}\left(\frac{d}{ds}\left[  \frac{H(s)}{s} x^s \right]\right) +2h(1)H(1)x \\
& = H(1)x \log x + [ H'(1)-H(1)+2h(1)H(1)]x.
 \end{align*}
This leads to the example, since $H(1)=\prod_p (1-3p^{-2}+2p^{-3}).$

\subsection{Notes on \cref{cor:Strong_2}}
 In \cite[Prop. 2.1]{PTBW20}, it is shown that 
 \[ A(s) = H(s)\prod_{\chi} L((d-1)s,\chi) ,\]
 in which $\chi$ varies over all Dirichlet characters modulo $d$ and $H(s)$ is holomorphic for $\Re(s) > \frac{1}{2(d-1)} = \frac{1}{d-1} - \frac{1}{2(d-1)}.$ For the trivial character $\chi_0$, $L(s,\chi_0)$ has a meromorphic extension to $\C$ with only a simple pole at $s=1$; for all $\chi \neq \chi_0$,  $L(s,\chi)$ has an analytic continuation to all of $\C$. Let $\phi(d)$ denote the Euler phi function. There are $\phi(d)$  many $L$-functions in the product over $\chi$, and each of these satisfies a convexity bound for $0 \leq \Re(s) \leq 1$ as in Lemma \ref{lemma_standard_convexity_zeta}, 
 so that for $s$ with  $0\leq \Re(s)\leq 1$,
 \[ \left| \frac{(s-1)}{(s+1)}L(s,\chi_0)\right| \ll_d (1+|\Im(s)|)^{\frac{1}{2}(1-\Re(s)) + \ep }\]
 for the trivial character,
 and 
 \[ |L(s,\chi)| \ll_d (1+|\Im(s)|)^{\frac{1}{2}(1-\Re(s)) + \ep }\]
 for any nontrivial character $\chi \neq \chi_0$,
 for any $\ep>0$. (Here we used the facts that for the trivial Dirichlet character $\chi_0$ modulo $d$, $L(s,\chi_0)=\zeta(s)\prod_{p |d}(1-p^{-s})$ as in \cite[Ch. 4 Eqn (6)]{Dav00}; and for $\chi$ a nontrivial imprimitive character, then $L(s,\chi) = L(s,\chi')\prod_{p|d}(1-\chi'(p)p^{-s})$ for a primitive character $\chi'$ with conductor $d'$ for some divisor $d'|d$, as in \cite[Ch. 5 Eqn (3)]{Dav00}.)
 Consequently, for $\Re(s) = \frac{1}{d-1} - \del$ for a given $\del< \frac{1}{2(d-1)}$,
  \[ \left|  \prod_\chi L((d-1)s,\chi)\right| \ll_d (1+|\Im(s)|)^{\phi(d)\frac{1}{2} (d-1)\del + \ep } .\]
Thus $A(s)$ satisfies Hypothesis \ref{hyp:Strong} with $\al=\frac{1}{d-1}$, $\del=\frac{1}{2(d-1)} -\ep'$, $m=1$, and  $\kappa = \varphi(d) \frac{1}{2} (d-1) \delta + \epsilon$ for any $\ep>0$ and $\ep' > 0$.   Simplifying this using the definition of $\del$, we can take $\kappa = \phi(d)/4 + \ep>0$ for any $\ep>0$.
Additionally, Lemma \ref{lemma_standard_convexity_zeta} confirms that (\ref{PH_prep_hyp_B}) holds (with $M_2 \leq \phi(d)/2,$ say).
Thus 
\[ \sum_{n \leq x} a_n = \Res_{s=1/(d-1)} \left[ \frac{A(s) x^s}{s}\right] + O(x^{\frac{1}{d-1} - \del'+\ep''})\]
for any $\ep''>0$, where $\del' = \frac{2}{(d-1)(\phi(d)+4)}.$ The residue is
\begin{align*}
  \Res_{s=1/(d-1)} \left[ \frac{A(s) x^s}{s}\right] & =  
\lim_{s \maps 1/(d-1)} \left[ \left(s-\frac{1}{d-1}\right)H(s) L((d-1)s,\chi_0) \prod_{\chi\neq \chi_0} L((d-1)s,\chi)\cdot \frac{x^s}{s}\right] \\
&= \lim_{s \maps 1} \left[ \left(\frac{s-1}{d-1}\right)H(\frac{s}{d-1}) L(s,\chi_0) \prod_{\chi\neq \chi_0} L(s,\chi)\cdot \frac{x^{s/(d-1)}}{s/(d-1)}\right]\\
& = H(\frac{1}{d-1})  \prod_{\chi\neq \chi_0} L(1,\chi) \cdot \Res_{s=1}[L(s,\chi_0)]\cdot  x^{1/(d-1)} .
 \end{align*}
Recalling that for the trivial Dirichlet character $\chi_0$ modulo $d$, $L(s,\chi_0)=\zeta(s)\prod_{p |d}(1-p^{-s})$,   the above expression defines 
 the constant \[c_d'= H(\frac{1}{d-1})  \prod_{\chi\neq \chi_0} L(1,\chi) \prod_{p|d}(1-p^{-1})\]
appearing in the example.

\section{Limiting counterexample: Proof of \cref{thm:Example}}\label{sec_example_unbounded}

 We now prove Theorem \ref{thm:Example}.
Fix $\alpha > 0, \delta \in (0,\alpha), m \in \Z^+$ and $\kappa > \frac{1}{2}$. All implied constants may depend on these parameters. For a real parameter $r > 1$, define the general Dirichlet series
\begin{equation} \label{eqn:DirichletSeries-A}
A(s)  =  \sum_{n=1}^{\infty} \frac{a_n}{\lambda_n^s}  := (-1)^{m-1}\zeta^{(m-1)}\Big(1 + \frac{r}{\delta}(s-\alpha)  \Big) = \sum_{n=1}^{\infty}  \frac{ (\log n)^{m-1} }{n^{1+ r (s-\alpha)/\delta}}.
\end{equation}
The parameter $r$ will be specified to control the growth of $A(s)$ along $\Re(s) = \alpha-\delta$. In particular, by assuming $r>1$, we assure that when $\Re(s)=\al-\del$ then the argument of $\zeta^{(m-1)}(z)$ for $z =1+(r/\del)(s-\al)$ lies in $\Re(z)<0$. For this purpose, we record a lemma about derivatives of the zeta function.
 
\begin{lemnum} \label{lem:Growth}
Let $k \geq 0$ be an integer, and let $z\in \mathbb{C}$  be such that $b\leq \Re(z) \leq a <0$ for $a, b$ fixed. Then, as $|\Im(z)|\to \infty$, 
	\[
	\left|\zeta^{(k)}(z)\right|  \ll_{k,a,b} \left(1+ |\Im(z)|\right)^{\frac{1}{2} - \Re(z)} \log^k(3+|\Im(z)|). 
	\]
\end{lemnum}
 \begin{proof} We may assume without loss of generality that $|\Im(z)| \geq 3$. 
We first consider the case $k=0$. Recall that the functional equation of $\zeta(s)$ is
\begin{equation}\label{eqn:funeq}
\zeta(s) = 2^s\pi^{s-1}\sin\left(\frac{\pi s}{2}\right)\Gamma(1-s)\zeta(1-s).
\end{equation}
Note that $|\sin(\pi s/2)|\ll e^{\pi|\Im(s)|/2}$ and that, for any $\varepsilon >0$, $|\zeta(1+\varepsilon +it)|\ll_{\varepsilon} 1$. Now consider $z\in \mathbb{C}$   such that $b\leq \Re(z) \leq a <0$ for the given $a, b$. Then by the functional equation,
\beq\label{zeta_needs_gamma}
|\zeta(z)|\ll_{a}e^{\pi |\Im(z)|/2}\,|\Gamma(1-z)|.
\eeq
We must verify for $|\Im(z)| \geq 3$ that 
\beq\label{Gamma_bound_in_strip}
|\Ga(1-z)| \ll_{a,b}|\Im(z)|^{1/2-\Re(z)}e^{-\pi |\Im(z)|/2}.
\eeq
In general, for $\del>0$ and $-\pi + \del \leq \arg s \leq \pi - \del$, Stirling's formula implies
 \[\log \Ga(s) = s\log s - s - \frac{1}{2} \log s + \log \sqrt{2\pi} + O_\del(|s|^{-1}),\]
 (see \cite[Appendix \S 3]{KarVor92}). Consequently,  in any vertical strip $0<\al \leq \Re(s) \leq \be$ and for $|\Im(s)| >1$, 
 \[  
 |\Ga(s)| \ll_{\al,\be} e^{(s-1/2)(\log|s|+i\arg{s})}\ll_{\al,\be} |s|^{\Re(s)-1/2}e^{-|\Im(s)|\arg{s}}\ll_{\al,\be}|\Im(s)|^{\Re(s)-1/2}e^{-\pi |\Im(s)|/2}.
 \]
Thus, for $b\leq \Re(z) \le a< 0$ and $|\Im(z)| \geq 3$, (\ref{Gamma_bound_in_strip}) holds. Hence we may conclude in (\ref{zeta_needs_gamma}) that
\begin{equation}\label{zetabound}
|\zeta(z)| \ll_{a,b} |\Im(z)|^{1/2-\Re(z)}.
\end{equation}
Note that this holds for any given $b<a<0$.
 
We turn to the proof of the lemma for $k \geq 1$. Given $b<a<0$, we may now bound  derivatives  $\zeta^{(k)}(z)$ for $b\leq \Re(z)\leq a<0$ by using the Cauchy integral formula and the bound (\ref{zetabound}) we have already obtained in the (wider) strip $2b \leq \Re(z) \leq  a/2<0$. Precisely,
\[
\zeta^{(k)}(z) = \frac{k!}{2\pi i}\int_{\mathcal{C}}\frac{\zeta(w)}{(w-z)^{k+1}}\,dw,
\]
where $\mathcal{C}$ is the compact circle $|s-z|=\rho$; we take 
\[
\rho = \min\Big\{\,|\Re(z)-a/2|\,,\, |\Re(z)-2b|\,, \frac{1}{\log(3 + |\Im(z)|)} \Big\}
\]
to ensure that any $w\in\mathcal{C}$ satisfies $2b\le \Re(w) \le a/2$, $|\Re(w) - \Re(z)| \leq \rho$, and $|\Im(w)| \asymp |\Im(z)|$ as $|\Im(z)| \geq 3$ and $\rho \leq 1$. Thus, if $w=w_0\in \mathcal{C}$ maximizes $|\zeta(w)|$ along $\mathcal{C}$, then by  \eqref{zetabound} (with $b$ replaced by $2b$ and $a$ replaced by $a/2$) it follows that
\[
|\zeta^{(k)}(z)| \le \frac{k!}{\rho^{k}}|\zeta(w_0)| \ll_{k,a,b} \frac{|\Im(w_0)|^{1/2-\Re(w_0)}}{\rho^k} \ll_{k,a,b} \frac{(1+|\Im(z)|)^{1/2 - \Re(z) + \rho}}{\rho^k}.
\]
The numerator is $\ll  (1+|\Im(z)|)^{1/2 - \Re(z) }$ since  $\rho \leq 1/\log(3+|\Im(z)|)$; additionally, for each fixed $b \leq \Re(z) \leq a$, for all sufficiently large $|\Im(z)| \gg_{a,b} 1$ we have $\rho \geq 1/ \log (3 + |\Im(z)|)$, and the lemma follows. 
 
 \end{proof}

\subsection{Proof of \cref{thm:Example}(i): verification of \cref{hyp:Strong}} \label{subsec:Example-Verfication} Since the Riemann zeta function has a meromorphic continuation to all of $\mathbb{C}$ with only a simple pole at $s=1$, it follows that $A(s)$ has a meromorphic continuation to all of $\mathbb{C}$ except for a pole at $s=\alpha$ of order $m$. 
As a consequence of Lemma \ref{lem:Growth}, for any $0 < \delta < \alpha$, and $r > 1$, 
\beq\label{generic_r}
|  A(s) |\ll_{\alpha, m, \delta, r} (1+|\Im(s)|)^{r-\frac{1}{2}} \log^{m-1}(3+|\Im(s)|),
\eeq
when $\Re(s) = \alpha-\delta$. As $\kappa > 1/2$, we may set 
\begin{equation} \label{eqn:DirichletSeries-r-Choice}
r = \kappa + \tfrac{1}{2} > 1
\end{equation}
to verify \cref{hyp:Strong} holds for $A(s)$ with $\alpha, \delta, m,$ and $\kappa$.

\subsection{Proof of \cref{thm:Example}(ii): large coefficients} 
From \eqref{eqn:DirichletSeries-A}, the coefficients $a_n$ are non-negative, and direct computation shows
\[    \lambda_n = n^{r/\delta}, \qquad a_n 
    = \lambda_n^{\alpha - \delta/r } \Big( \dfrac{\delta}{r} \log \lambda_n \Big)^{m-1}. \]
In particular $\lam_1=1$.
Our choice of $r$ in \eqref{eqn:DirichletSeries-r-Choice} immediately shows that  
\begin{equation} \label{eqn:Coefficients-A}
a_n = \Omega\big( \lambda_n^{\alpha-\frac{\delta}{\kappa+1/2} } (\log \lambda_n)^{m-1}  \big) \qquad \text{as $n \to \infty$.}
\end{equation}
We also remark that if $r/\del \in \N$ (which occurs when $(\kappa+1/2)/\del \in \N$), then $\lam_n$ are natural numbers so that $A(s)$ is a standard Dirichlet series.

\subsection{Proof of \cref{thm:Example}(iii): large fluctuations} 
Assume, for a contradiction, that 
\[
\sum_{\lambda_n \leq x} a_n = \Res_{s=\alpha} \Big[A(s) \frac{x^s}{s} \Big] + o\Big( x^{\alpha-\frac{\delta}{\kappa+1/2}} (\log \lambda_n)^{m-1} \Big) 
\]
as $x \to \infty$. For $N \geq 1$ and all sufficiently small $0 < \epsilon < \epsilon(N)$, we have that
\[
a_N = \sum_{\lambda_N-\epsilon < \lambda_n \leq \lambda_N} a_n = \sum_{\lambda_n \leq \lambda_N} a_n - \sum_{\lambda_n \leq \lambda_N-\epsilon} a_n. 
\]
By assumption, it follows that as $N \to \infty$,
\[
a_N = \Res_{s=\alpha} \Big[ A(s) \frac{\lambda_N^s - (\lambda_N-\epsilon)^s}{s} \Big] + o\Big( \lambda_N^{\alpha-\frac{\delta}{\kappa+1/2}} (\log \lambda_N)^{m-1}  \Big). 
\]
Taking $\epsilon \to 0^+$, the residue term vanishes by continuity (see \cref{remark_residue}),  so we may conclude that
\[
a_N =  o\Big( \lambda_N^{\alpha-\frac{\delta}{\kappa+1/2}} (\log \lambda_N)^{m-1}  \Big) 
\]
as $N \to \infty$. This contradicts \eqref{eqn:Coefficients-A}.

\section{Limiting counterexample: Proof of \cref{thm:Example_bounded}}\label{sec_example_bounded}

 We prove a more general theorem, which (in combination with Theorem \ref{thm:Example}) immediately implies Theorem \ref{thm:Example_bounded}.  
 \begin{thmnum}\label{thm:Example_deduction}
  Let $r(x)$ be a non-negative, monotone increasing function on $[1,\infty)$ with the property that $r(x+1)\ll r(x)$ uniformly in $x \geq 1$.  Given $\alpha > 0, \delta \in (0,\alpha)$, $\kappa > \frac{1}{2}$, and $m \in \Z^+,$ suppose there exists an increasing sequence   $\{\lambda_n\}_{n \geq 1}$ of real numbers tending to infinity with $\lam_1 \geq 1$, and a  sequence of non-negative real numbers $\{a_n\}_{n \geq 1}$ defining a general  Dirichlet series $A(s) = \sum_{n=1}^{\infty} a_n \lambda_n^{-s}$ which satisfies all of the following: 
	\begin{enumerate}[(i)]
		\item The Dirichlet series $A(s)$ satisfies \cref{hyp:Strong} with   $\alpha, \delta, \kappa,$ and $m$. 

        \medskip 

		\item As $n \to \infty$,  $ a_n = \Omega\Big( r(\lam_n)  \Big)$.


	\end{enumerate}
Then there exists a general Dirichlet series $\tilde{A}(s) = \sum_{n =1}^\infty a_n' (\lam_n')^{-s}$ that satisfies (i) and:   

(ii*) For all $n \geq 1$, $0 \leq a_n \leq 1$. 

\medskip

(iii)  As $x \to \infty$, $\displaystyle \sum_{\lambda_n \leq x} a_n = \Res_{s=\alpha} \Big[ A(s) \frac{x^s}{s} \Big]  + \Omega\Big( r(x) \Big). $	
 \end{thmnum}
To deduce Theorem \ref{thm:Example_bounded}, we apply Theorem \ref{thm:Example_deduction} to the example constructed in Theorem \ref{thm:Example} with the function $r(x) = x^{\al - \del/(\kappa+1/2)}(\log x)^{m-1}$. Then for all $x \geq 1$, $r(x+1) \leq r(2x) \ll_{\al,\del,\kappa,m} r(x)$, which suffices for the application.

 Under the hypotheses of Theorem \ref{thm:Example_deduction}, for infinitely many $n$ the series $A(s)$ has a ``large'' coefficient $a_n$ indexed by  $\lam_n$, and these contribute to the ``rough'' behavior of the truncated partial sums $\sum_{\lam_n \leq x}a_n$. The idea of the proof  is to spread out the effect of each coefficient $a_n$ into a consecutive string of coefficients, each of size at most 1, indexed by $\lam_n, \lam_n+\ep_n,\ldots, \lam_n + \lfloor a_n \rfloor \ep_n$, for some sequence of $\ep_n$ that decay to zero sufficiently fast as $n \maps \infty$. This process will nevertheless preserve the rough behavior of the partial sums. 

 We consider a sequence $\{\ep_n\}_n$ of small positive numbers, decreasing to zero sufficiently rapidly (according to criteria specified momentarily). Define 
 \[ \tilde{A}(s) = \sum_{n=1}^\infty \sum_{j=0}^{\lfloor a_n\rfloor} \frac{a_n}{\lfloor a_n \rfloor + 1} (\lam_n + j\ep_n)^{-s}.\]
We specify that each $\ep_n$ is sufficiently small that $\lam_n + \lfloor a_n \rfloor \ep_n <\min\{ \lam_{n+1}, \lam_n+1\}$. Consequently, $\tilde{A}(s)$ can be written as a general Dirichlet series, say  $\sum_{m=1}^\infty a_m'(\lam_m')^{-s}$, in which  each coefficient $a_m'$ takes the form $a_n (\lfloor a_n\rfloor +1)^{-1}$ (for some $n$). In particular, $a_m' \in [0,1)$ for all $m \geq 1$, so that (ii*) holds. 

We claim that $\tilde{A}(s)$ is absolutely convergent wherever $A(s)$ is, and thus in particular for $\Re(s)>\al$. Indeed, by the non-negativity of the $a_n$ and $\ep_n$,
\[ \sum_{m=1}^\infty |a'_m (\lam_m')^{-s}| = \sum_{n=1}^\infty \sum_{j=0}^{\lfloor a_n\rfloor} \frac{a_n}{\lfloor a_n \rfloor + 1} (\lam_n + j\ep_n)^{-\Re(s)}
\leq \sum_{n=1}^\infty  a_n \lam_n^{-\Re(s)} =  \sum_{n=1}^\infty  |a_n \lam_n^{-\Re(s)}|.\]
Consequently, in the region $\Re(s)>\al$, we can define $B(s):=A(s) - \tilde{A}(s) = \sum_{n=1}^\infty b_n(s)$ in which 
\[ b_n (s)= a_n \lam_n^{-s} - \sum_{j=0}^{\lfloor a_n \rfloor}\frac{a_n}{\lfloor a_n \rfloor + 1} (\lam_n + j \ep_n)^{-s}.\]

We claim it suffices to show that the general Dirichlet series  $B(s)$ is absolutely convergent (and hence holomorphic) in the region $\Re(s) >0$, and 
\beq\label{Bs_Hyp_B}
|B(s)| \ll_\al (1+|\Im(s)|)^\kappa \qquad \text{ for all $0 < \Re(s) \leq \al$.}
\eeq
Indeed, supposing these two properties hold, then $\Res_{s=\al}[A(s)x^s/s] = \Res_{s=\al}[\tilde{A}(s)x^s/s],$ and moreover $\tilde{A}(s)$ satisfies Hypothesis \ref{hyp:Strong} for $\al,\del,\kappa,m$, so that (i) holds. To verify (iii), suppose on the contrary that  
\beq\label{assumed_tail} \sum_{\lam_m' \leq x} a_m' = \Res_{s=\al}\left[ A(s) \frac{x^s}{s}\right] + o(r(x)) \qquad \text{as $x \maps \infty.$}
\eeq
Let $\{\tau_n\}_n$ be a sequence of small positive numbers, decreasing to zero sufficiently rapidly  that
\[\sum_{\lam_m' \leq \lam_n + \lfloor a_n \rfloor \ep_n} a_m' - \sum_{\lam_m' \leq \lam_n -\tau_n} a_m' 
= \sum_{j=0}^{\lfloor a_n\rfloor} \frac{a_n}{\lfloor a_n \rfloor + 1} = a_n.
\]
(It suffices that $\tau_n$ is chosen so that $\lam_n-\tau_n > \lam_{n-1} + \lfloor a_{n-1} \rfloor \epsilon_{n-1}$.)
On the other hand, applying (\ref{assumed_tail}) to each term on the left-hand side, and the hypothesis (ii) for $A(s)$ on the right-hand side,
\begin{multline*}
  \Res_{s=\al}\left[ A(s) \frac{(\lam_n +\lfloor a_n \rfloor\ep_n)^s}{s}\right] - \Res_{s=\al}\left[ A(s) \frac{(\lam_n -\tau_n)^s}{s}\right] +o(r(\lam_n +\lfloor a_n \rfloor\ep_n))+o(r(\lam_n-\tau_n))\\= \Omega(r(\lam_n)) \qquad \text{as $n \maps \infty$.}
\end{multline*}
Consequently,
\[  \Res_{s=\al}\left[ A(s) \frac{(\lam_n +\lfloor a_n \rfloor\ep_n)^s}{s}\right] - \Res_{s=\al}\left[ A(s) \frac{(\lam_n -\tau_n)^s}{s}\right]= \Omega(r(\lam_n)) \qquad \text{as $n \maps \infty$,}\]
under the hypothesis that $r(x+1) \ll r(x)$ uniformly in $x$.
However, the explicit expression computed for 
$\Res_{s=\al}[A(s)x^s/s]$ in Remark \ref{remark_residue} is a continuous function of $x$, so that for each $n$, there exists a choice of $\ep_n$ and $\tau_n$ sufficiently small that the left-hand side is as small as we like; in particular the sequences $\ep_n$ and $\tau_n$ can be chosen to decay sufficiently rapidly that the left-hand side    is $o(r(\lam_n))$ as $n \maps \infty$. This is a contradiction, and so (\ref{assumed_tail}) is false. This verifies property (iii) for the series $\tilde{A}(s)$. 

It remains to prove that the difference series  $B(s)$ is absolutely convergent   in the region $\Re(s) >0$ and (\ref{Bs_Hyp_B}) holds. To accomplish this, fix an $s$ with $\Re(s)>0$, and use the non-negativity of $a_n$ to write
\[ |b_n (s)| \leq   \sum_{j=0}^{\lfloor a_n \rfloor}\frac{a_n}{\lfloor a_n \rfloor + 1} |\lam_n^{-s}  - (\lam_n + j \ep_n)^{-s}| \leq a_n \cdot  \max_{\substack{t \in \R \\ 0 \leq t \leq \lfloor a_n \rfloor \ep_n}}  |\lam_n^{-s}  - (\lam_n + t)^{-s}|.\]
As a result,
\[ |b_n(s)| \leq a_n \cdot \min \{ 2, |s|\lam_n^{-1} a_n \ep_n\} = \min \{ 2a_n, |s| a_n^2 \lam_n^{-1} \ep_n\}.\]
Here the first option follows since $|\lam_n^{-s}  - (\lam_n + t)^{-s}| \leq 2$ by the triangle inequality, and the second option is via the mean value theorem (using that $\lam_n \geq 1$ and $\Re(s)>0$).
 Consequently, for any fixed cut-off $N$ (which we will choose momentarily in terms of $s$), 
 \beq\label{b_break_sum} \sum_{n=1}^\infty |b_n(s)|= \sum_{1 \leq n \leq N} |b_n(s)| + \sum_{n>N} |b_n(s)| \leq 2\sum_{n=1}^{N}a_n + |s|\sum_{n>N} a_n^2\lam_n^{-1} \ep_n.
 \eeq
 In order to achieve both convergence and (\ref{Bs_Hyp_B}), we now define for all real $t \geq 0$ that $N(t) \geq 0$ is the largest integer such that $\sum_{n=1}^{N(t)} a_n \leq (1+t)^\kappa$; then $N(t)$ is a non-decreasing function that tends to infinity as $t \maps \infty$. 
 We now specify that the sequence $\ep_n$ decreases sufficiently rapidly that $\sum_{n=1}^\infty a_n^2 \lam_n^{-1}\ep_n$ converges; moreover we specify that $\ep_n$ decreases so rapidly that for every $t \geq 0$, 
 \[\sum_{n>N(t)} a_n^2\lam_n^{-1} \ep_n \leq t^{-1}.\]
 Then applying (\ref{b_break_sum}) with $N=N(|s|)$ shows that 
 \[ \sum_{n=1}^\infty |b_n(s)| \leq 2(1+|s|)^\kappa + 1.
 \]
This verifies that $B(s)$ converges absolutely for $\Re(s)>0$ as well as (\ref{Bs_Hyp_B}), and the proof of Theorem \ref{thm:Example_deduction} is complete.

\section{Limiting counterexample: Proof of Theorem \ref{thm_hyp_weak_counterex} and Theorem \ref{thm_hyp_weak_counterex_vertical_growth}}\label{sec_hyp_weak_counterex}
 
We now record self-contained proofs of Theorem  \ref{thm_hyp_weak_counterex}  and Theorem \ref{thm_hyp_weak_counterex_vertical_growth}.
Theorem  \ref{thm_hyp_weak_counterex} is deduced from a statement about an auxiliary series, using ideas recorded in \cite{Ten15}, \cite{JS21}.
\begin{propnum}[Auxiliary series $B(s)$]\label{prop_B_to_define_A}
Consider the Dirichlet series $B(s) = \sum_{n \geq 1} b_n n^{-s}$ with non-negative coefficients defined by $b_n := 1 + \cos (\log^2 n)$ for each $n \geq 1$.  The series $B(s)$ is holomorphic in the region $\Re(s)>1$. Moreover $B(s)$   continues meromorphically to $\Re(s)>0$, and within this region has only a single simple pole at $s=1$.  
\end{propnum}

 \begin{cornum}\label{cor_A_mero_from_B}
 The series initially  defined in the region $\Re(s)>1$ by $A(s)=-B'(s)$ is holomorphic in the region $\Re(s)>1$   and continues meromorphically to $\Re(s)>0$, and within this region has only a pole of order $2$ at $s=1$.  
 \end{cornum}
 \begin{proof}[Deduction of Corollary \ref{cor_A_mero_from_B}]
 Because $B(s)$ is holomorphic in the region $\Re(s)>1$, $A(s)$ may be computed by differentiating the Dirichlet series $B(s) = \sum_{n \geq 1} b_n n^{-s}$ term by term, so that   $A(s) = \sum_{n\geq 1} a_n n^{-s}$ with non-negative coefficients $a_n := b_n \log n$; this converges locally uniformly in the region $\Re(s)>1$ (see e.g. \cite[p. 13]{MonVau07}). Since by Proposition \ref{prop_B_to_define_A}, $B(s)$   continues meromorphically to $\Re(s)>0$ with only a single simple pole at $s=1$, then  $A(s)$ is meromorphic in the region $\Re(s)>0$ with its only singularity in this region being a pole of order 2 at $s=1$.
 \end{proof}

\begin{proof}[Deduction of   Theorem \ref{thm_hyp_weak_counterex}]
Define 
\[A(s) =  - B'(s).\]
Corollary \ref{cor_A_mero_from_B} verifies the meromorphic property claimed in Theorem \ref{thm_hyp_weak_counterex}.
To complete the proof of Theorem \ref{thm_hyp_weak_counterex}, it only remains to show that
\beq\label{weak_counterexample_partial_sum} \sum_{n \leq x} a_n  = \sum_{n\leq x} (\log n + (\log n) \cos (\log^2 n)) = x \log x + x \left( \tfrac{1}{2} \sin (\log^2 x) -1\right) + O(x/\log x).
\eeq
To verify this expansion, first observe that by partial summation 
\[ \sum_{1 \leq n \leq x} \log n = x \log x - (x-1) + O(\log x).\]
Second, by the Euler-Maclaurin summation formula (\ref{EulerMaclaurin}) with $\phi(u) = (\log u) \cos (\log^2 u)$, 
\[ \sum_{1 \leq n \leq x} \log n  \cos (\log ^2(n)) = \int_{1}^x (\log t) \cos (\log^2 t)dt  + O(\log^3 x). 
\]
(This uses that $\phi(1)=0$, so that the sum is nominally over $1<n \leq x$, and is initially proved for $x$ an integer;  replacing $\lfloor x \rfloor$ by $x$ in the right-hand side only contributes a term that is dominated by the remainder term already present.)
The remainder term is sufficiently small for the claim (\ref{weak_counterexample_partial_sum}).
Now since 
\[ \frac{d}{dt}\left[ \tfrac{1}{2} t \sin (\log^2 t) \right]= (\log t)\cos (\log^2 t)  + \tfrac{1}{2} \sin (\log^2 t),\]
the integral evaluates to precisely the desired main term and a remainder term:
\[ 
\int_{1}^x (\log t) \cos (\log^2 t)dt  = \tfrac{1}{2}x \sin \left(\log^2 x\right) - \tfrac{1}{2}\int_1^x  \sin \left(\log^2 t\right) dt.
\]
Thus (\ref{weak_counterexample_partial_sum}) will follow if the last term on the right is $O(x/\log x)$. We estimate the contribution from integrating over $t \in [1,2]$ trivially by $O(1)$. Note that
\[ \frac{d}{dt} \left[   \frac{t}{2\log t} \cos (\log^2 t) \right] =- \sin \left(\log^2 t\right)   + \tfrac{1}{2}\cos \left(\log^2 t\right) \left( \frac{1}{\log t} - \frac{1}{(\log t)^2}\right).\]
Thus   
\begin{align*}\left|\int_2^x  \sin \left(\log^2 t\right) dt\right| & \ll \left|\left.\frac{t \cos \left(\log^2 t\right)}{\log t}\right|_{2}^{x}\;\right| + \left|\int_2^x  \cos \left(\log^2 t\right)\left( \frac{1}{\log t}- \frac{1}{(\log t)^2}\right) dt\right| \\
& \ll \frac{x}{\log x} +\int_2^x \frac{dt}{\log t}\ll \frac{x}{\log x},
\end{align*}
completing the verification of (\ref{weak_counterexample_partial_sum}).
 \end{proof}

\subsection{Proof of Proposition \ref{prop_B_to_define_A} for auxiliary series $B(s)$}\label{sec_proof_B_ppties}
First observe that the series $B(s)$ converges absolutely for $\Re(s)>1$ (and uniformly in any half-plane $\Re(s)\geq 1+\del$ for a fixed $\del>0$), since $b_n \in [0,2]$ for all $n$. Thus $B(s)$ is holomorphic in the region $\Re(s)>1$. Now we must verify its meromorphic continuation to $\Re(s)>0$.
Define the function $R(x)$ by writing 
\beq\label{bn_sum}
\sum_{1 \leq n \leq x}b_n = \int_1^x \left(1+ \cos (\log^2 t)\right)\,dt + R(x)
\eeq
for $x > 1$ (and set $R(x)=0$ for $x \leq 1$). 
By the Euler-Maclaurin summation formula (\ref{EulerMaclaurin}), 
it follows that $R(x) \ll    \log^2 x.$
If we denote the partial sum on the left-hand side by the function $b(x)$, then using the convention in (\ref{Mellin_dfn1}), the Mellin transform $\Mcal(b)(s)$ of $b(x)$   is   the associated Dirichlet series $B(s)$.
Thus for $\Re(s)>1$,
\[ B(s)  = \int_1^\infty x^{-s} d F(x) + \int_1^\infty x^{-s} dR(x) ,\]
in which $F(x)=\int_1^x (1+ \cos (\log^2 t))dt$. 
By applying the definition of the Riemann-Stieltjes integral to the first term (see \cite[\S 6.1 example  (2)]{BatDia04}) and 
  integration by parts to the second integral, this implies (again for $\Re(s)>1$):
\[ B(s) = \int_1^\infty \left(1+ \cos (\log^2 t)\right)t^{-s} dt 
+ s \int_1^\infty R(t) t^{-s-1}dt.\]
Since $R(x) \ll \log^2 x$, the second integral is convergent for $\Re(s)>0$.  By \cite[Thm. 6.20]{BatDia04}, the Mellin transform $\mathcal{M}(R)(s)$ of a function $R$ of at most polynomial growth is holomorphic on the right-half plane in which the Mellin transform converges; thus the second integral is holomorphic for $\Re(s)>0$.

The contribution of 1 to the first integral on the right-hand side can be integrated explicitly to $1/(s-1)$ \cite[\S 6.1 example 6.1(2)]{BatDia04}. For the cosine term, apply Euler's identity $\cos \theta = (e^{i\theta} + e^{-i\theta})/2$ and the change of variables $u=\log t$. Then in total the first integral on the right-hand side evaluates to 
\beq\label{sumII}
\frac{1}{s-1} + \frac{1}{2} I_+(s) +\frac{1}{2}I_-(s),
\eeq
in which the simple pole of $B(s)$ at $s=1$ is made visible in the first term, and we define two functions $I_+(s)$ and $I_-(s)$ of a complex variable by:
\beq\label{I_original}
I_\pm (s) = \int_0^\infty \exp\{ (1-s) u \pm iu^2 \} du.
\eeq
 
\begin{propnum}\label{prop_I_plus_minus}
    Each of $I_+(s)$ and $I_-(s)$ is holomorphic for $\Re(s)>1$, and extends to an entire function of $s$.  
\end{propnum}
Assuming this proposition, observe that we can conclude that $B(s)$ extends meromorphically to $\Re(s)>0$ with its only singularity in this region being the single simple pole at $s=1$ isolated in (\ref{sumII}).  
Thus the proof of Proposition \ref{prop_B_to_define_A} is complete, pending the proof of Proposition \ref{prop_I_plus_minus}.

\subsection{Proof of Proposition \ref{prop_I_plus_minus} for entire functions $I_\pm (s)$}
Each of $I_+(s)$ and $I_-(s)$ is absolutely convergent (as an integral of $u$) for $\Re(s)>1$, and uniformly convergent in any right-half plane $\Re(s)\geq 1+\del$ for a fixed $\del>0$. For $\Re(s) \geq 1+\del$, each truncation of the integral to $\int_0^N$ is an integral of a function, say $F(s,u)$, that is holomorphic for $s \in \Omega :=\{s: \Re(s)>1\}$ and is jointly continuous for $(s,u) \in \Omega \times [0,N]$; thus the integral $\int_0^N F(s,u)du$ is holomorphic for $s \in \Omega$ \cite[Thm. 5.4]{SSComp}. Taking $N \maps \infty$, we can write $I_\pm (s)$ as a uniform limit of holomorphic functions in the region $\Re(s) \geq 1+\del$, so that $I_\pm (s)$ is holomorphic in that region \cite[Thm. 5.2]{SSComp}. Since $\del>0$ is arbitrary, each of $I_+(s)$ and $I_-(s)$ is holomorphic for $\Re(s)>1$.

Next, we claim that each of $I_+(s)$ and $I_-(s)$ extends to an entire function of $s$.  Indeed, for $I_\pm(s)$ as defined above, we claim that by contour integration, the integral  over $u$ on the positive real axis can be shifted  to the half-ray at angle $\pm \pi/4$ to the positive real axis in $\C$ (formally by $u \mapsto e^{\pm \frac{\pi i }{4}} u$), leading to the expression
\beq\label{Ipm_after_contour}
I_\pm (s) = e^{\pm \frac{\pi i }{4}}\int_0^\infty \exp\{ (1-s) ue^{\pm \frac{\pi i }{4}} -u^2 \} du.
\eeq
Once we have verified this expression, it makes visible that each of $I_+(s)$ and $I_-(s)$ is an entire function of $s$, since the right-hand side is an entire function of $s$. Indeed, fix any compact set $S \subset \C$. By the familiar argument,  if the integral is truncated to $[0,N]$, then the integral is a holomorphic function (say $f_N$) of $s \in S$ (as an integral over a compact region of a holomorphic function jointly continuous in $s,u$). Moreover, the integrand is bounded by $\exp\{ \Re((1-s) ue^{\pm \frac{\pi i }{4}} -u^2)\} \ll \exp\{ -u^2 + C|u|\}$, for a constant $C$ depending on the set $S$. Because the quadratic exponential decay dominates, it follows that as $N \maps \infty$, the functions $f_N$ converge uniformly for  $s \in S$, and their limit (namely the integral on the right-hand side) is thus holomorphic in $S$. This holds for arbitrary compact sets $S$, which suffices for the claim. 

It remains to prove that $I_+(s)$, as originally defined in (\ref{I_original}), can be expressed as (\ref{Ipm_after_contour}) after a contour shift, and similarly for $I_-(s)$. For $I_+(s)$, we will initially prove this for any fixed $s=\sig+i\tau$ in a region $\mathcal{R}_+$, and for $I_-(s)$, we will initially prove this for any fixed $s$ in a region $\mathcal{R}_-$. We  define these regions by
\[ \mathcal{R}_+ = \{s=\sig+i\tau: \sig \geq 2, \tau< -1\}, \qquad 
\mathcal{R}_- = \{s=\sig+i\tau: \sig \geq 2, \tau>1\}.
\]
Once we have verified that $I_+(s)$, as initially defined, agrees with the right-hand side of (\ref{Ipm_after_contour}) for all $s$ in $\mathcal{R}_+$,  then (\ref{Ipm_after_contour}) provides an analytic continuation of (\ref{I_original}) to all $s \in \C$. Similarly once we have verified that $I_-(s)$, as initially defined, agrees with the right-hand side of (\ref{Ipm_after_contour}) for all $s$ in $\mathcal{R}_-$,  then (\ref{Ipm_after_contour}) provides an analytic continuation of (\ref{I_original}) to all $s \in \C$.

Fix $s \in \mathcal{R}_\pm$. To justify the change of contour applied to (\ref{I_original}) to  obtain (\ref{Ipm_after_contour}), we note first that the function $e^{(1-s)z}e^{\pm iz^2}$ has no pole in $z$ for $z$ in the upper right quadrant (for the $+$ case), and no pole in the lower right quadrant (for the $-$ case). Second, we must check that the contribution of the integral of this function along the arc $\ga^\pm_R$ (from angle $0$ to angle $\pm \pi /4$ at radius $R$), given by 
\[  
C^\pm_R :=\int_{\ga^\pm_R} e^{(1-s)z} e^{ \pm iz^2}dz = \int_0^{\pm \pi/4} \exp\{ (1-s) Re^{i\theta} \pm i(Re^{i\theta})^2\} (iRe^{i\theta}) d\theta,
\]
goes to zero as $R \maps \infty$. Recall that $s \in \mathcal{R}_\pm$ is fixed. 
 By writing $s=\sig +i\tau$ and taking the real part of the exponent,
 \begin{align*}
      |C^\pm_R| & \ll R \int_0^{\pm \pi/4} 
 \exp \{ \mp R^2 \sin (2\theta) + R ((1-\sigma) \cos \theta + \tau \sin \theta)\} d\theta 
\\
& \ll R \int_0^{\pm \pi/4} 
 \exp \{ \mp R^2 \sin (2\theta) + R( -\cos \theta \mp \sin \theta)\} d\theta.
 \end{align*}
Here we have applied the constraints on $\sig, \tau$ in the definition of $\mathcal{R}_+$ and $\mathcal{R}_-,$ respectively, and recalled  the signs of $\cos \theta$ and $\sin \theta$ in the respective region of integration.
We extract the decay guaranteed by the cosine term, and   bound the remaining integrand by 1:
 \[ |C^\pm_R|\ll R e^{-R \cos (\pm \pi/4)} |\int_0^{\pm \pi/4} 
   d\theta | \ll R e^{-c_0 R},\]
  for a fixed $c_0 >0$. 
 This confirms the vanishing of $C^\pm_R $ as $R \maps \infty$, for each fixed $s \in \mathcal{R}_\pm$, as desired. This completes the verification that each of $I_+(s)$ and $I_-(s)$ is an entire function.
 This completes the proof of Proposition \ref{prop_I_plus_minus}, and hence the proof of Proposition \ref{prop_B_to_define_A}, Corollary \ref{cor_A_mero_from_B}, and Theorem \ref{thm_hyp_weak_counterex}.

\subsection{Asymptotics along vertical lines}\label{sec_asymp_vertical_lines}
 
We now turn to the considerations of Theorem \ref{thm_hyp_weak_counterex_vertical_growth}.
We will prove the following general statement by contour integration:
\begin{lemnum}\label{lemma_asymptotic_exponential_growth}
 Let $\be = a +ib \in \C$ with $a,b \in \R$, $a>0, b \neq 0$. Fix $k \in \{0,1\}$.   Then 
\[ J_k(\be) :=\int_0^\infty u^k e^{2\be u - u^2} du = \be^k \sqrt{\pi}e^{\be^2}  + j_k(\be), \] 
in which $|j_0(\be)| \leq |\be| /2a^2$ and $|j_1(\be)| \ll |\be|^2 (1/a^2 + 1/a^4)$.
\end{lemnum}
 This lemma applies to non-real $\be$ with $\Re(\be)>0$. On the other hand, observe that uniformly for all $\be$ with $\Re(\be)\leq 0$, a trivial upper bound shows  
\beq\label{J_trivial_upper}
|J_k(\be)| \leq \int_0^\infty u^k e^{2\Re(\be)u -u^2} du \leq \int_0^\infty u^k e^{ -u^2} du \ll 1.
\eeq

\begin{cornum}\label{cor_I_plus_minus_growth}
Let $I_+(s)$ and $I_-(s)$ be the entire functions considered in Proposition \ref{prop_I_plus_minus}, as defined in (\ref{I_original}) or equivalently (\ref{Ipm_after_contour}), with derivatives denoted $I_+'(s)$ and $I_-'(s)$.
 For every fixed $\del \in (0,1)$,  $\limsup_{t \maps -\infty} |I_+'(1-\del+it)|\ll 1$, while 
\beq\label{I'plus}
I_+'(1-\del +it) = \tfrac{\sqrt{\pi}}{2}  \ep_+(\del,t) t e^{+ \tfrac{1}{2}\del t}  (1+O_\del(1/t))+O(1),
\eeq
  as $t \maps + \infty$, 
with  a continuous function $\ep_+(\del,t)$ such that $|\ep_+(\del,t)|=1$ for all $\del,t$.
 For every fixed $\del \in (0,1)$,  $\limsup_{t \maps +\infty} |I_-'(1-\del+it)|\ll 1$, while 
\beq\label{I'minus} 
I_-'(1-\del +it) = \tfrac{\sqrt{\pi}}{2}  \ep_-(\del,t) t e^{- \tfrac{1}{2}\del t}  (1+O_\del(1/|t|))+O(1),
\eeq
  as $t \maps - \infty$, 
with  a continuous function $\ep_-(\del,t)$ such that $|\ep_-(\del,t)|=1$ for all $\del,t$.
\end{cornum}
 Consequently, for $s$ on the line $\Re(s)=1-\del$ with $t=\Im(s)$,   $|I'_+(s)+I_-
'(s)|$ exhibits exponential growth in $|t|$ as $|t| \maps \infty$, with the exponential growth coming from $|I'_+(s)|$ for $t \maps + \infty$ and from $|I'_-(s)|$ for $t \maps -\infty$. Moreover, observe that because of the $O(1)$ remainder term, the statement  (\ref{I'plus}) remains valid as $t\maps -\infty$ as well (even though nominally the first term decays to zero in this regime). Analogously,   the statement (\ref{I'minus}) remains valid as $t \maps + \infty$. 

\begin{proof}[Deduction of Corollary \ref{cor_I_plus_minus_growth}]
    By differentiating the identity (\ref{Ipm_after_contour}) for the entire function $I_\pm (s)$, we may write
\[
I_\pm' (s) =- (e^{\pm \frac{\pi i }{4}})^2\int_0^\infty u\exp\{  (1-s)ue^{\pm \pi i/4} -u^2 \} du.
\]
Consider $s$ on the line $\Re(s) = 1-\del$ for a fixed $0<\del < 1$.  On this line, denoting $t=\Im(s)$,
\[
I_\pm' (s) =   \mp i \int_0^\infty u \exp\{ \tfrac{\sqrt{2}}{2}(\del-it) (1\pm i) u -u^2 \} du.
\]
  Define 
\begin{align*}
    \be_+= \be_+(\del,t) &= \tfrac{\sqrt{2}}{4}(\del-it)(1+i) = \tfrac{\sqrt{2}}{4}((\del+t) + i(\del-t))   = \tfrac{1}{2}te^{-\pi i /4}(1+O_\del(1/|t|)), \\
   \be_-=    \be_-(\del,t) &= \tfrac{\sqrt{2}}{4}(\del-it)(1-i) = \tfrac{\sqrt{2}}{4}((\del-t) + i(-\del-t)) = -\tfrac{1}{2}te^{\pi i /4}(1+O_\del(1/|t|)),
\end{align*} 
in which the last relations specify that $|t| \maps \infty$. 
Certainly we may assume $|t|>2$ is sufficiently large that $|\Re(\be_\pm)| \geq 1$ and $|t\pm \del| \geq |t|/2$. 
Note that $\be_+(\del,t)$ has positive real part as $t \maps + \infty$ and negative real part as $t \maps -\infty$; in the latter case we will apply the trivial bound (\ref{J_trivial_upper}). On the other hand, $\be_-(\del,t)$ has positive real part as $t \maps - \infty$ and negative real part as $t \maps +\infty$; in the latter case we will apply the trivial bound (\ref{J_trivial_upper}). Thus we record 
\begin{align*}
I'_+(1-\del +it) & \ll 1 \qquad \text{as $t \maps - \infty$}\\
I'_-(1-\del +it) & \ll 1 \qquad \text{as $t \maps +\infty$}
\end{align*}

 Next we consider the remaining cases, in which  Lemma \ref{lemma_asymptotic_exponential_growth}   proves:  
\begin{align*}
   I_+'(1-\del+it)&= - i J_1(\be_+(\del,t)) = - i \be_+ \sqrt{\pi}e^{\be_+^2} + O(|\be_+|^2 \cdot \Re(\be_+ )^{-2}), \qquad t>0,\\
     I_-'(1-\del+it)&=  i J_1(\be_-(\del,t)) =  i \be_- \sqrt{\pi}e^{\be_-^2} + O(|\be_-|^2 \cdot \Re(\be_- )^{-2}), \qquad t<0.
     \end{align*}
     
 Using the definitions of $\be_+$ and $\be_-$, observe that  as $|t| \maps \infty$, for $0< \del< 1$ fixed,
\[ \mp i  \be_{\pm}\sqrt{\pi} 
   =    -i \tfrac{\sqrt{\pi}}{2}e^{\mp \pi i /4} \cdot t(1+O_\del(1/|t|)).
\]
Additionally,
\begin{align*}
    |\be_\pm(\del,t)| & = \tfrac{\sqrt{2}}{4}( 2t^2 + 2\del^2)^{1/2} \ll |t|,\\
   | \Re(\be_\pm(\del,t))| & = \tfrac{\sqrt{2}}{4} |\del \pm t| \gg  |t| \\
    \Re(\be_\pm(\del,t)^2) &= \tfrac{1}{8}[(\del \pm t)^2 - (\del\mp t)^2 ]= \pm \tfrac{1}{2} \del t.
\end{align*} 
Thus
$O(|\be_\pm|^2 \cdot \Re(\be_\pm )^{-2})   = O(1).$
Define  constants  $ c_\pm = -i \tfrac{\sqrt{\pi}}{2}e^{\mp \pi i /4}.$
We may conclude that  
\[
   I_\pm'(1-\del +it) 
      = c_\pm t \exp\{\be_\pm(\del,t)^2\}  (1+O_\del(1/|t|))+O(1),\]
      for every fixed $\del \in (0,1)$, under the condition that $t \maps + \infty$ when we consider $I_+'$ while $t \maps -\infty$ when we consider $I_-'$. 
In each of these relevant cases, to understand the exponential growth of the main term, we can define  $\ep_\pm(\del,t):= -ie^{\mp \pi i/4}\exp\{i \Im(\be_\pm(\del,t)^2)\}$ so that $|\ep_\pm(\del,t)|=1$ for all $\del,t$. Then 
\begin{align*}
     I_\pm'(1-\del +it) 
     & = \tfrac{\sqrt{\pi}}{2}\ep_\pm(\del,t)t \exp\{\Re(\be_\pm(\del,t)^2)\}  (1+O_\del(1/|t|))+O(1)\\
     & = \tfrac{\sqrt{\pi}}{2} \ep_\pm(\del,t)t e^{\pm \tfrac{1}{2}\del t}  (1+O_\del(1/|t|))+O(1).
\end{align*}
In particular, this shows that on the line $\Re(s)=1-\del$,  $|I_+'(s)|$ exhibits exponential growth in $t$ as $t \maps + \infty$ since the main term  is of order $te^{\del t/2}$ (up to a factor of constant norm).     The complementary effect is seen for $|I_-'(s)|$: on the line $\Re(s)=1-\del$ it exhibits exponential growth in $|t|$ as $t \maps - \infty$ since the main term  is of order $|t|e^{-\del t/2}=|t|e^{\del |t|/2}$ (up to a factor of constant norm).
\end{proof}

\begin{proof}[Deduction of Theorem \ref{thm_hyp_weak_counterex_vertical_growth}]
We have already verified in Corollary \ref{cor_A_mero_from_B} that $A(s)$ extends meromorphically to $\Re(s)>0$ with only one pole, of order 2, at $s=1$. Now we must prove that $A(s)$ exhibits exponential growth in $|\Im(s)|$ along the vertical line $\Re(s)=1-\del$ for each $\del \in (0,1)$. Recall the auxiliary series $B(s)$ defined in Proposition \ref{prop_B_to_define_A}.
We have shown for $\Re(s)>0$ that 
\[ B(s)  = (s-1)^{-1} + \tfrac{1}{2} I_+(s) +\tfrac{1}{2}I_-(s) + s \int_1^\infty R(u) u^{-s-1}du,\]
in which $I_+ (s)$ and $I_-(s)$ are entire functions, and  the last integral is holomorphic for $\Re(s)>0$. Differentiating term by term, 
\[A(s) = B'(s) =-(s-1)^{-2} + \tfrac{1}{2} I_+'(s) +\tfrac{1}{2}I_-'(s) + \tilde{R}(s),\]
say. For any fixed $\del  \in (0,1)$, on the line $\Re(s)=1-\del$, $|s-1|^{-2} \ll_\del 1$. Differentiating the last term in $B(s)$ and applying the fact that $R(u)\ll \log^2 u$ shows that $|\tilde{R}(s)| \ll_\del (1+|s|) \ll (1+|t|)$ for $\Re(s)>0$. So to establish that $A(s)$ exhibits exponential growth on the line $\Re(s)=1-\del$ it suffices to show that $\tfrac{1}{2} I_+'(s) +\tfrac{1}{2}I_-'(s)$ does.

 By Corollary \ref{cor_I_plus_minus_growth},  for any fixed $\del \in (0,1)$, on the line $\Re(s)=1-\del$, the function 
 $|\tfrac{1}{2}I_+'(s) + \tfrac{1}{2} I_-'(s)|$ exhibits exponential growth in $|\Im(s)|$ as $|\Im(s) |\maps + \infty$.   Moreover, from the  asymptotics (and trivial upper bounds) provided by Corollary \ref{cor_I_plus_minus_growth}, we can consolidate all the cases to conclude 
that for every fixed $\del \in (0,1)$, as $|t| \maps \infty$,
 \[ A(1-\del+it) = \tfrac{1}{2}\cdot \tfrac{\sqrt{\pi}}{2} t (\ep_+(\del,t)e^{+ \tfrac{1}{2}\del t}  +\ep_-(\del,t)e^{- \tfrac{1}{2}\del t}  )(1+O_\del(1/|t|)) + O(1+|t|),\]
with a continuous function $\ep_\pm(\del,t)$ such that $|\ep_\pm(\del,t)|=1$ for all $\del,t$.
(For clarity, we note that this  statement makes no claims of exponential decay; for example while the first term decays as $t \maps -\infty$, this behavior is dominated by the $O(1)$ remainder term.)

  This completes the deduction of Theorem \ref{thm_hyp_weak_counterex_vertical_growth} from Corollary \ref{cor_I_plus_minus_growth}; all that remains is to prove Lemma \ref{lemma_asymptotic_exponential_growth}.
 \end{proof}

\begin{proof}[Proof of Lemma \ref{lemma_asymptotic_exponential_growth}]
For any fixed $\be \in \C$, the integral defining $J_k(\be)$ converges absolutely. In particular, $J_k(\overline{\be}) = \overline{J_k(\be)}$; the right-hand side of the claimed identity also admits such a symmetry.  Thus it suffices to prove the lemma in the case where $\be$ lies in the (open) first quadrant, as we henceforward assume. By completing the square,
\[ J_k(\beta) = e^{\be^2}\int_0^\infty u^ke^{-(u-\be)^2} du = \lim_{R \maps\infty} e^{\be^2}\int_{L_1} z^k e^{-(z-\be)^2}dz,\] 
in which $L_1$ denotes the oriented line from $0$ to $R$ along the real axis. We will think of $R$ as fixed and large (relative to $|\be|$), and will later take $R \maps \infty$. Construct a contour $\Gamma$ as pictured. The horizontal portion $L_3$ of the contour $\Ga$ passes through $\be$, and $L_5$ lies on a hyperbola chosen so that for $z$ on $L_5$, the imaginary part of $(z-\be)^2$ is constant, so that there is no oscillation from the phase when integrating along this hyperbola. The vertical portion $L_4$ lies on the line $\Re(z)=-R$, and truncates the hyperbola. The function $z^ke^{-(z-\be)^2}$ is entire (recall $k \in \{0,1\}$), so that no poles are contained in $\Gamma$. 

\begin{tikzpicture}[decoration={markings,
    mark=at position 7cm with {\arrowreversed[line width=1pt]{stealth}},
    mark=at position 9cm   with {\arrowreversed[line width=1pt]{stealth}},
    mark=at position 11.8cm with {\arrowreversed[line width=1pt]{stealth}},
     mark=at position 17cm with {\arrowreversed[line width=1pt]{stealth}}
  }] 

  \draw[thick, ->] (-7.4,0) -- (6,0) coordinate (xaxis);
  \draw[thick, ->] (0,-3.5) -- (0,4) coordinate (yaxis);
\node at (-4,-.3) {$-R$};
\node at (4,-.3) {$R$};
\node at (4.2,3.2) {$\Gamma$};
\node at (2.5,3) {\pgfuseplotmark{*}};
  \node at (2.5,3.3) {$\be$}; 
  \node at (1,.3) {$L_1$};
    \node at (1,3.3) {$L_3$};
    \node at (-2,1.7) {$L_5$};

   \node at (-2.5,-3) {\pgfuseplotmark{*}};
  \node at (-2.5,-3.3) {$-\be$}; 
    \node at (-5.5,-1) {$\ga$}; 
            
  \path[draw,black, line width=1.4pt, postaction=decorate] 
        (-4,3)
    --  (4, 3)  node[midway, above right, black] {} 
    --  (4, 0)  node[midway, right, black] {$L_2$} 
     --  (0,0)   arc[start angle=0, end angle=83.6, x radius=4.5cm, y radius=2.5cm]   node[midway, below, black] {} 
    --  (-4,2.6)  node[midway, below, black] { }  
      --  (-4,3)  node[midway, left, black] {$L_4$}   ;
        \draw[->,dotted, line width=1pt, postaction=decorate] 
       (-2.5,-3) 
    arc[start angle=0, end angle=95, x radius=4.5cm, y radius=2.5cm]       ; 
      
\end{tikzpicture}

Explicitly, suppose $\be = a+ib$ with  $a>0,b >0$. Write $z-\be = A+iB$ with $A,B \in \R$. Then $-(z-\be)^2 = -(A^2 - B^2) - 2ABi$, so that for this function to have constant imaginary part it must be that $AB$ is a constant; we choose this constant to be $ab \neq 0$. Thus we define $L_5$ to lie on the hyperbola defined by 
\[ AB = \Re(z-\be)\Im(z - \be) =ab.\] 
(This is a hyperbola with asymptotes $\Re(z)=\Re(\be)$ and $\Im(z)=\Im(\be)$, with components in the ``first and third quadrants'', relative to these asymptotes.)
For each fixed $\be$, the contribution of the vertical pieces is negligible as $R \maps \infty$, since (for some small $\del = \del_R>0$):
 \begin{align*} 
\int_{L_2} z^k e^{-(z-\be)^2} dz &=i \int_{0}^{b} (R+iy)^k e^{-(R+iy-\be)^2} dy \ll_\be R^k e^{-(R-a)^2} \int_0^b e^{(y-b)^2} dy \ll_{\be} R^k e^{-(R-a)^2}, \\
\int_{L_4} z^k e^{-(z-\be)^2} dz &=i \int_{b}^{b-\del} (-R+iy)^k e^{-(-R+iy-\be)^2} dy \ll_\be R^k e^{-(R+a)^2} \int_b^0 e^{(y-b)^2} dy \ll_{\be} R^k e^{-(R+a)^2}.
\end{align*}
Each of these terms is $\ll_{\be} R^k e^{-R^2/2}$ for all $R \gg a$ (that is, sufficiently large with respect to $\be$), which suffices for the claim.

For the integral on $L_3$ along the line $\Im(z) = \Im(\be)$, observe that for $k \in \{0,1\}$,
\[ - \lim_{R \maps \infty} \int_{L_3} z^k e^{-(z-\be)^2} dz = \int_{-\infty}^\infty (u+ib)^k e^{-(u-a)^2} du = \be^k \sqrt{\pi}.\] 
For clarity, in the case $k=0$ this is due to 
\[ \int_{-\infty}^\infty e^{-(u-a)^2}du = \int_{-\infty}^\infty e^{-v^2}dv = \sqrt{\pi}.\]
In the case $k=1$ we can verify this by splitting the integral as
\[  \int_{-\infty}^\infty u e^{-(u-a)^2} du+ ib \int_{-\infty}^\infty  e^{-(u-a)^2} du
= \int_{-\infty}^\infty (v+a) e^{-v^2} dv+ ib \int_{-\infty}^\infty  e^{-v^2} dv = 0 + a \sqrt{\pi} + ib \sqrt{\pi}.\]
By the residue calculus, it follows that for  fixed $\be$,
\beq\label{J_intermediate_conclusion} J_k(\be) = \lim_{R \maps \infty} e^{\be^2} \int_{L_1} z^k e^{-(z-\be)^2}dz
 = \be^k \sqrt{\pi} e^{\be^2}  + e^{\be^2} E_k(\be),
 \eeq
 in which 
 \[ E_k(\be):=  \lim_{R \maps \infty} \int_{- L_5} z^k e^{-(z-\be)^2}dz. \] 
 Here we use $-L_5$ to denote the hyperbola $L_5$ with orientation reversed.
 
It remains to estimate $E_k(\be)$. We re-parametrize the truncated hyperbola $-L_5$ by setting $\om = z-\be$, and under this mapping $-L_5$ becomes a portion of the hyperbola   that begins at $\om =-\be$ and approaches $\om = -\infty - i 0$ as $R \maps \infty$. (That is, a portion of a hyperbola that has asymptotes on the coordinate axes, and components in the first and third quadrants.) It is convenient to let $\ga$ denote the hyperbola obtained in the limit, as pictured in the figure. Note that $\ga$ can be parametrized for a real variable $s \in [-1,-\infty)$ by $\om = as+ib/s$, which does indeed satisfy $\Re(\om)\Im(\om) = ab$ for every $s \in [-1,-\infty)$. 
With this parameterization,
\[ E_k(\be) = \int_\ga (\om + \be)^k e^{-\om^2}d\om.\] 
To aid computation for the cases $k \in \{0,1\}$, define
\[ \tilde{E}_k(\be) = \int_\ga \om^k e^{-\om^2}d\om,\] 
so that 
\[E_0(\be)  = \tilde{E}_0(\be), \qquad E_1(\be) = \tilde{E}_1(\be) + \be \tilde{E}_0(\be).\]
Explicitly, 
\begin{align*} 
\tilde{E}_k(\be)&=  \int_{-1}^{-\infty} (as+ib/s)^k e^{-(as+ib/s)^2}(a-ib/s^2) ds \\
&=  (-1)^{k+1}\int_{1}^{\infty}s^k (a+ib/s^2)^k e^{-(as+ib/s)^2}(a-ib/s^2) ds.
\end{align*}
For $s \geq 1$, $|a \pm ib/s^2|^2 \leq a^2 + b^2 = |\be|^2$, so that
\[ |\tilde{E}_k(\be)| \leq |\be|^{k+1} \int_1^\infty s^k e^{-g(s)}ds = |\be|^{k+1} e^{-g(1)}\int_1^\infty s^k e^{-(g(s)-g(1))}ds ,\] 
in which $g(s):=\Re((as+ib/s)^2) = a^2s^2-b^2s^{-2}$ is an increasing function for $s \in [1,\infty)$.
Note that $g(1)=a^2-b^2 = \Re(\be^2)$. 
Also note that 
$g'(s) = 2a^2s + 2b^2s^{-3}   \geq 2a^2>0$ for $s \in [1,\infty)$.
Thus by the mean value theorem, $g(s)-g(1) \geq 2a^2(s-1)$ in the region of integration. 
Accordingly, 
\[ |\tilde{E}_k(\be)| \leq |\be|^{k+1}e^{-\Re(\be^2)} \int_1^\infty s^k e^{-2a^2(s-1)}ds = |\be|^{k+1} e^{-\Re(\be^2)}\int_0^\infty (u+1)^k e^{-2a^2u}du.\] 
Consequently $|\tilde{E}_0(\be)|\leq |\be|e^{-\Re(\be^2)}/(2a^2)$, while $|\tilde{E}_1(\be)| \leq |\be|^2 e^{-\Re(\be^2)}(1/(2a^2) + 1/(4a^4)).$
In total, $|E_0(\be)| \leq |\be| e^{-\Re(\be^2)}/2a^2$, while $|E_1(\be)| \ll |\be|^2 e^{-\Re(\be^2)}(1/a^2 + 1/a^4)$.
Inserting this in (\ref{J_intermediate_conclusion}) leads to the conclusion
\[J_k(\be) = \lim_{R \maps \infty} e^{\be^2} \int_{L_1} z^k e^{-(z-\be)^2}dz
 = \be^k \sqrt{\pi} e^{\be^2}  +   j_k(\be)\]
 in which $|j_0(\be)| \leq |\be|/2a^2$ and $|j_1(\be)| \ll |\be|^2 (1/a^2 + 1/a^4)$, and the lemma is proved.
Since  Lemma \ref{lemma_asymptotic_exponential_growth} is proved, this completes the proof of Corollary \ref{cor_I_plus_minus_growth} and hence  Theorem \ref{thm_hyp_weak_counterex_vertical_growth}.
\end{proof}

 \section{Further literature}\label{sec_lit}

 \subsection{References related to Tauberian theorems without explicit remainder terms}\label{sec_lit_HypA}

We mention several sources closely related to Hypothesis \ref{hyp:Weak} and Theorem \ref{thm:Weak}. Far more general results are known:  Delange's paper \cite{Del54} presents the original proofs, and Delange's lecture series \cite{Del55} describes the motivation for studying Abelian and Tauberian theorems, surveys his work, and provides many applications; these  remain comprehensive resources for Tauberian theorems for arithmetic applications. In this section, we briefly tour several of Delange's stronger results \cite[Thms. I-IV]{Del54}, also pointing to treatments in textbooks. 

We set notation for the following three sections. First, let $a(t)$ be a non-decreasing real-valued function on $[0,\infty)$ with Laplace transform 
\[ f(s) = \int_0^\infty a(t) \exp(-st)dt\] 
that is convergent for $\Re(s)>\al$ for a fixed real $\al>0$. Second, let $A(s) = \sum_{n \geq 1}a_n n^{-s}$   be a Dirichlet series with non-negative coefficients and abscissa of convergence $\al>0$. Then $\tilde{a}(t):=\sum_{n \leq \exp(t)} a_n$ is a non-decreasing real-valued function with Laplace transform 
\[ \tilde{f}(s) = \frac{A(s)}{s}\] 
for $\Re(s)>\al$.
A result for Dirichlet series will thus follow from a result for Laplace transforms, just as we deduced Theorem \ref{thm:Weak} from Theorem \ref{thm_Delange}. In each of the next three sections, we will state that if $f(s)$ satisfies a certain hypothesis (let us call it Hypothesis H for the moment) in the region $\Re(s) \geq \al$, then $a(t)$ admits a certain asymptotic. Then, supposing that $A(s)$ satisfies Hypothesis H, we will observe that $\tilde{f}(s)$ also satisfies a version of Hypothesis H; in each case this is a consequence of the fact that $\al>0$ so that the factor $1/s$ in $\tilde{f}(s)$ does not interact with the hypothesis. Thus we may apply the previous case to $\tilde{f}(s)$ in place of $f(s)$, and deduce a corresponding asymptotic for $\tilde{a}(s)$.

\subsubsection{``Pole'' of real order}
A Tauberian theorem can be obtained if $f(s)$ has a singularity that is analogous to a pole, with real order rather than integral order. Suppose that
\beq\label{Delange_Hyp1} 
f(s) = (s-a)^{-\ga}g(s) + h(s)
\eeq
for $\Re(s)>\al$, for a fixed real number $\ga$ with $\ga \not\in \{0,-1,-2,\ldots\}$, and functions $g,h$ that are holomorphic for $\Re(s) \geq \al$, with $g(\al)\neq 0$. Then 
\[
 a(t) \sim \frac{g(\al)}{\Ga(\ga)}e^{\al t}t^{\ga-1}.
 \]
  In particular, suppose that $A(s)$ satisfies (\ref{Delange_Hyp1}) for functions $g,h$.  Then   $\tilde{f}(s) = A(s)/s$ satisfies (\ref{Delange_Hyp1}) for functions $\tilde{g}(s)=g(s)/s$ and $\tilde{h}(s)=h(s)/s$, so that applying the result above to $\tilde{f}$ yields
 \beq\label{BatDia_Delange}
 \sum_{n \leq x} a_n \sim \frac{g(\al)}{\al \Ga(\ga)}x^\al (\log x)^{\ga-1}.
 \eeq
 This is due to Delange \cite[Thm. III Eqn (11)]{Del54}, \cite[Cor. I p. 61]{Del55}.
Bateman and Diamond present this in \cite[Thm. 7.7]{BatDia04} using Mellin transforms as in (\ref{Mellin_dfn1}); the proof is thorough for $\ga<2$, and sketched  for $\ga \geq 2$. Narkiewicz presents this for all real $\ga>0$ in \cite[Thm. 3.9]{Nar83}, and its Corollary (for standard Dirichlet series), closely following Delange's presentation;   this requires only a small modification of Theorem \ref{thm_Delange}.

 \subsubsection{A ``pole'' of real order, and ``lesser poles'' of complex order}\label{sec_pole_complex}
 A Tauberian theorem can be obtained if $f(s)$ has a singularity that is analogous to a pole, with real order, while further ``lesser'' singularities with complex order can occur at the same point.   Suppose in the region $\Re(s)>\al$, for a fixed real number $\ga$ with $\ga \not\in \{0,-1,-2,\ldots\}$,
\beq\label{complex_pole_expansion}
f(s)  =  (s-\al)^{-\ga}g(s) + \sum_{j=1}^N (s-\al)^{-(\lam_j + i\mu_j)} g_j(s) + h(s),
\eeq
in which the functions $g(s)$, $g_j(s)$ and $h(s)$ are holomorphic in $\Re(s) \geq \al$, $g(\al) \neq 0$, and $\lam_j, \mu_j$ are real with $\lam_j < \ga$ for $j=1,\ldots,N$, and $\lam_j +i \mu_j \not\in \{0,-1,-2,\ldots,\}$.
(The powers $(s-\al)^{-(\lam_j +i\mu_j)}$ are defined using the principal branch of the logarithm.)
Then 
\[
 a(t) \sim \frac{g(\al)}{\Ga(\ga)}e^{\al t}t^{\ga-1}.
 \]
  In particular, suppose that $A(s)$ satisfies (\ref{complex_pole_expansion}) for functions $g,g_j,h$.  Then   $\tilde{f}(s) = A(s)/s$ satisfies (\ref{complex_pole_expansion}) for functions $\tilde{g}(s)=g(s)/s$, $\tilde{g}_j(s)=g_j(s)/s$ and $\tilde{h}(s)=h(s)/s$, so that applying the result above to $\tilde{f}$ yields
  \beq\label{Delange_pole_complex_asympt}
 \sum_{n \leq x} a_n \sim \frac{g(\al)}{\al \Ga(\ga)}x^\al (\log x)^{\ga-1}.
 \eeq
 This is due to Delange \cite[Thm. III Eqn (12)]{Del54}, \cite[Cor. I p. 61]{Del55}. (Note that when this result is recorded in Delange \cite[Thm. 1]{Del62}, it is made clear that the hypothesis is the strict inequality  $\lam_j<\ga$.)

\subsubsection{Singularities of logarithmic type}
  Suppose in the region $\Re(s)>\al$, for a fixed real $\ga \geq 0$ and $k \geq 1$, 
\beq\label{Delange_Hyp3}
f(s)  =  \frac{1}{(s-\al)^\ga}\sum_{j=0}^k g_j(s) \left(\log \left(\frac{1}{s-\al}\right)\right)^j + h(s),
\eeq
in which the functions $g_j(s)$ and $h(s)$ are holomorphic in $\Re(s) \geq \al$, and $g_k(\al) \neq 0$. (Here $\log (1/(s-\al))$ is defined according to the principal branch.) 
The Tauberian statement in the case $\ga \not\in\{0,-1,-2,\ldots\}$ is that as $t \maps \infty$,
\[
 a(t) \sim  \frac{g_k(\al)}{\Ga(\ga)} e^{\al t} t^{\ga-1}(\log t)^k.
\]
If $\ga =-m$ for an integer $m \geq 0$,
\[ a(t) \sim (-1)^m m! \cdot  kg_k(\al) e^{\al  t} t^{-m-1}(\log t)^{k-1}.\]
  In particular, suppose that $A(s)$ satisfies (\ref{Delange_Hyp3}) for functions $g_j,h$.  Then   $\tilde{f}(s) = A(s)/s$ satisfies (\ref{Delange_Hyp3}) for functions $\tilde{g}_j(s)=g_j(s)/s$ and $\tilde{h}(s)=h(s)/s$, so that applying the result above to $\tilde{f}$ yields that for $\ga \not\in\{0,-1,-2,\ldots\}$, as $x \maps \infty$,
\[
 \sum_{n \leq x}a_n \sim  \frac{g_k(\al)}{\al \Ga(\ga)} x^\al (\log x)^{\ga-1}(\log \log x)^k.
\]
If $\ga =-m$ for an integer $m \geq 0$,
\[ \sum_{n \leq x}a_n \sim (-1)^m m! \cdot \frac{kg_k(\al)}{\al } x^\al (\log x)^{-m-1}(\log \log x)^{k-1}.\]
This is due to Delange \cite[Thm. IV]{Del54}, \cite[Cor. II p. 62]{Del55}. For $\ga \geq 0$ a proof is described in \cite[Thm. 3.10]{Nar83}, with an application to prove the prime number theorem \cite[Thm. 3.11]{Nar83}.
For any $\ga$ this is recorded in Tenenbaum \cite[Ch. II Notes for \S 7.4 Thm. 15]{Ten15}, without a proof.

 \subsubsection{Arithmetic example with ``pole'' of complex order}\label{sec_pole_complex_application}
Elementary arithmetic questions can lead to more general singularities than poles; here is an example of a singularity of the type described in \S \ref{sec_pole_complex}. 
Let $\Om(n)$ denote the number of prime factors of $n$ counted with multiplicity, so that if $n=p_1^{a_1} \cdots p_r^{a_r}$ then $\Omega(n)=a_1 +\cdots +a_r$. For $q \geq 2$, a Tauberian theorem of Delange can show that
\[ \#\{ n \leq x : q| \Om(n)\} \sim \tfrac{1}{q} x\qquad \text{as $x \maps \infty$.}\] 
This is a special case of \cite[Example A.2.1 p. 78]{Del55}; Delange also presents analogous results for $\om(n)$ (counting distinct prime divisors). The relation to \S \ref{sec_pole_complex} is as follows. Define $\mu_q$ to be the set of $q$-th roots of unity in $\C$, and define the Dirichlet series 
\[ D_q(s) = \sum_{\substack{n \geq 1 \\ q|\Om(n)}}n^{-s} = \frac{1}{q} \sum_{\be \in \mu_q} \prod_p \sum_{k \geq 0}(\be p^{-s})^k = \frac{1}{q} \sum_{\be \in \mu_q} \prod_p (\frac{1}{1-\be p^{-s}}) =: \frac{1}{q} \sum_{\be \in \mu_q}E_{q,\be}(s). \]  
(The second identity follows from the fact that 
$\tfrac{1}{q} \sum_{\be \in \mu_q} \be^{m} =1$ if $q|m$ and vanishes otherwise.) We claim that for each $\be \in \mu_q$, the Euler product  $E_{q,\be}(s) $ admits an expansion
\beq\label{E_expansion} 
E_{q,\be}(s)  = (s-1)^{-\be}h_{q,\be}(s) 
\eeq
in which   $h_{q,\be}(s)$ is holomorphic for $\Re(s) \geq 1$, and $(s-1)^{-\be}$ is defined using the principal branch of the logarithm (and is continuous on $\C \setminus \R_{\leq 1}$). (For $\be=1$ this certainly holds since $E_{q,1}(s) = \zeta(s)$ is the Riemann zeta function, so that moreover   $h_{q,1}(1)=1$). That is, (\ref{E_expansion}) claims that $E_{q,1}(s)$ has a simple pole at $s=1$, while $E_{q,\be}(s)$ has a ``pole'' of complex order $\be$ for each $\be \in \mu_q \setminus \{1\}$. Once this is verified, $D_q(s)$ meets the hypothesis of (\ref{complex_pole_expansion}) with $\al=1$,  $\ga=1$ and $\lam_j + i \mu_j \in \mu_q \setminus \{1\}$, and $g(s)=\tfrac{1}{q}h_{q,1}(s)$. Hence the claimed asymptotic follows from (\ref{Delange_pole_complex_asympt}), once (\ref{E_expansion}) is known. 

To verify (\ref{E_expansion}) for $\be \in \mu_q \setminus \{1\}$,
observe that 
\[ \zeta(s)^\be = \prod_p \frac{1}{(1-p^{-s})^{\be}} = \prod_p \frac{1}{1-\be p^{-s}} \cdot \sig_p(s), \qquad \sig_p(s):= \frac{1-\be p^{-s}}{(1-p^{-s})^\be}. \] 
Since $\sig_p(s) = (1-\be p^{-s})(1+\be p^{-s} + O_\be(p^{-2s}))= (1-\be^2p^{-2s} + O_\be(p^{-2s}))$, the function $H(s):=\prod_p \sig_p(s)$ is an absolutely convergent Euler product that represents a holomorphic function   for $\Re(s) >1/2$; moreover since each factor in $H(s)$ is nonzero, $H(s)$ is nonvanishing in this region. Thus $E_{q,\be}(s) = \zeta(s)^{\be} H(s)^{-1}$ inherits the expansion (\ref{E_expansion}) from $\zeta(s)$.  

\subsubsection{Relaxing holomorphicity on the line $\Re(s) =\al$}
Theorem \ref{thm:Weak} only assumes continuity on the boundary line $\Re(s)=\al$ for $s \neq \al$; this required only a small adaptation to Delange's original method, which assumed holomorphicity on $\Re(s) \geq \al$ for $s \neq \al$. Our inspiration for working out this modification was recovering an unpublished result credited to Stark (Remark \ref{remark_Stark}), but other recent papers have also worked under hypotheses that relax holomorphicity on the line $\Re(s)=\al$ (for $s \neq \al$). For example, consider the case in which $\ga=1/q$ for an integer $q\geq 2$.  In the notation above, if $A(s) = \sum_{n \geq 1} a_n n^{-s}$ with $a_n \geq 0$ has abscissa of convergence $\al>0$, and it is already known that for $\Re(s)> \al$ (and some $a>0$),
\beq\label{Kable_hyp_A_itself}
A(s) = \frac{a}{(s-\al)^{1/q}} + h(s), \qquad \text{with $h(s)$ holomorphic for $\Re(s) \geq \al$,}
\eeq
 then the Delange-type result (\ref{BatDia_Delange}) stated above (with $\ga=1/q$) shows that 
\beq\label{Kable_thm}
\sum_{n \leq x}a_n \sim \frac{a}{\al \Ga(1/q)}x^\al (\log x)^{1/q-1}.
\eeq

Kable \cite{Kab08} considers a different hypothesis instead of 
 (\ref{Kable_hyp_A_itself}). Suppose that instead of  (\ref{Kable_hyp_A_itself}), it is known that 
\beq\label{Kable_hyp_A_power}
(A(s))^q = \frac{a^q}{s-\al} + \tilde{h}(s), \qquad \text{with $\tilde{h}(s)$ holomorphic for $\Re(s) \geq \al$.}
\eeq
Kable remarks that if every zero of $(A(s))^q$ on the line $\Re(s)=\al$ (if any exist) has order divisible by $q$, then (\ref{Kable_hyp_A_power}) implies (\ref{Kable_hyp_A_itself}) and hence the desired asymptotic (\ref{Kable_thm}) holds. But it is desirable to remove this extra hypothesis on the zeroes (which can be difficult to verify). Kable's work \cite[Thm. 2]{Kab08} proves that (\ref{Kable_hyp_A_power}) itself suffices to imply (\ref{Kable_thm}).

Kable does so by showing that for $\ga=1/q$ ($q \geq 2$ integral), 
a  theorem for Laplace transforms  holds under slightly weaker conditions \cite[Thm. 1]{Kab08}. Imitating the notation of Theorem \ref{thm_Delange}, if $a(t)$ is the non-negative, non-decreasing function of interest and $f(s) = \int_0^\infty a(t) \exp(-st)dt$ is its Laplace transform, then consider $F(s) = f(s ) -\frac{a}{(s-\al)^{1/q}}$.   Kable shows that to deduce the relevant Tauberian asymptotic, it suffices to assume $F(s)$ extends continuously to $\Re(s) \geq \al$, and that the map $t \mapsto F(\al+it)$ is locally of bounded variation. 
Finally, Kable  proves that under the hypothesis (\ref{Kable_hyp_A_power}), these properties hold for the function $A(s) - \frac{a}{(s-\al)^{1/q}}$. In this work, Kable builds on the description of the general theory in Korevaar \cite[Ch. III \S 4]{Kor04}.
 Kato \cite{Kat15} extends the work of Kable to an analogous result when $\ga=\ell/q$ for an integer $\ell$.

\subsubsection{Real-variable methods, and multiplicative coefficients}\label{sec_Wirsing}
 
 Tenenbaum presents an exposition of some Abelian and Tauberian theorems in \cite[Ch. II \S 7]{Ten15}. There he distinguishes between the earliest types of Tauberian theorem, which only considered Laplace transforms as functions of a real variable (e.g. in work of Tauber, Hardy-Littlewood, Karamata-Freud),
and the type we have considered in this manuscript, which allow the study of complex-variable functions; these are called transcendental Tauberian theorems \cite[Ch. II \S 7.5]{Ten15}. 

On a related note, if a Dirichlet series $\sum_n a(n) n^{-s}$ has non-negative coefficients that are multiplicative, other methods can extract an asymptotic for the partial sums $\sum_{n \leq x} a(n)$; see for example elementary methods in \cite[\S 1.5-1.6]{IwaKow04}. Furthermore, in this setting a method of Wirsing applies a type of real-variable Tauberian theorem  \cite{Wir61,Wir67}. In broad strokes, for a non-negative multiplicative function $a(n)$, the strategy is to use multiplicativity to establish that $\sum_{n \leq x}a(n) \sim c \tfrac{x}{\log x}\sum_{n \leq x}a(n)/n$, where $c$ is the constant such that $\sum_{p \leq x}a(p)\sim c \tfrac{x}{\log x}$ as $x \maps \infty$. This transforms the problem to studying  the sum $\sum_{n \leq x}a(n)/n $, which is shown to be asymptotic to $\Ga(1+c)^{-1}e^{-\ga c}P(x)$  by a real-variable Tauberian theorem; here $\ga$ denotes the Euler-Mascheroni constant and $P(x)=\prod_{p \leq x} (1+ a(p)p^{-1} + a(p^2)p^{-2} + \cdots )$ is a truncated Euler product. Under suitable hypotheses on $a(n)$, asymptotic behavior of $P(x)$ can be computed. This is not in general an easy problem, but see for example \cite[Thm. 1.1]{IwaKow04}. 

 \subsection{References related to Tauberian theorems with   power-saving remainder terms}\label{sec_lit_HypB}
As mentioned before, one motivation for this article was to correct the statement of \cite[Theorem A.1]{CLT01}; see Remark \ref{remark_CLT}.  
Brumley et al \cite[Appendix B]{BLM21x} prove a result   similar to Theorem \ref{thm:Strong}, which holds for general Dirichlet series, but is stated only for a simple pole $(m=1)$ and has a factor $x^\ep$ rather than a power of $\log x$ in the remainder term. We now mention  three further options in the literature.

\subsubsection{A measure space formulation} 

 Let $(X,\mu)$ be a measure space and $f$ a non-negative-valued  measurable function on $X$. An analogue of a Dirichlet series is defined by
 \[ Z(s) = \int_X f(x)^{-s}d\mu(x).\]
The corresponding counting function for each real $B>0$ is 
$V(B) = \mu(\{f(x) \leq B\}),$ and then one can write $Z(s)  = \int_0^\infty B^{-s} dV(B).$
There is a version of Hypothesis \ref{hyp:Strong} in this measure space setting. A corresponding Tauberian theorem  provides an asymptotic with a main term and a power-saving remainder term, but the exponent of the power-saving remainder  term is not explicit in the existing reference  \cite[Thm. A.1]{CLT10}.

The hypothesis is that there exist real numbers $\al$ and $\del>0$ such that $Z(s)$ converges for $\Re(s)>\al$ and extends to a meromorphic function in   the closed half-plane $\Re(s) \geq \al - \del$. If $\al <0$ then it is assumed that $\lim_{B \maps \infty} V(B) = Z(0).$ 
 It is assumed that 
 $Z$ has no singularity in the half-plane $\Re(s)\geq \al - \del$ other than a pole  of order $m \geq 1$ at $s=\al$. It is assumed that    there exists a real $\kappa>0$ such that for $\Re(s)=\al-\del,$ $|Z(s+it)| \ll (1+|t|)^\kappa.$
 
  The corresponding Tauberian theorem then states that 
  there exists a  polynomial $P$  and a positive $\ep>0$ such that as $B \maps \infty$,
  \[   V(B) = 
  \begin{cases}
   B^\al P(\log B) + O(B^{\al - \ep}) & \text{if $\al \geq 0$}\\
  Z(0)  B^\al P(\log B) + O(B^{\al - \ep}) & \text{if $\al <0$.}
  \end{cases}
  \]
If $\al \neq 0$ then 
$ \deg P = m-1$ and the leading coefficient in $P$ is $\frac{1}{\al (m-1)!}\lim_{s \maps \al}(s-\al)^m Z(s).$ If 
   $\al=0$ then
$\deg P = m$ and the leading coefficient in $P$ is $\frac{1}{m!} \lim_{s \maps \al} (s-\al)^m Z(s).$
 It is reasonable to expect that the savings in the remainder term can be made explicit, by adapting the methods presented in this paper (or, for a weaker but still explicit result, adapting the ideas described in Remark \ref{remark_CLT}).
 
   \subsubsection{Comparing a Dirichlet series to a non-negative coefficient Dirichlet series}\label{sec_Alberts_Roux}
Our next example moves beyond the regime of imposing that the coefficients of the Dirichlet series are \emph{non-negative}. This is attainable if there is a dominating Dirichlet series whose coefficients are indeed non-negative. Precisely, suppose that $\{d_n\}$ is a sequence of complex numbers and $\{a_n\}$ is a sequence of non-negative real numbers with $|d_n| \leq a_n$ for every $n \geq 1$. Define the corresponding standard Dirichlet series 
\[ D(s) = \sum_{n \geq 1}d_n n^{-s}, \qquad A(s)= \sum_{n \geq 1}a_n n^{-s},\]
and denote their corresponding abscissa of absolute convergence by $\al_D$ and $\al_A$, respectively. 
(Note that naturally $\al_A \geq \al_D$.)

The hypothesis is that for some $\del>0$, both $D(s)$ and $A(s)$ have meromorphic continuation in the region $\Re(s)>\al_A-\del$, and they each have at most finitely many poles in this region. Moreover, it is assumed that for some $\kappa >0$,
\[ |D(s)| \ll_{\kappa} (1+|\Im(s)|)^{\kappa+\ep}, \qquad  |A(s)| \ll_{\kappa} (1+|\Im(s)|)^{\kappa+\ep}, \quad \text{for all $\ep>0$,}\]
for all $s$ with $\Re(s)> \al_A - \del$ and $\Im(s)$ sufficiently large (to avoid any pole of $D(s)$ or $A(s)$).

The corresponding Tauberian theorem is that 
\[ \sum_{n \leq x} d_n = \sum_{j=1}^r \Res_{s=s_j}\left[D(s) \frac{x^s}{s}\right] + O(x^{\al_A - \frac{\del}{\kappa+1} + \ep}),\]
for every $\ep>0$, in which $s_1,\ldots,s_r$ denotes the finite set of poles of $D(s)$ in the region $\Re(s) > \al_A  - \frac{\del}{\kappa+1}$.
This is obtained by  Alberts \cite[Thm. 6.1]{Alb24xa}, in an improvement of work of Roux \cite[Thms. 13.3, 13.8]{Rou11}, which in turn improves on Landau \cite{Lan15}.
  In the setting of comparative results, we also remark that \cite[Thm. 2]{LDTT22} provides a Tauberian theorem with explicit remainder term, for generalized Dirichlet series with non-negative coefficients, in which the remainder term is described in terms of a `dual' series that arises in the functional equation. This also builds on Landau \cite{Lan15}, as well as Chandrasekharan and Narasimhan \cite{ChaNar62}.

 \subsubsection{Comparing a Dirichlet series to a power of $\zeta(s)$}
 Our final remark again allows a Dirichlet series to have signed coefficients,  under the condition that there is a dominating Dirichlet series with non-negative coefficients. 
The setting is more general than \S \ref{sec_Alberts_Roux}, in that the singularity type need not be restricted to poles, and it does not assume holomorphicity in a right half-plane, but in a region motivated by the known zero-free region of the Riemann zeta function $\zeta(s)$. Here the key comparison is to a (complex) power of $\zeta(s)$; see for example \cite[Example 5.2]{Alb24xb} for  an Euler product that is compared to $\zeta(2s)^{1+i}$. We now quote the relevant hypothesis and asymptotic from \cite[Ch. II \S 5.3 Thm. 3]{Ten15}, in the context of the Selberg--Delange method.

Suppose that $\{d_n\}$ is a sequence of complex numbers and $\{a_n\}$ is a sequence of non-negative real numbers with $|d_n| \leq a_n$ for every $n \geq 1$. Define the corresponding standard Dirichlet series 
\[ D(s) = \sum_{n \geq 1}d_n n^{-s}, \qquad A(s)= \sum_{n \geq 1}a_n n^{-s}.\]
Suppose that for some $\ga_d \in \C,$ $c_0>0$, $0<\mu \leq 1$ and $M >0$, $D(s)\zeta(s)^{-\ga_d}$ admits an analytic continuation to the region 
\[   \left\{s =\sig+it: \sig \geq 1-\frac{c_0}{1+\log^+|t|}\right\},\]
and in this region satisfies the growth bound
\[ |D(s)\zeta(s)^{-\ga_d}|\leq M ( 1+|t|)^{1-\mu}.\]
Moreover suppose that $A(s) \zeta(s)^{-\ga_a}$ also admits these properties for some $\ga_a \in \C$. 

The Tauberian conclusion is that for all $x \geq 3$ and any  $N \geq 0$, 
\[ \sum_{n \leq x}d_n =x (\log x)^{\ga_d-1}\left[\sum_{j=0}^N \frac{\lam_j(\ga_d)}{(\log x)^j} + O\left(M \exp\{-c_1 \sqrt{\log x}\} + \left(\frac{c_2N+1}{\log x}\right)^{N+1}\right)\right],
\]
in which $c_1,c_2$ and the implied constants depend only on $c_0,\mu,|\ga_d|,|\ga_a|$. Within the order $N$ truncation in this expansion, the coefficients $\lam_j(\ga_d)$ are values of entire functions arising from linear combinations of various derivatives of $D(s)\zeta(s)^{-\ga_d}$  in its region of holomorphicity; we refer to \cite[Ch. II \S 5.3 Eqn. (15)]{Ten15} for the precise definition.

This last example is not strictly speaking in the same vein as the Tauberian-type citations mentioned throughout the remainder of this article; instead, it is in the spirit of comparing a Dirichlet series to a known Euler product in order to confirm inheritance of appropriate analytic properties; we refer to Alberts \cite{Alb24xb} for a   survey on this latter topic.

 \subsection*{\bf Acknowledgments} 
 The authors thank researchers in arithmetic statistics for raising questions and concerns about Tauberian theorems over the past five years. We also thank readers for sending in comments on an earlier version of this manuscript: O. Gorodetsky for pointing out  Wirsing's strategy (\S \ref{sec_Wirsing}), M. \v{C}ech for asking whether Corollary \ref{cor_thm_Delange_Stark} could be obtained, which spurred us to strengthen our proof of Theorem \ref{thm:Weak}; S. Fan for a simplification at (\ref{step_show_H_integrable}), B. Alberts for helpful remarks and references. We finally thank  the anonymous referees for numerous questions, comments, and corrections that improved the exposition significantly. In particular,  anonymous referees suggested simpler methods to prove Theorem \ref{thm:Kronecker}, improvements and simplifications related to the proof of   \cref{lem:weight} using \cref{lem:indicator} and the proof of \cref{cor:Pointwise}, a more general strategy to prove Theorem \ref{thm:Example_bounded} (encapsulated here as Theorem \ref{thm:Example_deduction}), and the application in \S \ref{sec_pole_complex_application}. We are indebted to their cleverness and generosity in describing improvements.  
 LBP thanks P. Gressman for suggesting an approach for Lemma \ref{lemma_asymptotic_exponential_growth}, and K. Soundararajan for suggesting ideas  contributing to earlier versions of Theorems \ref{thm:Example} and \ref{thm:Example_bounded}, as well as I. Petrow, J. Roos, and S. Steinerberger for helpful conversations. 
 
  LBP   has been partially supported   by NSF DMS-2200470, NSF
CAREER grant DMS-1652173, a Joan and Joseph Birman Fellowship, a Guggenheim Fellowship, and a Simons Fellowship;  the Number Theory group at Duke is supported by RTG DMS-2231514. LBP thanks the Hausdorff Center for Mathematics for  productive research visits as a Bonn Research Chair; the Mittag-Leffler Institute for a research visit in March 2024; and Rhodes House in 2025. CLT-B is partially supported by NSF CAREER DMS-2239681 and, previously, by NSF DMS-1902193 and  NSF FRG DMS-1854398. CLT-B thanks Duke University for their hospitality during several research visits and the Mittag-Leffler Institute for a research visit in March 2024.  AZ was partially supported by NSERC grant RGPIN-2022-04982.

\bibliographystyle{alpha}
\bibliography{tauberian_bibliography}  

\begin{thebibliography}{CDyDO02}

\bibitem[Alb24a]{Alb24xb}
B.~Alberts.
\newblock Explicit analytic continuation of {E}uler products
  (arxiv:2406.18190), 2024.

\bibitem[Alb24b]{Alb24xa}
B.~Alberts.
\newblock Power savings for counting (twisted) abelian extensions of number
  fields (arxiv:2402.03475), 2024.

\bibitem[Alb25]{Alb25x}
B.~Alberts.
\newblock An explicit {T}auberian theorem taking averaged inputs with an
  application to counting abelian number fields, (arxiv:2508.20814), 2025.

\bibitem[BBP10]{BBP10}
K.~Belabas, M.~Bhargava, and C.~Pomerance.
\newblock Error estimates for the {D}avenport-{H}eilbronn theorems.
\newblock {\em Duke Mathematical Journal}, 153(1):173--210, May 2010.

\bibitem[BBS14]{BBS14}
V.~Blomer, J.~Br\"{u}dern, and P.~Salberger.
\newblock On a certain senary cubic form.
\newblock {\em Proc. Lond. Math. Soc. (3)}, 108(4):911--964, 2014.

\bibitem[BD04]{BatDia04}
P.~T. Bateman and H.~G. Diamond.
\newblock {\em Analytic number theory}, volume~1 of {\em Monographs in Number
  Theory}.
\newblock World Scientific Publishing Co. Pte. Ltd., Hackensack, NJ, 2004.
\newblock An introductory course.

\bibitem[BD19]{BetDes19}
S.~Bettin and K.~Destagnol.
\newblock The power-saving {M}anin-{P}eyre conjecture for a senary cubic.
\newblock {\em Mathematika}, 65(4):789--830, 2019.

\bibitem[BDV21]{BDV21}
F.~Broucke, G.~Debruyne, and J.~Vindas.
\newblock On the absence of remainders in the {W}iener-{I}kehara and
  {I}ngham-{K}aramata theorems: a constructive approach.
\newblock {\em Proc. Amer. Math. Soc.}, 149(3):1053--1060, 2021.

\bibitem[Bha05]{Bha05}
M.~Bhargava.
\newblock The density of discriminants of quartic rings and fields.
\newblock {\em Annals of Mathematics}, 162(2):1031--1063, September 2005.

\bibitem[Bha07]{Bha07}
M.~Bhargava.
\newblock Mass formulae for extensions of local fields, and conjectures on the
  density of number field discriminants.
\newblock {\em International Mathematics Research Notices}, 2007:20 pp., 2007.

\bibitem[Bha10]{Bha10a}
M.~Bhargava.
\newblock The density of discriminants of quintic rings and fields.
\newblock {\em Annals of Mathematics}, 172(3):1559--1591, October 2010.

\bibitem[BLM21]{BLM21x}
F.~Brumley, D.~Lesesvre, and D.~Milićević.
\newblock {Conductor zeta function for the GL(2) universal family
  (arXiv:2105.02068)}, 2021.

\bibitem[BM90]{BatMan90}
V.~V. Batyrev and Yu.~I. Manin.
\newblock Sur le nombre des points rationnels de hauteur born\'{e} des
  vari\'{e}t\'{e}s alg\'{e}briques.
\newblock {\em Math. Ann.}, 286(1-3):27--43, 1990.

\bibitem[Bro09]{Bro09}
T.~D. Browning.
\newblock {\em Quantitative arithmetic of projective varieties}, volume 277 of
  {\em Progress in Mathematics}.
\newblock Birkh\"{a}user Verlag, Basel, 2009.

\bibitem[Bro10]{Bro10}
T.~D. Browning.
\newblock Recent progress on the quantitative arithmetic of del {P}ezzo
  surfaces.
\newblock In {\em Number theory}, volume~6 of {\em Ser. Number Theory Appl.},
  pages 1--19. World Sci. Publ., Hackensack, NJ, 2010.

\bibitem[BT98a]{BatTsc98b}
V.~V. Batyrev and Y.~Tschinkel.
\newblock Manin's conjecture for toric varieties.
\newblock {\em J. Algebraic Geom.}, 7(1):15--53, 1998.

\bibitem[BT98b]{BatTsc98a}
V.~V. Batyrev and Y.~Tschinkel.
\newblock {Tamagawa numbers of polarized algebraic varieties (Nombre et
  r\'{e}partition de points de hauteur born\'{e}e (Paris, 1996))}.
\newblock {\em Ast\'{e}risque}, 251:299--340, 1998.

\bibitem[CDyDO02]{CDO02}
H.~Cohen, F.~Diaz~y Diaz, and M.~Olivier.
\newblock Enumerating quartic dihedral extensions of {$\Bbb Q$}.
\newblock {\em Compositio Math.}, 133(1):65--93, 2002.

\bibitem[CLT01]{CLT01}
A.~Chambert-Loir and Y.~Tschinkel.
\newblock Fonctions z\^{e}ta des hauteurs des espaces fibr\'{e}s.
\newblock In {\em Rational points on algebraic varieties}, volume 199 of {\em
  Progr. Math.}, pages 71--115. Birkh\"{a}user, Basel, 2001.

\bibitem[CLT02]{CLT02}
A.~Chambert-Loir and Y.~Tschinkel.
\newblock On the distribution of points of bounded height on equivariant
  compactifications of vector groups.
\newblock {\em Invent. Math.}, 148(2):421--452, 2002.

\bibitem[CLT10]{CLT10}
A.~Chambert-Loir and Y.~Tschinkel.
\newblock Igusa integrals and volume asymptotics in analytic and adelic
  geometry.
\newblock {\em Confluentes Math.}, 2(3):351--429, 2010.

\bibitem[CLT12]{CLT12}
A.~Chambert-Loir and Y.~Tschinkel.
\newblock Integral points of bounded height on partial equivariant
  compactifications of vector groups.
\newblock {\em Duke Math. J.}, 161(15):2799--2836, 2012.

\bibitem[CN62]{ChaNar62}
K.~Chandrasekharan and R.~Narasimhan.
\newblock Functional equations with multiple gamma factors and the average
  order of arithmetical functions.
\newblock {\em Ann. of Math. (2)}, 76:93--136, 1962.

\bibitem[Dav00]{Dav00}
H.~Davenport.
\newblock {\em Multiplicative number theory}, volume~74 of {\em Graduate Texts
  in Mathematics}.
\newblock Springer-Verlag, New York, third edition, 2000.
\newblock Revised and with a preface by Hugh L. Montgomery.

\bibitem[Del54]{Del54}
H.~Delange.
\newblock {G\'{e}n\'{e}ralisations du theor\`{e}me de Ikehara}.
\newblock {\em Ann. Sci. \'Ecole Norm. Sup. (3)}, 71:213--242, 1954.

\bibitem[Del55]{Del55}
H.~Delange.
\newblock Th\'eor\`emes taub\'eriens et applications arithm\'etiques.
\newblock {\em M\'em. Soc. Roy. Sci. Li\`ege (4)}, 16(1-2):5--87, 1955.

\bibitem[Del63]{Del62}
H.~Delange.
\newblock Th\'eor\`emes taub\'eriens et applications arithm\'etiques.
\newblock {\em S\'{e}minaire Delange-Pisot-Poitou, Th\'{e}orie des nombres},
  4(expos\'{e} 16):1--17, 1962-1963.

\bibitem[DGH03]{DGH03}
A.~Diaconu, D.~Goldfeld, and J.~Hoffstein.
\newblock Multiple {D}irichlet series and moments of zeta and {$L$}-functions.
\newblock {\em Compositio Math.}, 139(3):297--360, 2003.

\bibitem[DH71]{DavHei71}
H.~Davenport and H.~Heilbronn.
\newblock On the density of discriminants of cubic fields {II}.
\newblock {\em Proc. Roy. Soc. Lond. A.}, 322:405--420, 1971.

\bibitem[dSG00]{SauGru00}
M.~du~Sautoy and F.~Grunewald.
\newblock Analytic properties of zeta functions and subgroup growth.
\newblock {\em Ann. of Math. (2)}, 152(3):793--833, 2000.

\bibitem[dSW08]{SauWoo08}
M.~du~Sautoy and L.~Woodward.
\newblock {\em Zeta functions of groups and rings}, volume 1925 of {\em Lecture
  Notes in Mathematics}.
\newblock Springer-Verlag, Berlin, 2008.

\bibitem[DV18]{DebVin18}
G.~Debruyne and J.~Vindas.
\newblock Note on the absence of remainders in the {W}iener-{I}kehara theorem.
\newblock {\em Proc. Amer. Math. Soc.}, 146(12):5097--5103, 2018.

\bibitem[EPW17]{EPW17}
J.~Ellenberg, L.~B. Pierce, and M.~M. Wood.
\newblock On {$\ell$}-torsion in class groups of number fields.
\newblock {\em Algebra Number Theory}, 11(8):1739--1778, 2017.

\bibitem[FMT89]{FMT89}
J.~Franke, Y.~I. Manin, and Y.~Tschinkel.
\newblock Rational points of bounded height on {F}ano varieties.
\newblock {\em Invent. Math.}, 95(2):421--435, 1989.

\bibitem[Fol99]{Fol99}
G.~B. Folland.
\newblock {\em Real analysis}.
\newblock Pure and Applied Mathematics (New York). John Wiley \& Sons, Inc.,
  New York, second edition, 1999.
\newblock Modern techniques and their applications, A Wiley-Interscience
  Publication.

\bibitem[GMO08]{GMO08}
A.~Gorodnik, F.~Maucourant, and H.~Oh.
\newblock Manin's and {P}eyre's conjectures on rational points and adelic
  mixing.
\newblock {\em Ann. Sci. \'{E}c. Norm. Sup\'{e}r. (4)}, 41(3):383--435, 2008.

\bibitem[GO11]{GorOh11}
A.~Gorodnik and H.~Oh.
\newblock Rational points on homogeneous varieties and equidistribution of
  adelic periods.
\newblock {\em Geom. Funct. Anal.}, 21(2):319--392, 2011.
\newblock With an appendix by Mikhail Borovoi.

\bibitem[Gro56]{Gro56}
E.~Grosswald.
\newblock The average order of an arithmetic function.
\newblock {\em Duke Math. Journal}, 23:41--44, 1956.

\bibitem[GSS88]{GSS88}
F.~J. Grunewald, D.~Segal, and G.~C. Smith.
\newblock Subgroups of finite index in nilpotent groups.
\newblock {\em Invent. Math.}, 93(1):185--223, 1988.

\bibitem[Hil05]{Hil05}
A.~J. Hildebrand.
\newblock {Introduction to Analytic Number Theory, Math 531 Lecture Notes, Fall
  2005 (version 2013.01.07)}, 2005.

\bibitem[HR15]{HarRie15}
G.~H. Hardy and M.~Riesz.
\newblock {\em The general theory of {D}irichlet's series}.
\newblock Cambridge University Press, 1915.

\bibitem[IK04]{IwaKow04}
H.~Iwaniec and E.~Kowalski.
\newblock {\em Analytic number theory}, volume~53 of {\em American Mathematical
  Society Colloquium Publications}.
\newblock American Mathematical Society, Providence, RI, 2004.

\bibitem[Ike31]{Ike31}
S.~Ikehara.
\newblock An extension of landau's theorem in the analytical theory of numbers.
\newblock {\em J. of Math. and Physics}, 10:1--12, 1931.

\bibitem[Ing32]{Ing32}
A.~E. Ingham.
\newblock {\em The Distribution of Prime Numbers}, volume~30 of {\em Cambridge
  Mathematical Tracts in Mathematics and Mathematical Physics}.
\newblock Cambridge University Press, 1932.

\bibitem[Ivi85]{Ivi85}
A.~Ivi\'c.
\newblock {\em The {R}iemann {Z}eta-{F}unction: The {T}heory of the {R}iemann
  {Z}eta-{F}unction with {A}pplications}.
\newblock John Wiley \& Sons, Inc., New York, 1985.

\bibitem[JU17]{JS21}
User Js21~URL:https://mathoverflow.net/users/21724/js21.
\newblock A question concerning {T}auberian theory.
\newblock MathOverflow, 2017.
\newblock URL:https://mathoverflow.net/q/286546 (version: 2017-12-05).

\bibitem[Kab08]{Kab08}
A.~C. Kable.
\newblock A variation of the {I}kehara-{D}elange {T}auberian theorem and an
  application.
\newblock {\em Comment. Math. Univ. St. Pauli}, 57(2):137--146, 2008.

\bibitem[Kat15]{Kat15}
R.~Kato.
\newblock A remark on the {W}iener-{I}kehara {T}auberian theorem.
\newblock {\em Comment. Math. Univ. St. Pauli}, 64(1):47--58, 2015.

\bibitem[Kl{\"{u}}05]{Klu05}
J.~Kl{\"{u}}ners.
\newblock A counterexample to {M}alle's conjecture on the asymptotics of
  discriminants.
\newblock {\em C. R. Math. Acad. Sci. Paris}, 340:411--414, 2005.

\bibitem[Kl{\"u}22]{Klu22}
J.~Kl{\"u}ners.
\newblock The asymptotics of nilpotent {G}alois groups.
\newblock {\em Acta Arith.}, 204(2):165--184, 2022.

\bibitem[Kor02]{Kor02}
J.~Korevaar.
\newblock A century of complex {T}auberian theory.
\newblock {\em Bull. Amer. Math. Soc. (N.S.)}, 39(4):475--531, 2002.

\bibitem[Kor04]{Kor04}
J.~Korevaar.
\newblock {\em Tauberian theory: A century of developments}, volume 329 of {\em
  Grundlehren der mathematischen Wissenschaften [Fundamental Principles of
  Mathematical Sciences]}.
\newblock Springer-Verlag, Berlin, 2004.

\bibitem[KV92]{KarVor92}
A.~A. Karatsuba and S.~M. Voronin.
\newblock {\em The {R}iemann zeta-function}, volume~5 of {\em De Gruyter
  Expositions in Mathematics}.
\newblock Walter de Gruyter \& Co., Berlin, 1992.
\newblock Translated from the Russian by Neal Koblitz.

\bibitem[Lan12]{Landau1912}
E.~Landau.
\newblock \"uber eine idealtheoretische {F}unktion.
\newblock {\em Trans. Amer. Math. Soc.}, 13(1):1--21, 1912.

\bibitem[Lan15]{Lan15}
E.~Landau.
\newblock {\em Uber die {Anzahl der Gitterpunkte in Gewissen Bereichen (pp.
  209-243)}}.
\newblock Nachrichten von der Gesellschaft der Wissenschaften zu G\"ottingen,
  Mathematisch-Physikalische Klasse, 1915.

\bibitem[Lan17]{Landau1917}
E.~Landau.
\newblock ueber die heckesche funktionalgleichung.
\newblock {\em Nachrichten von der {G}esellschaft der {W}issenschaften zu
  {G}\"{o}ttingen, {M}athematisch-{P}hysikalische {K}lasse}, 1917:102--111,
  1917.

\bibitem[Lan18]{Landau1918}
E.~Landau.
\newblock {\em Einf{\"u}hrung in die elemetare und analytische {T}heorie der
  algebraischen {Z}ahlen und der {I}deale}.
\newblock BG Teubner, 1918.

\bibitem[LDTT22]{LDTT22}
D.~Lowry-Duda, T.~Taniguchi, and F.~Thorne.
\newblock Uniform bounds for lattice point counting and partial sums of zeta
  functions.
\newblock {\em Math. Z.}, 300(3):2571--2590, 2022.

\bibitem[LMS93]{LMS93}
A.~Lubotzky, A.~Mann, and D.~Segal.
\newblock Finitely generated groups of polynomial subgroup growth.
\newblock {\em Israel J. Math.}, 82(1-3):363--371, 1993.

\bibitem[Mal02]{Mal02}
G.~Malle.
\newblock On the distribution of {G}alois groups.
\newblock {\em J. Num. Th.}, 92:315--219, 2002.

\bibitem[Mal04]{Mal04}
G.~Malle.
\newblock On the distribution of {G}alois groups {II}.
\newblock {\em Exp. Math.}, 13:129--135, 2004.

\bibitem[MV07]{MonVau07}
H.~L. Montgomery and R.~C. Vaughan.
\newblock {\em Multiplicative number theory. {I}. {C}lassical theory},
  volume~97 of {\em Cambridge Studies in Advanced Mathematics}.
\newblock Cambridge University Press, Cambridge, 2007.

\bibitem[Nar83]{Nar83}
W.~Narkiewicz.
\newblock {\em Number theory}.
\newblock World Scientific Publishing Co., Singapore; distributed by Heyden \&
  Son, Inc., Philadelphia, PA, 1983.
\newblock Translated from the Polish by S. Kanemitsu.

\bibitem[Nar00]{Nar00}
W.~Narkiewicz.
\newblock {\em The Development of Prime Number Theory: from Euclid to Hardy and
  Littlewood}.
\newblock Springer, 2000.

\bibitem[Pey95]{Pey95}
E.~Peyre.
\newblock Hauteurs et mesures de {T}amagawa sur les vari\'{e}t\'{e}s de {F}ano.
\newblock {\em Duke Math. J.}, 79(1):101--218, 1995.

\bibitem[Pie23]{Pie23}
L.~B. Pierce.
\newblock Counting problems: class groups, primes, and number fields.
\newblock In {\em I{CM}---{I}nternational {C}ongress of {M}athematicians.
  {V}ol. {III}. {S}ections 1--4}, pages 1940--1965. EMS Press, Berlin, [2023]
  \copyright 2023.

\bibitem[PTBW20]{PTBW20}
L.~B. Pierce, C.~L. Turnage-Butterbaugh, and M.~M. Wood.
\newblock An effective {C}hebotarev density theorem for families of number
  fields, with an application to {$\ell$}-torsion in class groups.
\newblock {\em Invent. Math.}, 219(2):701--778, 2020.

\bibitem[QQ13]{QueQue13}
H.~Queff{\'{e}}lec and M.~Queff{\' {e}}lec.
\newblock {\em Diophantine Approximation and Diophantine Series}, volume~2 of
  {\em Harish-Chandra Research Institute Lecture Notes}.
\newblock Hindustan Book Agency, 2013.

\bibitem[Rad60]{Rad59}
H.~Rademacher.
\newblock On the {P}hragm\'en-{L}indel\"of theorem and some applications.
\newblock {\em Math. Z.}, 72:192--204, 1959/60.

\bibitem[Rou11]{Rou11}
M.~Roux.
\newblock {\em Th\'eorie de l’information, s\'eries de {D}irichlet, et
  analyse d’algorithmes}.
\newblock Th\'eorie de l’information [cs.IT]. Universit\'e de Caen, 2011.

\bibitem[SS03a]{SSComp}
E.~M. Stein and R.~Shakarchi.
\newblock {\em Complex analysis}, volume~2 of {\em Princeton Lectures in
  Analysis}.
\newblock Princeton University Press, Princeton, NJ, 2003.

\bibitem[SS03b]{SSFour}
E.~M. Stein and R.~Shakarchi.
\newblock {\em Fourier analysis}, volume~1 of {\em Princeton Lectures in
  Analysis}.
\newblock Princeton University Press, Princeton, NJ, 2003.
\newblock An introduction.

\bibitem[SS05]{SSReal}
E.~M. Stein and R.~Shakarchi.
\newblock {\em Real analysis}, volume~3 of {\em Princeton Lectures in
  Analysis}.
\newblock Princeton University Press, Princeton, NJ, 2005.
\newblock Measure theory, integration, and Hilbert spaces.

\bibitem[STBT07]{STT07}
J.~Shalika, R.~Takloo-Bighash, and Y.~Tschinkel.
\newblock Rational points on compactifications of semi-simple groups.
\newblock {\em J. Amer. Math. Soc.}, 20(4):1135--1186, 2007.

\bibitem[Ten15]{Ten15}
G.~Tenenbaum.
\newblock {\em Introduction to analytic and probabilistic number theory},
  volume 163 of {\em Graduate Studies in Mathematics}.
\newblock American Mathematical Society, Providence, RI, third edition, 2015.

\bibitem[Tit32]{Tit32}
E.C. Titchmarsh.
\newblock {\em The Theory of Functions, second edition}.
\newblock Oxford University Press, Oxford, New York, 1932.

\bibitem[TT13]{TanTho13}
T.~Taniguchi and F.~Thorne.
\newblock The secondary term in the counting function for cubic fields.
\newblock {\em Duke Mathematical Journal}, 162:2451--2508, 2013.

\bibitem[TZ19]{TZ19}
J.~Thorner and A.~Zaman.
\newblock A unified and improved {C}hebotarev density theorem.
\newblock {\em Algebra Number Theory}, 13(5):1039--1068, 2019.

\bibitem[Vol11]{Vol11}
C.~Voll.
\newblock A newcomer's guide to zeta functions of groups and rings.
\newblock In {\em Lectures on profinite topics in group theory}, volume~77 of
  {\em London Math. Soc. Stud. Texts}, pages 99--144. Cambridge Univ. Press,
  Cambridge, 2011.

\bibitem[Wei83]{Wei83}
A.~Weiss.
\newblock The least prime ideal.
\newblock {\em J. Reine Angew. Math.}, 338:56--94, 1983.

\bibitem[Wie32]{Wie32}
N.~Wiener.
\newblock Tauberian theorems.
\newblock {\em Ann. of Math.}, 33:1--100, 1932.

\bibitem[Wir61]{Wir61}
E.~Wirsing.
\newblock Das asymptotische {V}erhalten von {S}ummen \"uber multiplikative
  {F}unktionen.
\newblock {\em Math. Ann.}, 143:75--102, 1961.

\bibitem[Wir67]{Wir67}
E.~Wirsing.
\newblock Das asymptotische {V}erhalten von {S}ummen \"uber multiplikative
  {F}unktionen. {II}.
\newblock {\em Acta Math. Acad. Sci. Hungar.}, 18:411--467, 1967.

\bibitem[Woo16]{Woo16}
M.~M. Wood.
\newblock Asymptotics for number fields and class groups.
\newblock In {\em Directions in number theory}, volume~3 of {\em Assoc. Women
  Math. Ser.}, pages 291--339. Springer, [Cham], 2016.

\bibitem[Zam17]{Zam16}
A.~Zaman.
\newblock Bounding the least prime ideal in the {C}hebotarev density theorem.
\newblock {\em Funct. Approx. Comment. Math.}, 57(1):115--142, 2017.

\end{thebibliography}

\end{document}